\documentclass[a4paper]{amsart}
\usepackage{graphicx}
\usepackage{amsmath}
\usepackage{amssymb}
\usepackage{amsfonts}
\usepackage{amsthm}
\usepackage{eucal}
\usepackage{tikz}
\usetikzlibrary{cd}

\theoremstyle{plain}
\newtheorem{theorem}{Theorem}[section]

\newtheorem*{theorem*}{Theorem}
\newtheorem{lemma}[theorem]{Lemma}
\newtheorem{proposition}[theorem]{Proposition}

\newtheorem{corollary}[theorem]{Corollary}

\newtheorem*{proposition*}{Proposition}
\newtheorem*{corollary*}{Corollary}

\theoremstyle{definition}
\newtheorem{definition}[theorem]{Definition}

\newtheorem{remark}[theorem]{Remark}

\makeatletter
\@addtoreset{theorem}{section}
\@addtoreset{subsection}{section}
\makeatother

\begin{document}

\title{Lie algebras and $v_n$-periodic spaces}

\author[Gijs Heuts]{Gijs Heuts}
\address{Utrecht University, Department of Mathematical Sciences, Budapestlaan 6, 3584CD Utrecht, the Netherlands}
\email{g.s.k.s.heuts@uu.nl}

\date{}

\begin{abstract}
We consider a homotopy theory obtained from that of pointed spaces by inverting the maps inducing isomorphisms in $v_n$-periodic homotopy groups. The case $n=0$ corresponds to rational homotopy theory. In analogy with Quillen's results in the rational case, we prove that this $v_n$-periodic homotopy theory is equivalent to the homotopy theory of Lie algebras in $T(n)$-local spectra. We also compare it to the homotopy theory of commutative coalgebras in $T(n)$-local spectra, where it turns out there is only an equivalence up to a certain convergence issue of the Goodwillie tower of the identity.
\end{abstract}

\maketitle

\tableofcontents

\section{Introduction}

Quillen \cite{rationalhomotopy} proves a remarkable result on the global structure of the homotopy theory of rational spaces, showing that it is in a sense completely algebraic. More precisely, he shows that the homotopy theory of simply-connected rational spaces can be modelled by the homotopy theory of differential graded rational Lie algebras or that of cocommutative differential graded rational coalgebras. This paves the way for the development of very explicit calculational methods for dealing with such spaces, e.g. Sullivan's theory of minimal models \cite{sullivan}, which are not available when dealing with general spaces.

Rationalization is the first (or rather, zeroth) in a hierarchy of localizations of homotopy theory. The rational homotopy groups of a pointed space $X$ arise when considering homotopy classes of maps from spheres $S^k \rightarrow X$ and then inverting the action of the multiplication maps $p\colon S^k \rightarrow S^k$ for all primes $p$. Adams \cite{adams} showed that if one takes the cofiber $S^k/p$ then (for $k$ large enough) there is a self-map $v_1\colon \Sigma^d S^k/p \rightarrow S^k/p$ which induces an isomorphism in $K$-theory. One can now consider homotopy classes of maps $S^k/p \rightarrow X$ and invert the action of the self-map $v_1$ to obtain the (mod $p$) \emph{$v_1$-periodic homotopy groups} of $X$. A \emph{$v_1$-periodic equivalence} of spaces is a map inducing isomorphisms in these $v_1$-periodic homotopy groups. Since the periodicity results of Hopkins and Smith \cite{hopkinssmith} it is known that this pattern continues indefinitely: for every $n \geq 0$ there exist suitable finite type $n$ spaces with $v_n$ self-maps and one can consider the associated notions of $v_n$-periodic homotopy groups and $v_n$-periodic equivalences. Their results, together with the nilpotence theorem of \cite{devinatzhopkinssmith}, have proved to be very powerful organizing principles in stable homotopy theory.

In this paper we study a homotopy theory $\mathcal{S}_{v_n}$ which is essentially obtained from that of pointed spaces by inverting the $v_n$-periodic equivalences. We write $T(n)$ for the spectrum obtained as the telescope of a $v_n$ self-map on a finite type $n$ spectrum. The associated Bousfield localization of stable homotopy theory is independent of choices. The main result of this paper can be stated as follows:

\emph{The homotopy theory $\mathcal{S}_{v_n}$ is equivalent to the homotopy theory of Lie algebras in the category of $T(n)$-local spectra.}

We will make this statement precise below (Theorem \ref{thm:MnfLiealg}) and outline related results and consequences; in particular we also include a version applicable to $K(n)$-local homotopy theory. Our results generalize Quillen's rational homotopy theory to the cases $n > 0$. We also compare $\mathcal{S}_{v_n}$ to the homotopy theory of cocommutative coalgebras in $T(n)$-local spectra. It turns out these theories are not quite the same, but only equivalent `up to Goodwillie convergence'. We make this precise in Theorem \ref{thm:GoodwillietowerMnf}. This second comparison is closely related to recent work of Behrens and Rezk \cite{behrensrezk}.

\subsection*{Acknowledgments} The results of Section \ref{sec:MnfGoodwillie} grew out of an attempt to understand and contextualize the results of Behrens and Rezk \cite{behrensrezk, behrensrezk2} on the relation between the Bousfield--Kuhn functor and topological Andr\'{e}-Quillen homology. I have benefitted much from reading their work, as well as from several inspiring talks by and conversations with Mark Behrens. Theorem \ref{thm:MnfLiealg} (the comparison with Lie algebras) offers a different (and sharper) perspective on $v_n$-periodic homotopy theory, which builds on joint work with Rosona Eldred, Akhil Mathew, and Lennart Meier \cite{ehmm} carried out at the Hausdorff Research Institute for Mathematics. I wish to thank my collaborators for an inspiring semester and the Institute for its hospitality and excellent working conditions. Moreover, I thank Greg Arone and Lukas Brantner for useful conversations relating to this paper. A large intellectual debt is owed to Bousfield and Kuhn, whose work is indispensable. I thank Haynes Miller for comments on an earlier version of this paper. Finally, I thank Jacob Lurie for many useful suggestions, among which a slick proof of the crucial Proposition \ref{prop:siftedcolimitscoanalytic}. While this paper was being written, Mike Hopkins and Jacob Lurie ran a seminar at Harvard on (amongst other things) the results presented here; the reader might find their excellent notes \cite{thursday} to be a useful resource.

\section{Main results}
\label{sec:mainresults}

Throughout this paper we will use \emph{$\infty$-categories}, or \emph{quasicategories}, as our preferred formalism for higher category theory. There are several instances in this paper where this turns out to be convenient, for example in Theorem \ref{thm:Phimonadic}, in applying a theorem of Lurie on bar-cobar duality in Section \ref{subsec:PhiThetamonad}, and when applying the formalism of Goodwillie towers of $\infty$-categories in Section \ref{sec:MnfGoodwillie}. We will assume basic familiarity with the theory of $\infty$-categories. The works of Joyal \cite{joyalpaper, joyal} and Lurie \cite{htt} are the standard references. All of the spaces and spectra we consider will (implicitly) be localized at a fixed prime $p$. In other words, we consider their Bousfield localization at the homology theory determined by $\mathbb{S}_{(p)}$, the $p$-local sphere spectrum. We will write $\mathcal{S}_*$ and $\mathrm{Sp}$ for the $\infty$-categories of $p$-local pointed spaces and $p$-local spectra, respectively.

For an integer $n \geq 0$, a finite pointed space $V$ is \emph{of type $n$} if $K(m)_* V = 0$ for $m < n$ and $K(n)_* V \neq 0$. Here $K(m)_* V$ denotes the reduced $m$th Morava $K$-theory of $V$. The periodicity results of Hopkins and Smith \cite{hopkinssmith} imply that any pointed type $n$ space $V$ (after sufficiently many suspensions) admits a \emph{$v_n$ self-map}, i.e., a map $v\colon \Sigma^d V \rightarrow V$ so that
\begin{equation*}
K(m)_* v \quad \text{is} \quad \begin{cases}
\text{an isomorphism} & \text{if } m = n, \\
\text{nilpotent} & \text{if } m \neq n.
\end{cases}
\end{equation*}
For $X$ a pointed space one can define its \emph{$v$-periodic homotopy groups} with coefficients in $V$ by taking the homotopy groups of the mapping space $\mathrm{Map}_*(V, X)$ and inverting the action of $v$ by precomposition. It is convenient to formulate this definition as follows. One can define a spectrum $\Phi_v X$ by setting
\begin{equation*}
(\Phi_v X)_0 = \mathrm{Map}_*(V, X), \, (\Phi_v X)_d = \mathrm{Map}_*(V, X), \, \ldots \, , \, (\Phi_v X)_{kd} = \mathrm{Map}_*(V, X), \, \ldots
\end{equation*}
and using the maps
\begin{equation*}
(\Phi_v X)_{kd} = \mathrm{Map}_*(V, X) \xrightarrow{v^*} \mathrm{Map}_*(\Sigma^d V, X) = \Omega^d (\Phi_v X)_{(k+1)d} 
\end{equation*}
as structure maps. This defines the \emph{telescopic functor}
\begin{equation*}
\Phi_v\colon \mathcal{S}_* \rightarrow \mathrm{Sp}
\end{equation*}
associated to $v$ (see \cite{kuhntelescopic}). The homotopy groups $\pi_* \Phi_v(X)$ are then precisely the $v$-periodic homotopy groups of $X$ described above. In fact, the functor $\Phi_v$ takes values in the $\infty$-category $\mathrm{Sp}_{T(n)}$ of \emph{$T(n)$-local} spectra (see Theorem 4.2 of \cite{kuhntelescopic}). Here $T(n)$ denotes the telescope of a $v_n$ self-map on a finite type $n$ spectrum. Although $T(n)$ itself depends on the choice of spectrum, the corresponding Bousfield localization does not.

The \emph{Bousfield--Kuhn functor} conveniently packages the various telescopic functors $\Phi_v$ into one. It is a functor
\begin{equation*}
\Phi\colon \mathcal{S}_* \rightarrow \mathrm{Sp}_{T(n)}
\end{equation*}
enjoying the following properties (see Theorem 1.1 and Lemma 8.6 of \cite{kuhntelescopic}):
\begin{itemize}
\item[(i)] For $V$ a finite type $n$ space with $v_n$ self-map $v$, there is a natural equivalence
\begin{equation*}
\mathbf{D}V \otimes \Phi(X) \simeq \Phi_v(X).
\end{equation*}
Here $\mathbf{D}V$ denotes the Spanier--Whitehead dual of $V$. Of course one could replace $\mathbf{D}V \otimes \Phi(X)$ with the function spectrum $F(V,\Phi(X))$.
\item[(ii)] There is a natural equivalence of functors
\begin{equation*}
\Phi \Omega^\infty \simeq L_{T(n)}.
\end{equation*}
\item[(iii)] The functor $\Phi$ preserves finite limits.
\end{itemize}
In fact, Kuhn \cite{kuhntelescopic} shows that $\Phi$ is essentially determined by property (i).

\begin{definition}
A map of pointed spaces $f\colon X \rightarrow Y$ is a \emph{$v_n$-periodic equivalence} if $\Phi(f)$ is an equivalence of spectra.
\end{definition}

A map of $T(n)$-local spectra is an equivalence if and only if it is an equivalence after smashing with a finite type $n$ spectrum. Indeed, the latter condition implies it is an equivalence after smashing with \emph{any} finite type $n$ spectrum (by the thick subcategory theorem \cite{hopkinssmith}) and one then uses the fact that up to $T(n)$-equivalence the sphere spectrum can be written as a filtered colimit of finite type $n$ spectra (this trick goes back to Kuhn \cite{kuhninfiniteloop}). It then follows from property (i) above that a map $f$ of pointed spaces is a $v_n$-periodic equivalence if and only if $\Phi_v(X)$ is an equivalence, i.e., if and only if it induces an isomorphism on $v$-periodic homotopy groups. Our first goal will be to describe a homotopy theory obtained from $\mathcal{S}_*$ by inverting the $v_n$-periodic equivalences. The following was essentially proved by Bousfield in \cite{bousfieldtelescopic}. We formulate it here in a form which suits our purposes (and a proof is included in Section \ref{subsec:Mnf}):

\begin{theorem}
\label{thm:Mnf}
For $n \geq 1$ there exists an $\infty$-category $\mathcal{S}_{v_n}$ and a functor $M\colon \mathcal{S}_* \rightarrow \mathcal{S}_{v_n}$ such that for any $\infty$-category $\mathcal{C}$, precomposition by $M$ gives an equivalence 
\begin{equation*}
\mathrm{Fun}(\mathcal{S}_{v_n}, \mathcal{C}) \xrightarrow{M^*} \mathrm{Fun}_{v_n}(\mathcal{S}_*, \mathcal{C}).
\end{equation*} 
Here $\mathrm{Fun}_{v_n}$ denotes the full subcategory of functors which send $v_n$-periodic equivalences in $\mathcal{S}_*$ to equivalences in $\mathcal{C}$.
\end{theorem}

Of course one can always formally invert a class of morphisms in a category (or an $\infty$-category), but generally only at the cost of passing to a larger universe. The content of the theorem is therefore that $\mathcal{S}_{v_n}$ is still locally small. Note that the universal property described in the theorem implies that $\mathcal{S}_{v_n}$ and the functor $M$ are unique up to equivalence. In fact, we will construct $\mathcal{S}_{v_n}$ as a full subcategory of $\mathcal{S}_*$. We write $i\colon \mathcal{S}_{v_n} \rightarrow \mathcal{S}_*$ for the inclusion. Viewed in this way, the functor $M$ should be thought of as a projection to this subcategory. The following theorem summarizes the properties we need. Its proof will be given in Section \ref{subsec:Mnf} and again leans very heavily on the work of Bousfield \cite{bousfieldlocalization,bousfieldtelescopic}.

\begin{theorem}
\label{thm:Mnf2}
There is a natural equivalence $M \circ i \simeq \mathrm{id}_{\mathcal{S}_{v_n}}$. Furthermore, $\mathcal{S}_{v_n}$ enjoys the following properties:
\begin{itemize}
\item[(i)] A map $\varphi$ of pointed spaces is a $v_n$-periodic equivalence if and only if $M(\varphi)$ is an equivalence.
\item[(ii)] The Bousfield--Kuhn functor factors through $M$; we still denote the resulting functor by
\begin{equation*}
\Phi\colon \mathcal{S}_{v_n} \rightarrow \mathrm{Sp}_{T(n)}.
\end{equation*}
This functor admits a left adjoint $\Theta$. In particular, this $\Phi$ preserves all limits (rather than just finite limits).
\item[(iii)] The functor $M\colon \mathcal{S}_* \rightarrow \mathcal{S}_{v_n}$ preserves finite limits and filtered colimits.
\item[(iv)] The $\infty$-category $\mathcal{S}_{v_n}$ is compactly generated.
\end{itemize}
\end{theorem}

\begin{remark}
When comparing with Bousfield's paper \cite{bousfieldtelescopic} the reader will note a change of notation. Bousfield writes $\mathcal{UN}_n^f$ for a category which is essentially the homotopy category of what we call $\mathcal{S}_{v_n}$ here. 
\end{remark}

\begin{remark}
As already expressed by Theorem \ref{thm:Mnf}, the homotopy theory $\mathcal{S}_{v_n}$ and the functor $M$ are unique up to equivalence. However, the embedding $i\colon \mathcal{S}_{v_n} \rightarrow \mathcal{S}_*$ \emph{does} depend on a choice, namely that of the `connectivity' $d_{n+1}$ featuring in Section \ref{subsec:Lnf}.
\end{remark}

The $\infty$-category $\mathcal{S}_{v_n}$ is related to the $\infty$-category of $T(n)$-local spectra by two different adjunctions (left adjoints on top):
\[
\begin{tikzcd}
\mathrm{Sp}_{T(n)} \arrow[r,"\Theta", shift left] & \mathcal{S}_{v_n} \arrow[r,"\Sigma^\infty_{T(n)}", shift left] \arrow[l, "\Phi", shift left] & \mathrm{Sp}_{T(n)}. \arrow[l, "\Omega^\infty_{T(n)}", shift left] 
\end{tikzcd}
\]
The pair on the left exists by part (ii) of Theorem \ref{thm:Mnf2}. The notation $\Sigma^\infty_{T(n)}$ is short-hand for the composition 
\begin{equation*}
\mathcal{S}_{v_n} \xrightarrow{\quad i \quad} \mathcal{S}_* \xrightarrow{L_{T(n)}\Sigma^\infty} \mathrm{Sp}_{T(n)}.
\end{equation*}
Its right adjoint $\Omega^\infty_{T(n)}$ is the composition $M \circ \Omega^\infty$. We will show (see Proposition \ref{prop:Mnfstabilization} and Remark \ref{rmk:MnfLTn}) that this functor exhibits $\mathrm{Sp}_{T(n)}$ as the stabilization of $\mathcal{S}_{v_n}$, i.e., it is the terminal functor from a stable $\infty$-category to $\mathcal{S}_{v_n}$ that preserves limits. 

The two adjunctions above offer complementary perspectives on the $\infty$-category $\mathcal{S}_{v_n}$. In joint work with Eldred, Mathew, and Meier \cite{ehmm} we prove that the adjoint pair $(\Theta, \Phi)$ is \emph{monadic}, meaning that $\Phi$ gives an equivalence between $\mathcal{S}_{v_n}$ and the $\infty$-category of algebras for the monad $\Phi\Theta$ on $\mathrm{Sp}_{T(n)}$. Here we go further and explicitly identify this monad as the \emph{free Lie algebra monad}. As a functor, $\Phi\Theta$ admits the following description (see Theorem \ref{thm:PhiThetafunctor}):

\begin{equation*}
\Phi\Theta(X) \simeq L_{T(n) }\bigoplus_{k \geq 1} (\partial_k \mathrm{id} \otimes X^{\otimes k})_{h\Sigma_k}.
\end{equation*}

Here $\partial_k \mathrm{id}$ is the \emph{$k$th Goodwillie derivative of the identity}. It is a finite spectrum with $\Sigma_k$-action which can be described explicitly as the Spanier--Whitehead dual of a certain partition complex \cite{aronemahowald}. Informally speaking, the monad structure on $\Phi\Theta$ corresponds to the fact that the spectra $\partial_k \mathrm{id}$ assemble into an operad, as demonstrated by Ching \cite{ching}. This operad can be thought of as a version of the Lie operad in the stable homotopy category. We will adapt Ching's work to the setting in which we work here and give a precise definition of Lie algebras in Section \ref{sec:Liealgebras}. Subsequently we will prove the following:

\begin{theorem}
\label{thm:MnfLiealg}
For $n \geq 1$ there is an equivalence between the $\infty$-category $\mathcal{S}_{v_n}$ and the $\infty$-category $\mathrm{Lie}(\mathrm{Sp}_{T(n)})$ of Lie algebras in $T(n)$-local spectra. This equivalence has the property that the resulting composition
\begin{equation*}
\mathcal{S}_{v_n} \simeq \mathrm{Lie}(\mathrm{Sp}_{T(n)}) \xrightarrow{\mathrm{forget}} \mathrm{Sp}_{T(n)}
\end{equation*}
is equivalent to the Bousfield--Kuhn functor $\Phi$.
\end{theorem}

\begin{remark}
\label{rmk:whiteheadbracket}
One of the most basic manifestations of the relation between Lie algebras and the homotopy theory of spaces is the Whitehead bracket on the homotopy groups of a pointed space. In Proposition \ref{prop:whiteheadbracket} we show that for a pointed space $X$, the Lie algebra structures on $\Phi(X)$ produced by Theorem \ref{thm:MnfLiealg} in particular agrees with the Whitehead bracket on the $v_n$-periodic homotopy groups of $X$.
\end{remark}

There is a variant of Theorem \ref{thm:MnfLiealg} for $K(n)$-local homotopy theory (which might or might not be the same as $T(n)$-local homotopy theory, depending on the telescope conjecture). Let us say a map $f$ of pointed spaces is a \emph{$\Phi_{K(n)}$-equivalence} if the map of spectra $\Phi(f)$ is a $K(n)_*$-equivalence. We derive the following in Section \ref{subsec:Knlocal}:

\begin{corollary}
\label{cor:Knlocal}
The localization of $\mathcal{S}_{v_n}$ at the $\Phi_{K(n)}$-equivalences exists. More precisely, there exists a full subcategory $\mathcal{M}_{K(n)} \rightarrow \mathcal{S}_{v_n}$ for which the inclusion admits a left adjoint, satisfying the following two properties: 
\begin{itemize}
\item[(i)] The unit is a $\Phi_{K(n)}$-equivalence.
\item[(ii)] A map in $\mathcal{M}_{K(n)}$ is an equivalence if and only if it is a $\Phi_{K(n)}$-equivalence.
\end{itemize}
Moreover, the $\infty$-category $\mathcal{M}_{K(n)}$ is equivalent to the $\infty$-category $\mathrm{Lie}(\mathrm{Sp}_{K(n)})$ of Lie algebras in $K(n)$-local spectra.
\end{corollary}

This corollary is not proved by analogy with the proof of Theorem \ref{thm:MnfLiealg}; rather, it is a formal consequence. We do not know of a direct proof avoiding the use of that theorem. This highlights the fundamental role of $T(n)$-local (as opposed to $K(n)$-local) homotopy theory in this paper.

\begin{remark}
Our proof of Theorem \ref{thm:MnfLiealg} does not extend in an evident way to include the case $n=0$ of rational homotopy theory and therefore does not give an independent proof Quillen's results. The reason is that the definition of the telescopic functors $\Phi_v$ (and the Bousfield--Kuhn functor $\Phi$) does not admit an evident analog in the rational setting. However, many of our results on operads and Lie algebras in $T(n)$-local spectra in Section \ref{sec:Liealgebras} do extend to the case of rational spectra.
\end{remark}

Theorem \ref{thm:MnfLiealg} tells us that the adjoint pair $(\Theta, \Phi)$ is as good as one could hope for. The question remains to what extent this is true for the other adjunction
\[
\begin{tikzcd}
\mathcal{S}_{v_n} \arrow[r,"\Sigma^\infty_{T(n)}", shift left] & \mathrm{Sp}_{T(n)}. \arrow[l, "\Omega^\infty_{T(n)}", shift left] 
\end{tikzcd}
\]
This adjunction gives a comparison between $\mathcal{S}_{v_n}$ and the $\infty$-category of coalgebras for the comonad $\Sigma^\infty_{T(n)}\Omega^\infty_{T(n)}$. As a functor, a theorem of Kuhn implies that this comonad can be described by the formula
\begin{equation*}
\Sigma^\infty_{T(n)}\Omega^\infty_{T(n)} X \simeq L_{T(n)} \bigoplus_{k \geq 1}X^{\otimes k}_{h\Sigma_k},
\end{equation*}
see Theorem \ref{thm:SigmaOmegacoanalytic}. This formula is suggestive of the fact that such coalgebras are closely related to commutative coalgebras in $\mathrm{Sp}_{T(n)}$. We will make this relationship explicit in Section \ref{sec:MnfGoodwillie}, using the formalism of \emph{Goodwillie towers of $\infty$-categories} developed in \cite{heutsgoodwillie}. The result is the following, which we prove in Section \ref{sec:GoodwillietowerMnf}:

\begin{theorem}
\label{thm:GoodwillietowerMnf}
The functor $\Sigma^\infty_{T(n)}$ induces an equivalence between the $k$th Goodwillie approximation $\mathcal{P}_k\mathcal{S}_{v_n}$ of the $\infty$-category $\mathcal{S}_{v_n}$ and the $\infty$-category of $k$-truncated commutative ind-coalgebras in $T(n)$-local spectra, denoted $\mathrm{coAlg}^{\mathrm{ind}}(\tau_k\mathrm{Sp}_{T(n)}^{\otimes})$.
\end{theorem}

We will expain the terms in the statement of this theorem more precisely in Section \ref{sec:GoodwillietowerMnf}, but for now we mention that a $k$-truncated commutative coalgebra is essentially a spectrum $E$ equipped with comultiplication maps
\begin{equation*}
\delta_j\colon E \rightarrow (E^{\otimes j})^{h\Sigma_j}
\end{equation*}
for $2 \leq j \leq k$ together with a coherent system of homotopies expressing the necessary compatibilities between the various $\delta_j$. Theorem \ref{thm:GoodwillietowerMnf} cannot be strengthened to say that $\mathcal{S}_{v_n}$ is equivalent to the $\infty$-category $\mathrm{coAlg}^{\mathrm{ind}}(\mathrm{Sp}_{T(n)}^\otimes)$ of commutative ind-coalgebras in $T(n)$-local spectra (or that of commutative coalgebras, without the ind, for that matter). Rather, one can think of the difference between the two as the issue of \emph{convergence} of the Goodwillie tower of the identity of $\mathcal{S}_{v_n}$.

Theorem \ref{thm:GoodwillietowerMnf} gives another perspective on the Bousfield--Kuhn functor that we discuss in Section \ref{sec:GoodwillieBK}. There is a functor
\begin{equation*}
\mathrm{triv}\colon \mathrm{Sp}_{T(n)} \rightarrow \mathrm{coAlg}(\mathrm{Sp}_{T(n)}^{\otimes})
\end{equation*}
which assigns to a $T(n)$-local spectrum $X$ the \emph{trivial coalgebra structure} on $X$. This functor admits a right adjoint
\begin{equation*}
\mathrm{prim}\colon \mathrm{coAlg}(\mathrm{Sp}_{T(n)}^\otimes) \rightarrow \mathrm{Sp}_{T(n)}
\end{equation*}
which we will refer to as the \emph{primitives functor}. The construction of this functor is formally dual to the construction of topological Andr\'{e}-Quillen homology (TAQ) of commutative ring spectra (see for example \cite{basterramandell}). Similarly, there are primitives functors
\begin{equation*}
\mathrm{prim}_k\colon \mathrm{coAlg}(\tau_k\mathrm{Sp}_{T(n)}^\otimes) \rightarrow \mathrm{Sp}_{T(n)}
\end{equation*}
for every $k \geq 1$. Write
\begin{equation*}
C_{T(n)}\colon \mathcal{S}_{v_n} \rightarrow \mathrm{coAlg}^{\mathrm{ind}}(\mathrm{Sp}_{T(n)}^\otimes)
\end{equation*}
for the functor that assigns to a space $X \in \mathcal{S}_{v_n}$ its $T(n)$-local suspension spectrum $\Sigma^\infty_{T(n)} X$ together with its natural coalgebra structure with respect to the smash product. Similarly, write $\tau_k C_{T(n)}(X)$ for the $k$-truncations of this coalgebra, simply obtained by forgetting the comultiplication maps $\delta_j$ for $j > k$. We will derive the following as a formal consequence of Theorem \ref{thm:GoodwillietowerMnf}:

\begin{theorem}
\label{thm:BR}
The Goodwillie tower of the functor $\Phi$ can be described in terms of primitives by equivalences
\begin{equation*}
P_k\Phi(X) \rightarrow \mathrm{prim}_k(\tau_k C_{T(n)} X)
\end{equation*}
natural in $X$ and $k$.
\end{theorem}

This reproduces a recent result of Behrens and Rezk: in \cite{behrensrezk} they prove the analogous statement in the $K(n)$-local setting, stating it for TAQ rather than for the primitives functor we consider here. Let us say a pointed space $X$ is \emph{$\Phi$-good} if the Goodwillie tower of $\Phi$ converges on $X$, i.e., if the natural map
\begin{equation*}
\Phi(X) \rightarrow \varprojlim_k P_k\Phi(X)
\end{equation*}
is an equivalence. The natural maps of Theorem \ref{thm:BR} arise from a natural transformation $\Phi \rightarrow \mathrm{prim} \circ C_{T(n)}$ which is a variant of the \emph{comparison map} of Behrens and Rezk. In Section \ref{sec:MnfGoodwillie} we show how to derive the following:

\begin{corollary}
\label{cor:BR}
A space $X \in \mathcal{S}_{v_n}$ is $\Phi$-good if and only if the comparison map
\begin{equation*}
\Phi(X) \rightarrow \mathrm{prim}(C_{T(n)} X)
\end{equation*}
is an equivalence.
\end{corollary}

It is an interesting question which spaces are $\Phi$-good. Arone and Mahowald (Theorem 4.1 of \cite{aronemahowald}) prove that spheres are $\Phi$-good. Behrens and Rezk (Section 8 of \cite{behrensrezk2}) show that the same is true for the special unitary groups and the symplectic groups (at least in the $K(n)$-local setting). In \cite{brantnerheuts} it is proved that wedges of spheres are \emph{not} $\Phi$-good and that the same is true for Moore spaces. We will establish a novel class of $\Phi$-good spaces in Corollary \ref{cor:aronemahowald}, namely the spaces of the form $\Theta(L_{T(n)}\mathbb{S}^\ell)$, with $\mathbb{S}^{\ell}$ denoting the $\ell$-fold suspension of the sphere spectrum and $\ell$ any integer.

The fact that not all spaces are $\Phi$-good explains why Theorem \ref{thm:MnfLiealg} gives a much sharper description of $v_n$-periodic unstable homotopy theory than the coalgebra model of Theorem \ref{thm:GoodwillietowerMnf} and the resulting comparison between the Bousfield--Kuhn functor and primitives (or TAQ). In particular, it should be noted that the Lie algebras $\Phi X$ produced by Theorem \ref{thm:MnfLiealg} are generally \emph{not} the same as the Lie algebras $\mathrm{prim}(C_{T(n)} X)$ or $\mathrm{TAQ}(\mathbb{S}_{T(n)}^{X_+})$, which are considered by Behrens and Rezk in \cite{behrensrezk2}.

%Here we relate $\Phi$-goodness to a certain notion of completeness in $\mathcal{S}_{v_n}$. The adjoint pair $(\Sigma^\infty_{T(n)}, \Omega^\infty_{T(n)})$ gives rise to a monad $\Omega^\infty_{T(n)}\Sigma^\infty_{T(n)}$ on $\mathcal{S}_{v_n}$ which we will denote
%\begin{equation*}
%Q_{\mathcal{S}_{v_n}}
%\end{equation*}
%in analogy with the usual functor $Q$ on pointed spaces. Then for any $X \in \mathcal{S}_{v_n}$ we may form its $Q_{\mathcal{S}_{v_n}}$-resolution, which is a coaugmented cosimplicial object as follows:
%\[
%\begin{tikzcd}
%X \arrow[r] & Q_{\mathcal{S}_{v_n}}(X) \arrow[r, shift left] \arrow[r, shift right] & Q_{\mathcal{S}_{v_n}}^2(X) \arrow[r, shift left=2]\arrow[r] \arrow[r, shift right=2]  \arrow[l]  & Q_{\mathcal{S}_{v_n}}^3(X) \cdots .  \arrow[l, shift left] \arrow[l, shift right]
%\end{tikzcd}
%\]
%This resolution gives rise to an unstable Adams spectral sequence for $v_n$-periodic homotopy groups. We say an object $X \in \mathcal{S}_{v_n}$ is \emph{$Q_{\mathcal{S}_{v_n}}$-complete} if the natural map
%\begin{equation*}
%X \rightarrow \mathrm{Tot}\bigl(Q_{\mathcal{S}_{v_n}}^{\bullet + 1}X \bigr)
%\end{equation*}
%is an equivalence. We will prove the following in Section \ref{sec:MnfGoodwillie}:

%\begin{theorem}
%\label{thm:Qcomplete}
%A space $X \in \mathcal{S}_{v_n}$ is $\Phi$-good if and only if it is $Q_{\mathcal{S}_{v_n}}$-complete.
%\end{theorem}

Finally, we sample some calculational consequences of our results in Section \ref{subsec:applications}. The following theorem shows that the calculation of the $v_n$-periodic homotopy groups of a large class of type $n$ spaces with $v_n$ self-maps can be translated completely into stable terms. 

\begin{theorem}
\label{thm:Phitypenspace}
Suppose $V$ is a pointed finite type $n$ space with a $v_n$ self-map and write $W = \Sigma^2 V$. Then there is an equivalence of spectra as follows:
\begin{equation*}
\Phi(W) \simeq L_{T(n)}\bigoplus_{k \geq 1}(\partial_k \mathrm{id} \otimes \Sigma^\infty W^{\otimes k})_{h\Sigma_k}.
\end{equation*}
\end{theorem}

This formula will result from an identification of the Goodwillie derivatives of $\Phi$ and a splitting of the Goodwillie tower of this functor when evaluated on spaces $W$ of the form descibed in the theorem. As an example, if $n=1$ and $p$ is odd, then the mod $p$ Moore space $M(\mathbf{Z}/p, 3)$ supports a $v_1$ self-map. In this case Theorem \ref{thm:Phitypenspace} expresses the $v_1$-periodic homotopy groups of $M(\mathbf{Z}/p, j)$, for $j \geq 5$, in terms of the stable homotopy groups of the homotopy orbit spectra on the right. It would be interesting to compare this splitting with the computations of Thompson \cite{thompson}, who uses the splitting of loop spaces of Moore spaces constructed by Cohen, Moore, and Neisendorfer.

Another application of Theorem \ref{thm:MnfLiealg} is the following. After identifying $\mathcal{S}_{v_n}$ with $\mathrm{Lie}(\mathrm{Sp}_{T(n)})$, the functor 
\begin{equation*}
\Omega^\infty_{T(n)}\colon \mathrm{Sp}_{T(n)} \rightarrow \mathcal{S}_{v_n}
\end{equation*}
can be identified with the \emph{trivial Lie algebra functor}
\begin{equation*}
\mathrm{triv}\colon \mathrm{Sp}_{T(n)} \rightarrow \mathrm{Lie}(\mathrm{Sp}_{T(n)}).
\end{equation*}
The left adjoint $\Sigma^\infty_{T(n)}$ is then identified with the functor taking \emph{derived indecomposables} or \emph{topological Quillen homology} (see for example \cite{basterramandell}). Roughly speaking, for any operad $\mathcal{O}$ (say in spectra) and an $\mathcal{O}$-algebra $A$, there is a natural filtration of the topological Quillen homology of $A$ whose associated graded can be expressed in terms of the bar construction of $\mathcal{O}$ (see Section 4 of \cite{behrensrezk}, where this filtration is called the \emph{Kuhn filtration}). We will define the relevant filtration of $\Sigma^\infty_{T(n)} X$ in the case of interest to us in Section \ref{subsec:applications} and show that the associated graded is the free (nonunital) symmetric algebra on $\Phi(X)$:

\begin{theorem}
\label{thm:Sigmainftyfiltration}
For $X \in \mathcal{S}_{v_n}$, the suspension spectrum $\Sigma^\infty_{T(n)} X$ admits a natural filtration with associated graded the spectrum
\begin{equation*}
\mathrm{Sym}^\ast(\Phi (X)) = L_{T(n)}\bigoplus_{k \geq 1}\Phi(X)^{\otimes k}_{h\Sigma_k}.
\end{equation*}
\end{theorem}

\begin{remark}
In the analogous case of differential graded Lie algebras over the rational numbers, topological Quillen homology is equivalent to the complex calculating Chevalley-Eilenberg homology. That complex, evaluated on a differential graded Lie algebra $L$, indeed has an evident filtration whose associated graded is the free nonunital symmetric algebra on the (shifted) underlying chain complex of $L$.
\end{remark}

Of course Theorem \ref{thm:Sigmainftyfiltration} becomes more useful when one knows something about $\Phi(X)$. As an example, take $n=1$ and $p$ odd again. Then Bousfield's interpretation \cite{bousfieldHspace} of a computation of Thompson  \cite{thompsonunstablesphere} (which is the odd-primary version of a calculation of Mahowald \cite{mahowaldImJ} at $p=2$) describes the spectrum $\Phi(S^{2\ell+1})$ as the localization of a mod $p^\ell$ Moore spectrum:
\begin{equation*}
\Phi(S^{2\ell+1}) \simeq L_{K(1)} \mathbb{S}^{2\ell}/p^{\ell}.
\end{equation*}

This leads to the following:
\begin{corollary}
\label{cor:K1localsphere}
Take $\ell \geq 1$. Then the $K(1)$-local sphere spectrum admits a filtration with associated graded the spectrum
\begin{equation*}
L_{K(1)}\Sigma^{-2\ell - 1}\Bigl(\bigoplus_{k \geq 1}(\mathbb{S}^{2\ell}/p^{\ell})^{\otimes k}_{h\Sigma_k}\Bigr).
\end{equation*}
\end{corollary}

It would be interesting to see if recent calculations of the Bousfield--Kuhn functor on spheres at heights greater than 1 \cite{wang, zhu} can be used to deduce further results along these lines.

\section{The $v_n$-periodic homotopy theory of spaces}
\label{sec:Mnf}

The goal of this section is to present the construction of the $\infty$-category $\mathcal{S}_{v_n}$ and prove Theorems \ref{thm:Mnf} and \ref{thm:Mnf2}. Most of this material is based on the work of Bousfield and Dror-Farjoun on localizations (or `nullifications') of spaces \cite{bousfieldlocalization,farjoun}. We will frequently cite results from Bousfield's papers \cite{bousfieldlocalization} and \cite{bousfieldtelescopic}. An alternative exposition of the theory of localizations specifically aimed at the results of this paper is contained in the notes for Harvard's Thursday seminar \cite{thursday}.

\subsection{The $\infty$-category $\mathcal{L}_n^f$}
\label{subsec:Lnf}

Fix a finite type $n+1$ suspension space $V_{n+1}$. We need not assume that $V_{n+1}$ admits a $v_{n+1}$ self-map. Bousfield chooses $V_{n+1}$ so that its connectivity is as low as possible. We do not make this restriction, because the additional freedom in choosing $V_{n+1}$ is useful in proving Theorem \ref{thm:Phimonadic}. Define $d_{n+1}$ to be $\mathrm{conn}(V_{n+1}) + 1$. In other words, $d_{n+1}$ is the dimension of the lowest nonvanishing homotopy group of $V_{n+1}$. As an example, $V_1$ may be taken to be the Moore space $M(\mathbb{Z}/p, 2)$, in which case $d_1 = 2$. If $k > l$, we will always assume that our choices of $V_{k+1}$ and $V_{l+1}$ are such that $d_k \geq d_l$.

For $d$ a positive integer, we will write $\mathcal{S}_*\langle d\rangle$ for the full subcategory of $\mathcal{S}_*$ spanned by the $d$-connected spaces. We will carry out many constructions using the $\infty$-category $\mathcal{S}_*\langle 1 \rangle$ of pointed $p$-local simply-connected spaces. Recall that our convention is to leave the adjective $p$-local implicit and assume it applies throughout. The $\infty$-category $\mathcal{S}_*\langle 1 \rangle$ is easily seen to be compactly generated.

\begin{definition}
The $\infty$-category $L_n^f\mathcal{S}_*\langle 1 \rangle$ is the localization of $\mathcal{S}_*\langle 1 \rangle$ with respect to the map $V_{n+1} \rightarrow *$. We write $L_n^f\colon \mathcal{S}_*\langle 1 \rangle \rightarrow \mathcal{S}_*\langle 1 \rangle$ for the corresponding localization functor, so that $L_n^f\mathcal{S}_*\langle 1 \rangle$ is the essential image of $L_n^f$.
\end{definition}

Here the word \emph{localization} is intended in the sense of Definition 5.2.7.2 of \cite{htt}. Thus $L_n^f\mathcal{S}_*\langle 1 \rangle$ is the full subcategory of $\mathcal{S}_*\langle 1 \rangle$ on the objects which are local with respect to the map $V_{n+1} \rightarrow *$ and the inclusion of this subcategory into $\mathcal{S}_*\langle 1 \rangle$ is a right adjoint functor. This is the $\infty$-categorical analogue of left Bousfield localization of model categories.

When comparing our definition of $L_n^f$ to Bousfield's in 4.3 of \cite{bousfieldtelescopic} there might seem to be a discrepancy: Bousfield also localizes with respect to prime to $p$ Moore spaces. Since we work in the $p$-local setting throughout, this point can be safely ignored.

\begin{lemma}
\label{lem:Lnfcompactlygen}
An object $Y \in L_n^f\mathcal{S}_*\langle 1 \rangle$ is compact if and only if it is equivalent to a retract of one of the form $L_n^f X$ for $X$ a compact object of $\mathcal{S}_*\langle 1 \rangle$. Moreover, the $\infty$-category $L_n^f\mathcal{S}_*\langle 1 \rangle$ is compactly generated.
\end{lemma}
\begin{proof}
Since $V_{n+1}$ is finite, the class of $L_n^f$-local objects of $\mathcal{S}_*\langle 1 \rangle$ is closed under filtered colimits and $L_n^f$ preserves filtered colimits. The lemma follows easily from this.
\end{proof}

At this point the reader might wonder to what extent the localization $L_n^f$ depends on the choice of $V_{n+1}$. It turns out only its connectivity is relevant. To be precise, let us say that spaces $W$ and $W'$ have the same \emph{Bousfield class} if the localizations with respect to $W \rightarrow *$ and $W' \rightarrow *$ are the same. We write $\langle W \rangle$ for the equivalence class of $W$ with respect to this relation and call it the Bousfield class of $W$. Define $\langle W \rangle \leq \langle W' \rangle$ if every space local with respect to $W' \rightarrow *$ is also local with respect to $W \rightarrow *$. Then Bousfield establishes the following unstable analogue of a stable classification result of Hopkins and Smith (see Theorem 9.15 of \cite{bousfieldlocalization}), which we state here for future reference:

\begin{theorem}
\label{thm:bousfieldclasses}
For $W$ and $W'$ finite suspension spaces of type greater than zero, the following are equivalent:
\begin{itemize}
\item[(i)] $\langle W \rangle \leq \langle W' \rangle$,
\item[(ii)] $\mathrm{type}(W) \geq \mathrm{type}(W')$ and $\mathrm{conn}(W) \geq \mathrm{conn}(W')$.
\end{itemize}
\end{theorem}

The localization functor $L_n^f$ captures the $v_i$-periodic homotopy groups for $i \leq n$, but kills them for $i > n$. More precisely, let $X$ be any pointed simply-connected space and $F \in \mathcal{S}_*\langle 1 \rangle$ a finite pointed type $i$ space with a $v_i$ self-map $v_i\colon \Sigma^d F \rightarrow F$. Then Theorem 5.2 of \cite{bousfieldlocalization} gives
\begin{equation*}
\pi_k L_0^f X \simeq \pi_k X \otimes \mathbb{Q} \quad\quad \text{if } k > d_1
\end{equation*}
and 4.6 of \cite{bousfieldtelescopic} states, for $i,n \geq 1$, that
\begin{equation*}
v_i^{-1}\pi_*(L_n^f X; F) \simeq \begin{cases} v_i^{-1} \pi_*(X; F) & \text{if } i \leq n, \\
0 & \text{if } i > n.
\end{cases}
\end{equation*}

We will need a functor $M_n^f$ which captures only the $v_n$-periodic homotopy groups of a pointed space. For this we use the following:

\begin{lemma}
\label{lem:LnLn-1}
An $L_{n-1}^f$-local pointed space is also $L_n^f$-local, so that $L_{n-1}^f \simeq L_n^f L_{n-1}^f$. 
\end{lemma}
\begin{proof}
This is immediate from Theorem \ref{thm:bousfieldclasses} and the fact that $\langle V_{n+1} \rangle \leq \langle V_n \rangle$.
\end{proof}

By the previous lemma, the natural transformation $L_n^f \rightarrow L_n^f L_{n-1}^f$ induces a natural transformation $L_n^f \rightarrow L_{n-1}^f$.

\begin{definition}
The functor $M_n^f$ is the fiber of $L_n^f \rightarrow L_{n-1}^f$.
\end{definition}

Our previous description of the periodic homotopy groups of $L_n^f X$ gives the following for $i, n \geq 1$:
\begin{equation*}
v_i^{-1}\pi_*(M_n^f X; F) \simeq \begin{cases} v_n^{-1} \pi_*(X; F) & \text{if } i = n, \\
0 & \text{if } i \neq n.
\end{cases}
\end{equation*}
Of course $M_0^f = L_0^f$. We have described the behaviour of $L_n^f$ and $M_n^f$ with respect to periodic homotopy groups. However, it is not quite true that these functors \emph{detect} periodic homotopy equivalences in any reasonable sense, since they do not affect the homotopy groups of a space in dimensions below the connectivity of $V_{n+1}$. This can be fixed by Theorem \ref{thm:periodicequivalence} below. First we need to know how $L_n^f$ interacts with taking highly connected covers. The following is Proposition 13.1 of \cite{bousfieldlocalization}:

\begin{proposition}
\label{prop:conncover}
The functor $L_n^f$ commutes with $d_{n+1}$-connected covers. More precisely, for any $X \in \mathcal{S}_*$ the space $L_n^f(X\langle d_{n+1}\rangle)$ is $d_{n+1}$-connected, so that there is a natural map
\begin{equation*}
L_n^f(X\langle d_{n+1}\rangle) \rightarrow (L_n^f X)\langle d_{n+1} \rangle,
\end{equation*}
and this map is an equivalence.
\end{proposition}

The reader should note that $v_n$-periodic homotopy groups are invariant under taking arbitrarily highly connected covers (at least for $n \geq 1$), so that replacing $X$ by $X\langle d_{n+1} \rangle$ is harmless from the point of view of periodic homotopy. The compatibility expressed by Proposition \ref{prop:conncover} is very useful because of the following two theorems. We write $\mathcal{S}_*\langle d_{n+1} \rangle$ for the full subcategory of $\mathcal{S}_*$ spanned by $d_{n+1}$-connected spaces; also, we write $\mathcal{L}_n^f$ for the $L_n^f$-localization of $\mathcal{S}_*\langle d_{n+1} \rangle$, which is the full subcategory of $\mathcal{S}_*\langle d_{n+1} \rangle$ spanned by the $L_n^f$-local spaces.

\begin{theorem}
\label{thm:periodicequivalence}
A map $\varphi\colon X \rightarrow Y$ in $\mathcal{S}_*\langle d_{n+1} \rangle$ is a $v_i$-periodic equivalence for each $0 \leq i \leq n$ if and only if $(L_n^f \varphi)\langle d_{n+1}\rangle$ is an equivalence.  Furthermore, it is a $v_n$-periodic equivalence if and only if $(M_n^f \varphi)\langle d_{n+1}\rangle$ is an equivalence.
\end{theorem}
\begin{proof}
The first sentence is Corollary 4.8 of \cite{bousfieldtelescopic}. If $n=0$ the second sentence adds no information, since $L_0^f = M_0^f$. Therefore take $n > 0$. Note first that $\varphi$ is a $v_n$-periodic equivalence if and only if $(M_n^f \varphi)\langle d_{n+1} \rangle$ is a $v_n$-periodic equivalence. Indeed, this follows from our description of the homotopy groups of $M_n^f$ above, together with the observation that the $v_n$-periodic homotopy groups of $(M_n^f X)\langle d_{n+1} \rangle$ are the same as those of $M_n^f X$ and similarly for $M_n^f Y$. Note that we may as well say that $(M_n^f \varphi)\langle d_{n+1} \rangle$ is a $v_i$-periodic equivalence for $0 \leq i \leq n$, since this is vacuous for $i < n$. We claim that $(M_n^f X)\langle d_{n+1} \rangle$ and $(M_n^f Y)\langle d_{n+1} \rangle$ are $L_n^f$-local; the conclusion that $(M_n^f \varphi)\langle d_{n+1} \rangle$ is a $v_n$-periodic equivalence if and only if it is an actual equivalence then follows from the first part of the theorem. To establish this claim, observe that $M_n^f X$ is $L_n^f$-local because it is the fiber of a map between $L_n^f$-local spaces. Applying Proposition \ref{prop:conncover} we find
\begin{equation*}
(M_n^f X) \langle d_{n+1} \rangle \simeq (L_n^f M_n^f X)\langle d_{n+1} \rangle \simeq L_n^f\bigl((M_n^f X)\langle d_{n+1} \rangle\bigr)
\end{equation*}
and similarly for $Y$, which completes the proof.
\end{proof}

\begin{theorem}
\label{thm:Lnfinitelimits}
The functor 
\begin{equation*}
\mathcal{S}_*\langle d_{n+1}\rangle \rightarrow \mathcal{L}_n^f: X \mapsto L_n^f X
\end{equation*}
preserves finite limits.
\end{theorem}

In other words, when restricted to $d_{n+1}$-connected spaces the localization $L_n^f$ is left exact. This will be important when comparing Goodwillie calculus in $\mathcal{S}_*$ to Goodwillie calculus in $\mathcal{L}_n^f$. For the proof of Theorem \ref{thm:Lnfinitelimits} we need two lemmas.

\begin{lemma}
\label{lem:loopspullback}
Let $\mathcal{C}$ and $\mathcal{D}$ be $\infty$-categories admitting finite limits and let $F\colon \mathcal{C} \rightarrow \mathcal{D}$ be a functor satisfying the following:
\begin{itemize}
\item[(i)] $F$ preserves terminal objects,
\item[(ii)] the evident natural transformation $F\Omega \rightarrow \Omega F$ is an equivalence.
\end{itemize}
Then for any pullback square $P$ in $\mathcal{C}$, the square $\Omega F(P)$ is a pullback in $\mathcal{D}$.
\end{lemma}
\begin{proof}
Consider a pullback square $P$ as follows:
\[
\begin{tikzcd}
A \arrow[r]\arrow[d] & C \arrow[d] \\
B \arrow[r] & D.
\end{tikzcd}
\]
Write $\eta$ for the resulting map $F(A) \rightarrow F(B) \times_{F(D)} F(C)$. We will define a map
\begin{equation*}
\varepsilon\colon \Omega F(C) \times_{\Omega F(D)} \Omega F(B) \rightarrow \Omega F(A)
\end{equation*}
and show it is homotopy inverse to the composition of $\Omega\eta$ and the evident `switch map'
\begin{equation*}
\tau\colon \Omega F(B) \times_{\Omega F(D)} \Omega F(C)  \simeq \Omega F(C) \times_{\Omega F(D)} \Omega F(B). 
\end{equation*} 
To do this consider the following diagram, in which all squares are pullbacks:
\[
\begin{tikzcd}
\Omega A	\arrow[d]\arrow[r]	& \Omega B \arrow[r]\arrow[d]		& \ast \arrow[d]					& \\
\Omega C \arrow[d]\arrow[r]	& \Omega D \arrow[r]\arrow[d]		& A \times_C \ast \arrow[d]\arrow[r]	& \ast \arrow[d] \\
\ast \arrow[r]				& \ast \times_B A \arrow[r]\arrow[d]	& A \arrow[d]\arrow[r]				& C \arrow[d] \\
						& \ast \arrow[r]					& B \arrow[r]					& D.
\end{tikzcd}
\]
Applying $F$ to this and replacing the square in the upper left-hand corner by a pullback, we obtain a diagram
\[
\begin{tikzcd}
\Omega F(C) \times_{\Omega F(D)} \Omega F(B) \arrow[d]\arrow[r]	& \Omega F(B) \arrow[r]\arrow[d]		& \ast \arrow[d]							& \\
\Omega F(C) \arrow[d]\arrow[r]								& \Omega F(D) \arrow[r]\arrow[d]		& F(A \times_C \ast) \arrow[d]\arrow[r]		& \ast \ar[d]  \\
\ast \arrow[r]											& F(\ast \times_B A) \arrow[r]\arrow[d]	& F(A) \arrow[d]\arrow[r]					& F(C) \arrow[d] \\
													& \ast \arrow[r]						& F(B) \arrow[r]					& F(D).
\end{tikzcd}
\]
Note that we have used assumptions (i) and (ii) in identifying the objects in this diagram. The four squares in the top left compose to give a larger square, whose homotopy-coherent commutativity defines the necessary map 
\begin{equation*}
\varepsilon\colon \Omega F(C) \times_{\Omega F(D)} \Omega F(B) \rightarrow \Omega F(A).
\end{equation*}
Note that the definition of $\varepsilon$ immediately implies that $\varepsilon \circ \tau \circ \Omega\eta$ is homotopic to the identity map of $\Omega F(A)$, as one observes by precomposing with 
\begin{equation*}
\tau \circ \Omega\eta\colon \Omega F(A) \rightarrow \Omega F(C) \times_{\Omega F(D)} \Omega F(B)
\end{equation*}
in the upper left-hand corner. The map $\eta$ arises from the square on the lower right. An elementary chase of the homotopies defined by the diagram now shows that $\Omega \eta \circ \varepsilon$ is homotopic to $\tau$, which completes the proof.
\end{proof}

\begin{lemma}
\label{lem:loopsequiv}
A map $\varphi$ in $\mathcal{L}_n^f$ is an equivalence if $\Omega \varphi$ is an equivalence and $\varphi$ is a rational equivalence. Here $\Omega$ is to be interpreted in the $\infty$-category $\mathcal{L}_n^f$, so that it is the $d_{n+1}$-connected cover of the usual loop space functor.
\end{lemma}
\begin{proof}
Assume $\varphi$ satisfies the two conditions of the lemma. Note that $\Omega \varphi$ is a $v_i$-periodic equivalence for each $1 \leq i \leq n$ if and only if $\varphi$ itself is, since the $v_i$-periodic homotopy groups of a pointed space $X$ (for $i \geq 1$) coincide up to a shift with those of $\Omega X$ (or any highly connected cover of it). Combining this with the second assumption, we see that $\varphi$ is a $v_i$-periodic equivalence for $0 \leq i \leq n$ and Theorem \ref{thm:periodicequivalence} applies.
\end{proof}

\begin{proof}[Proof of Theorem \ref{thm:Lnfinitelimits}]
Lemma 12.5 of \cite{bousfieldtelescopic} implies that the functor
\begin{equation*}
L_n^f\colon \mathcal{S}_*\langle d_{n+1} \rangle \rightarrow \mathcal{L}_n^f
\end{equation*}
preserves fiber sequences. Indeed, one simply uses that fibers in $\mathcal{S}_*\langle d_{n+1}\rangle$ and $\mathcal{L}_n^f$ are computed by taking the $d_{n+1}$-connected cover of the usual homotopy fiber. Moreover, the functor $L_n^f$ above clearly preserves the terminal object and therefore satisfies the conditions of Lemma \ref{lem:loopspullback}. It is well-known that rationalization preserves homotopy pullback squares of simply-connected spaces; alternatively, this is not hard to derive directly from the preservation of fiber sequences mentioned above. We can now apply Lemma \ref{lem:loopsequiv} together with the conclusion of Lemma \ref{lem:loopspullback} to see that $L_n^f$ (in the $d_{n+1}$-connected setting) preserves pullbacks. Any functor between $\infty$-categories with finite limits that preserves terminal objects and pullbacks preserves all finite limits (cf. the dual of Corollary 4.4.2.5 of \cite{htt}), which completes the proof.
\end{proof}

\subsection{The $\infty$-category $\mathcal{S}_{v_n}$}
\label{subsec:Mnf}

We aim to describe a homotopy theory of spaces in which the equivalences are the $v_n$-periodic equivalences, so that Theorem \ref{thm:periodicequivalence} clearly suggests the following:

\begin{definition}
The $\infty$-category $\mathcal{S}_{v_n}$ is the full subcategory of $\mathcal{S}_*$ spanned by the spaces of the form $(M_n^f X)\langle d_{n+1} \rangle$.
\end{definition}

We write $i$ for the inclusion of $\mathcal{S}_{v_n}$ in $\mathcal{S}_*$ and $M$ for the functor
\begin{equation*}
\mathcal{S}_* \rightarrow \mathcal{S}_{v_n}\colon X \mapsto (M_n^f X)\langle d_{n+1} \rangle.
\end{equation*}
This section is devoted to proving the necessary properties of $\mathcal{S}_{v_n}$. We briefly summarize these at the end when we state the proofs of Theorems \ref{thm:Mnf} and \ref{thm:Mnf2}.

\begin{lemma}
\label{lem:Mnidempotent}
There is a natural equivalence $M \circ i \simeq \mathrm{id}_{\mathcal{S}_{v_n}}$.
\end{lemma}
\begin{proof}
This is easy for $n=0$, so let us assume $n > 0$ for the rest of this proof. Let $X = (M_n^f Y)\langle d_{n+1} \rangle$ for some pointed space $Y$. By definition there is a fiber sequence
\begin{equation*}
M_n^f X \rightarrow L_n^f X \rightarrow L_{n-1}^f X.
\end{equation*}
Since $X$ is already $L_n^f$-local and $d_{n+1}$-connected we have $X \simeq L_n^f X \simeq L_n^f X \langle d_{n+1} \rangle$ and it suffices to prove that the first map in the sequence is an equivalence. In other words, it suffices to prove that $L_{n-1}^f X$ is contractible. For this, we apply $L_{n-1}^f\langle d_n \rangle$ to the fiber sequence
\begin{equation*}
M_n^f Y \rightarrow L_n^f Y \rightarrow L_{n-1}^f Y
\end{equation*}
and apply Theorem \ref{thm:Lnfinitelimits} to conclude that
\begin{equation*}
L_{n-1}^f(M_n^f Y)\langle d_n \rangle \rightarrow L_{n-1}^f(L_n^fY)\langle d_n \rangle \rightarrow L_{n-1}^fY\langle d_n\rangle
\end{equation*}
is a fiber sequence in $\mathcal{L}_{n-1}^f$. Since $L_{n-1}^f L_n^f \simeq L_{n-1}^f$ the map on the right is an equivalence and the fiber $L_{n-1}^f(M_n^f Y)\langle d_n \rangle$ is contractible. To finish the proof, we claim that the map 
\begin{equation*}
L_{n-1}^f X = L_{n-1}^f((M_n^f Y)\langle d_{n+1} \rangle) \rightarrow L_{n-1}^f(M_n^f Y)\langle d_n \rangle
\end{equation*}
is an equivalence. This follows from Theorem \ref{thm:periodicequivalence} if
\begin{equation*}
(M_n^f Y)\langle d_{n+1} \rangle \rightarrow (M_n^f Y)\langle d_n \rangle
\end{equation*}
is a $v_i$-periodic equivalence for all $i$. This is clear for $i=0$ (since both spaces are rationally trivial) and true for $i>0$ since the homotopy groups of $(M_n^f Y)\langle d_{n+1} \rangle$ and $(M_n^f Y)\langle d_n \rangle$ can differ only in the \emph{finite} range of dimensions $[d_n + 1, d_{n+1}]$.
\end{proof}

\begin{corollary}
\label{cor:Mnequiv}
A map $\varphi\colon X \rightarrow Y$ in $\mathcal{S}_{v_n}$ is an equivalence if and only if it is a $v_n$-periodic equivalence.
\end{corollary}
\begin{proof}
Assume $\varphi$ is a $v_n$-periodic equivalence. Then $M(i\varphi)$ is an equivalence by Theorem \ref{thm:periodicequivalence}. The previous lemma gives a commutative square
\[
\begin{tikzcd}
M(iX) \arrow[r,"M(i\varphi)"]\arrow[d] & M(iY) \arrow[d] \\
X \arrow[r,"\varphi"] & Y
\end{tikzcd}
\]
in which the vertical arrows are equivalences, so that $\varphi$ is an equivalence by two-out-of-three.
\end{proof}

The last ingredient we will need for the proof of Theorem \ref{thm:Mnf2} is the following:

\begin{proposition}
\label{prop:Mncompactlygen}
The $\infty$-category $\mathcal{S}_{v_n}$ is compactly generated. For $n > 0$ and $V$ a finite $d_{n+1}$-connected type $n$ space that is also a suspension, the space $L_n^f V$ is contained in $\mathcal{S}_{v_n}$ and is a compact generator.
\end{proposition}

To prove this result let us introduce an auxiliary $\infty$-category:

\begin{definition}
\label{def:Vn}
The $\infty$-category $\mathcal{V}_n$ is the full subcategory of $\mathcal{L}_n^f$ generated under colimits by pointed spaces of the form $L_n^f V$, with $V$ ranging over suspension spaces which are finite, $d_{n+1}$-connected, and of type $n$.
\end{definition}

Clearly $\mathcal{V}_n$ is compactly generated by the spaces $L_n^f V$. Also, the inclusion $\mathcal{V}_n \rightarrow \mathcal{L}_n^f$ preserves all colimits (by definition). The adjoint functor theorem (Corollary 5.5.2.9 of \cite{htt}) implies that it admits a right adjoint $r\colon \mathcal{L}_n^f \rightarrow \mathcal{V}_n$. We will prove that $\mathcal{V}_n$ and $\mathcal{S}_{v_n}$ coincide. Thus $\mathcal{S}_{v_n}$ is a \emph{colocalization} of $\mathcal{L}_n^f$, meaning a full subcategory for which the inclusion functor admits a right adjoint.

\begin{lemma}
\label{lem:VninMn}
The subcategory $\mathcal{V}_n$ is contained in the subcategory $\mathcal{S}_{v_n}$.
\end{lemma}
\begin{proof}
Let $V$ be as in Definition \ref{def:Vn}. Then $L_{n-1}^f V$ is null by the classification of Bousfield classes of spaces described in Theorem \ref{thm:bousfieldclasses} (note that this uses that $V$ is a suspension). It follows that $L_{n-1}^f X$ is null for any $X \in \mathcal{V}_n$, so that $X \simeq M_n^f X \simeq M_n^f X\langle d_{n+1} \rangle$.
\end{proof}

\begin{proof}[Proof of Proposition \ref{prop:Mncompactlygen}]
It is easy to see that $\mathcal{V}_0 = \mathcal{M}_0^f$, since both are in fact the $\infty$-category of $d_1$-connected rational pointed spaces, so we focus our attention on the case $n>0$. The right adjoint $r\colon \mathcal{L}_n^f \rightarrow \mathcal{V}_n$ restricts to give a right adjoint $r\colon \mathcal{S}_{v_n} \rightarrow \mathcal{V}_n$ to the inclusion $\iota\colon \mathcal{V}_n \rightarrow \mathcal{S}_{v_n}$. We will prove that the latter adjoint pair is an equivalence. First, the unit
\begin{equation*}
\eta\colon \mathrm{id}_{\mathcal{V}_n} \rightarrow r \circ \iota
\end{equation*}
is an equivalence because $\iota$ is fully faithful (being an inclusion of full subcategories). We claim (see below) that $r$ detects equivalences. It follows that the adjoint pair $(\iota,r)$ is an adjoint equivalence by the following standard argument: to check that the counit $\varepsilon\colon \iota \circ r \rightarrow \mathrm{id}_{\mathcal{S}_{v_n}}$ is an equivalence we may check that $r\varepsilon$ is an equivalence. This follows from the triangle identity
\[
\begin{tikzcd}
& r \circ \iota \circ r \arrow[dr,"\eta r"] & \\
r \arrow[rr, equal]\arrow[ur,"r\varepsilon"] && r
\end{tikzcd}
\]
and the fact that $\eta r$ is an equivalence.

To establish our claim, consider a map $\varphi\colon X \rightarrow Y$ in $\mathcal{S}_{v_n}$ and assume $r(\varphi)$ is an equivalence. Pick a highly connected finite type $n$ suspension space $W$ (so that $L_n^f W \in \mathcal{V}_n$) that admits a $v_n$ self-map $v\colon \Sigma^d W \rightarrow W$. By assumption, the map
\begin{equation*}
\mathrm{Map}_*(L_n^f W, rX) \rightarrow \mathrm{Map}_*(L_n^f W,rY)
\end{equation*} 
is an equivalence. By adjunction these spaces can be identified with $\mathrm{Map}_*(W,X)$ and $\mathrm{Map}_*(W,Y)$ respectively (omitting the inclusion $i$ from the notation). It follows immediately that $\Phi_v(\varphi)$ is an equivalence, so that $\varphi$ is a $v_n$-periodic equivalence, and Corollary \ref{cor:Mnequiv} implies that $\varphi$ itself is an equivalence.

Since $\mathcal{V}_n = \mathcal{S}_{v_n}$ it follows that $\mathcal{S}_{v_n}$ is compactly generated. The argument above actually shows that $W$ is a generator. But for any $V$ as in the statement of Proposition \ref{prop:Mncompactlygen}, some suspension of $V$ admits a $v_n$ self-map, so that $V$ is a generator as well.
\end{proof}

\begin{proof}[Proof of Theorem \ref{thm:Mnf}]
This theorem is a straightforward consequence of Lemma \ref{lem:Mnidempotent} and our constructions. Indeed, consider the functor
\begin{equation*}
i^*\colon \mathrm{Fun}_{v_n}(\mathcal{S}_*, \mathcal{C}) \rightarrow \mathrm{Fun}(\mathcal{S}_{v_n}, \mathcal{C})\colon F \mapsto F \circ i.
\end{equation*}
Then $i^*M^*$ is equivalent to the identity simply because $M \circ i \simeq \mathrm{id}_{\mathcal{S}_{v_n}}$. To show that $M^*i^*$ is equivalent to the identity, consider a functor $F\colon \mathcal{S}_* \rightarrow \mathcal{C}$ which sends $v_n$-periodic equivalences to equivalences in $\mathcal{C}$. For $X$ a pointed space there is a natural zigzag of maps
\[
\begin{tikzcd}
X & X\langle d_{n+1} \rangle \arrow[l] \arrow[r] & L_n^f X \langle d_{n+1} \rangle & iM(X) \arrow[l],
\end{tikzcd} 
\]
and all of these are $v_n$-periodic equivalences. Consequently there is a natural zigzag of equivalences
\[
\begin{tikzcd}
F(X) & F(X\langle d_{n+1} \rangle) \arrow[l] \arrow[r] & F(L_n^f X \langle d_{n+1} \rangle) & F(iM(X)) \arrow[l],
\end{tikzcd} 
\]
showing in particular that $F$ is naturally equivalent to $M^*i^*F$.
\end{proof}

\begin{proof}[Proof of Theorem \ref{thm:Mnf2}]
The first sentence of the theorem is Lemma \ref{lem:Mnidempotent}. Item (i) is part of Theorem \ref{thm:periodicequivalence}, whereas the fact that $\Phi$ factors through $M$ is immediate from Theorem \ref{thm:Mnf}. For part (iii) of the theorem, note that $M$ is the composition of functors
\begin{equation*}
\mathcal{S}_*\langle d_{n+1} \rangle \xrightarrow{\quad L_n^f \quad} \mathcal{L}_n^f \xrightarrow{\quad r \quad} \mathcal{S}_{v_n}.
\end{equation*}
The first one preserves finite limits by Theorem \ref{thm:Lnfinitelimits} and filtered colimits by the fact that it is a left adjoint. The second functor preserves finite limits because it is a right adjoint and filtered colimits since its left adjoint sends a compact generator of $\mathcal{S}_{v_n}$ to a compact object of $\mathcal{L}_n^f$, cf. Lemma \ref{lem:Lnfcompactlygen} and Proposition \ref{prop:Mncompactlygen}. That proposition of course also implies part (iv).

Finally, we should establish the existence of the left adjoint $\Theta$ to $\Phi$. Bousfield shows that $\Theta$ exists on the level of homotopy categories in Theorem 5.4(i),(ii) of \cite{bousfieldtelescopic}. However, his techniques also prove the stronger result. Indeed, first one considers a type $n$ space $V$ with $v_n$ self-map $v$ and the resulting telescopic functor
\begin{equation*}
\Phi_v\colon \mathcal{L}_n^f \rightarrow \mathrm{Sp}_{T(n)}.
\end{equation*}
The $\infty$-categories involved can be constructed from the simplicial model categories Bousfield uses and he shows that $\Phi_v$ is a simplicial right Quillen functor between those categories (Lemma 10.6 of \cite{bousfieldtelescopic}). Therefore its adjoint, being a simplicial left Quillen functor, gives a functor
\begin{equation*}
\Theta_v\colon \mathrm{Sp}_{T(n)} \rightarrow \mathcal{L}_n^f 
\end{equation*}
after passing back to the corresponding $\infty$-categories. (An alternative exposition of the same ideas can be found in Section 6 of \cite{kuhntelescopic}.) The Bousfield--Kuhn functor $\Phi$ is a homotopy limit of telescopic functors $\Phi_v$. Indeed, following Kuhn \cite{kuhntelescopic,kuhninfiniteloop}, one fixes a directed system of finite type $n$ spectra
\begin{equation*}
F(1) \rightarrow F(2) \rightarrow F(3) \rightarrow \cdots 
\end{equation*}
with a map to the sphere spectrum
\begin{equation*}
\varinjlim_k F(k) \rightarrow S
\end{equation*}
which is a $T(n)$-equivalence. Then
\begin{equation*}
\Phi \simeq \varprojlim_k \mathbf{D}F(k) \otimes \Phi,
\end{equation*}
and each term $\mathbf{D}F(k) \otimes \Phi$ is equivalent to a telescopic functor of the form $\Phi_v$. It follows that $\Phi$ preserves limits, being a limit of right adjoints. Since it also preserves filtered colimits, it is accessible, and the adjoint functor theorem (Corollary 5.5.2.9 of \cite{htt}) applies to guarantee the existence of a left adjoint 
\begin{equation*}
\Theta\colon \mathrm{Sp}_{T(n)} \rightarrow \mathcal{L}_n^f.
\end{equation*}
Theorem 5.4 of \cite{bousfieldtelescopic} guarantees that the essential image of $\Theta$ is contained in $\mathcal{S}_{v_n}$, so that the adjoint pair $(\Theta, \Phi)$ restricts to give an adjunction between $\mathrm{Sp}_{T(n)}$ and $\mathcal{S}_{v_n}$.
\end{proof}

We will come back to $\Theta$ in Section \ref{subsec:BKleftadjoint} and describe it more explicitly. To end this section we record the following for later use:

\begin{lemma}
\label{lem:Mnffinitelim}
The inclusion $\iota\colon \mathcal{S}_{v_n} \rightarrow \mathcal{L}_n^f$ preserves colimits and finite limits.
\end{lemma}
\begin{proof}
We already concluded above that $\mathcal{S}_{v_n} = \mathcal{V}_n$ is a colocalization of $\mathcal{L}_n^f$, meaning in particular that the inclusion $\iota$ preserves colimits. Now consider a finite diagram
\begin{equation*}
F\colon I \rightarrow \mathcal{S}_{v_n}.
\end{equation*}
Writing $\varprojlim_{\mathcal{L}_n^f} \iota F$ for its limit when considered as a diagram in $\mathcal{L}_n^f$, we have a fiber sequence in $\mathcal{L}_n^f$ as follows:
\begin{equation*}
M_n^f(\varprojlim_{\mathcal{L}_n^f} \iota F)\langle d_{n+1} \rangle \rightarrow \varprojlim_{\mathcal{L}_n^f} \iota F \rightarrow L_{n-1}^f(\varprojlim_{\mathcal{L}_n^f} \iota F)\langle d_{n+1} \rangle.
\end{equation*}
But $L_{n-1}^f(\iota F(i))$ is null for every $i \in I$, so that Theorem \ref{thm:Lnfinitelimits} implies that the rightmost expression is contractible. Thus we find
\begin{equation*}
\varprojlim_{\mathcal{S}_{v_n}} F = M_n^f(\varprojlim_{\mathcal{L}_n^f} \iota F)\langle d_{n+1} \rangle \simeq  \varprojlim_{\mathcal{L}_n^f} \iota F. 
\end{equation*}
\end{proof}

\subsection{The stabilization of $\mathcal{S}_{v_n}$}

In this section we determine the stabilization of $\mathcal{S}_{v_n}$ and verify a claim made in Section \ref{sec:mainresults}. The results in this section will also be useful later, when we determine the Goodwillie tower of $\mathcal{S}_{v_n}$. 

The localization functor $L_n^f$ studied before has a stable counterpart, where one localizes the $\infty$-category $\mathrm{Sp}$ with respect to the map $V \rightarrow *$ for a finite type $n+1$ \emph{spectrum} (rather than space) $V$. Note that here we mean localization in the stable sense, so that a spectrum $E$ is $L_n^f$-local if and only if $[\Sigma^i V, E] = 0$ for all $i \in \mathbb{Z}$, rather than just $i \geq 0$. We write
\begin{equation*}
L_n^f\colon \mathrm{Sp} \rightarrow L_n^f\mathrm{Sp}
\end{equation*}
for this localization; whether the stable or unstable $L_n^f$ is meant should always be clear from context.

\begin{proposition}
\label{prop:Lnfstabilization}
The functor $L_n^f\Sigma^\infty\colon \mathcal{L}_n^f \rightarrow L_n^f\mathrm{Sp}$ induces an equivalence of stable $\infty$-categories $\mathrm{Sp}(\mathcal{L}_n^f) \rightarrow L_n^f\mathrm{Sp}$.
\end{proposition}
\begin{proof}
This follows from universal properties. Indeed, let $\mathcal{C}$ be any presentable stable $\infty$-category. The universal property of the stabilization $\mathrm{Sp}(\mathcal{L}_n^f)$ is that there is a natural equivalence
\begin{equation*}
\mathrm{Fun}^L(\mathrm{Sp}(\mathcal{L}_n^f), \mathcal{C}) \rightarrow \mathrm{Fun}^L(\mathcal{L}_n^f, \mathcal{C}),
\end{equation*}
where $\mathrm{Fun}^L$ denotes the $\infty$-category of colimit-preserving functors. By the universal property of localization, the latter is naturally equivalent to $\mathrm{Fun}^L_{V_{n+1}}(\mathcal{S}_*\langle d_{n+1} \rangle, \mathcal{C})$, where the subscript $V_{n+1}$ indicates the full subcategory of $\mathrm{Fun}^L(\mathcal{S}_*\langle d_{n+1}\rangle, \mathcal{C})$ spanned by functors sending the map $V_{n+1} \rightarrow *$ to an equivalence. We can now conclude by observing the natural equivalences
\begin{equation*}
\mathrm{Fun}^L(L_n^f\mathrm{Sp}, \mathcal{C}) \rightarrow \mathrm{Fun}^L_{\Sigma^\infty V_{n+1}}(\mathrm{Sp}, \mathcal{C}) \rightarrow \mathrm{Fun}^L_{V_{n+1}}(\mathcal{S}_*\langle d_{n+1} \rangle, \mathcal{C}).
\end{equation*}
\end{proof} 

Considering $\mathcal{S}_{v_n}$ as a full subcategory of $\mathcal{L}_n^f$, we can restrict the functor of the previous proposition to obtain a functor
\begin{equation*}
L_n^f\Sigma^\infty\colon \mathcal{S}_{v_n} \rightarrow L_n^f\mathrm{Sp}.
\end{equation*}
As before there is a functor $M_n^f$ defined as the fiber between the stable localizations $L_n^f \rightarrow L_{n-1}^f$. We write $M_n^f\mathrm{Sp}$ for the full subcategory of $L_n^f\mathrm{Sp}$ on spectra of the form $M_n^f X$. The subcategory $M_n^f\mathrm{Sp}$ is precisely the one generated under colimits by spectra of the form $L_n^f V$, for $V$ ranging over finite type $n$ spectra. The reader not familiar with these facts can consult Section 3 of \cite{bousfieldtelescopic} for an exposition. The notation $M_n^f$ is Bousfield's and refers to the term \emph{monocular spectra}, which was coined by Ravenel.

Since $\mathcal{S}_{v_n}$ is generated under colimits by the $L_n^f$-localization of a highly connected finite type $n$ suspension space, the essential image of the functor $L_n^f\Sigma^\infty$ is contained in the subcategory $M_n^f\mathrm{Sp}$.

\begin{proposition}
\label{prop:Mnfstabilization}
The left adjoint functor $L_n^f\Sigma^\infty$ induces an equivalence of stable $\infty$-categories $\mathrm{Sp}(\mathcal{S}_{v_n}) \rightarrow M_n^f\mathrm{Sp}$. Its right adjoint $M_n^f\mathrm{Sp} \rightarrow \mathcal{S}_{v_n}$ is the functor $M \Omega^\infty$, or equivalently just the $d_{n+1}$-connected cover of the usual functor $\Omega^\infty$.
\end{proposition}
\begin{proof}
Consider the following commutative diagram of $\infty$-categories, where the superscripts $\omega$ indicate the full subcategories on compact objects:
\[
\begin{tikzcd}
(\mathcal{S}_{v_n})^\omega \arrow[r,"\Sigma"] \arrow[d,"\iota"] & (\mathcal{S}_{v_n})^\omega \arrow[r,"\Sigma"] \arrow[d,"\iota"] & (\mathcal{S}_{v_n})^\omega \arrow[r,"\Sigma"] \arrow[d,"\iota"] & \cdots \\
(\mathcal{L}_n^f)^{\omega} \arrow[r,"\Sigma"] & (\mathcal{L}_n^f)^{\omega} \arrow[r,"\Sigma"] & (\mathcal{L}_n^f)^{\omega} \arrow[r,"\Sigma"] & \cdots.
\end{tikzcd}
\]
The vertical functors are fully faithful, so that the colimit
\begin{equation*}
\mathrm{Sp}(\mathcal{S}_{v_n})^{\omega} \rightarrow (L_n^f\mathrm{Sp})^{\omega}
\end{equation*}
is fully faithful as well. This functor factors through the full subcategory $(M_n^f\mathrm{Sp})^\omega$. To see that this subcategory is also its essential image, we only have to show that this image contains a generator of $M_n^f\mathrm{Sp}$. As already mentioned, the $L_n^f$-localization of a finite type $n$ spectrum is such a generator. Any generator of $\mathcal{S}_{v_n}$ as described in Proposition \ref{prop:Mncompactlygen} is sent to such a spectrum by $L_n^f\Sigma^\infty$. The identification of the right adjoint $M\Omega^\infty$ is immediate from the fact that $M\colon \mathcal{L}_n^f \rightarrow \mathcal{S}_{v_n}$ is right adjoint to the inclusion $\mathcal{S}_{v_n} \rightarrow \mathcal{L}_n^f$. To conclude the final statement of the proposition, note that for a spectrum $X \in M_n^f\mathrm{Sp}$ the space $\Omega^\infty X$ is $L_n^f$-local and that its $v_i$-periodic homotopy groups vanish for $i < n$. Theorem \ref{thm:periodicequivalence} then implies that $L_{n-1}^f(\Omega^\infty X) \langle d_{n+1} \rangle \simeq 0$, so that 
\begin{equation*}
M\Omega^\infty X = (M_n^f\Omega^\infty X) \langle d_{n+1} \rangle \simeq (L_n^f \Omega^\infty X)\langle d_{n+1} \rangle \simeq \Omega^\infty X \langle d_{n+1} \rangle.
\end{equation*}
\end{proof}

\begin{remark}
\label{rmk:MnfLTn}
It is a standard fact (and easy to show) that the functors
\[
\begin{tikzcd}
M_n^f\mathrm{Sp} \arrow[r,"L_{T(n)}", shift left] & \mathrm{Sp}_{T(n)} \arrow[l,"M_n^f", shift left]
\end{tikzcd}
\]
form an adjoint equivalence of stable $\infty$-categories, where $\mathrm{Sp}_{T(n)}$ denotes the Bousfield localization of $\mathrm{Sp}$ with respect to $T(n)$-homology (see for example Theorem 3.3 of \cite{bousfieldtelescopic}). Note that there is no corresponding simple statement in the unstable setting. The relation between the $\infty$-category $\mathcal{S}_{v_n}$ and the localization $L_{T(n)}\mathcal{S}_*$ is much more subtle. The interested reader is encouraged to delve into Bousfield's detailed results expressing the relation between $v_n$-periodic equivalences and $T(n)_*$-equivalences of spaces, e.g. Theorem 13.15 of \cite{bousfieldlocalization}.

The adjoint equivalence above shows that the adjoint pair
\[
\begin{tikzcd}
\mathcal{S}_{v_n} \arrow[r,"\Sigma^\infty_{T(n)}", shift left] & \mathrm{Sp}_{T(n)} \arrow[l,"\Omega^\infty_{T(n)}", shift left]
\end{tikzcd}
\]
exhibits $\mathrm{Sp}_{T(n)}$ as the stabilization of $\mathcal{S}_{v_n}$ as well. Here $\Sigma^\infty_{T(n)} = L_{T(n)}(L_n^f\Sigma^\infty) \simeq L_{T(n)}\Sigma^\infty$ and for the right adjoint we have the formula 
\begin{equation*}
\Omega^\infty_{T(n)} = M\Omega^\infty\circ M_n^f = (\Omega^\infty M_n^f) \langle d_{n+1} \rangle.
\end{equation*}
%Since the map $M_n^f E \rightarrow E$ is a $v_n$-periodic equivalence of spectra (for any $L_n^f$-local $E$) and $M$ sends $v_n$-periodic equivalences to equivalences, there is in fact also the simpler description $\Omega^\infty_{T(n)} \simeq M\Omega^\infty$.
\end{remark}

\subsection{The Bousfield--Kuhn functor}
\label{subsec:BKleftadjoint}

In this section we collect some facts about the Bousfield--Kuhn functor and its left adjoint $\Theta$, some of which we will use later.

\begin{proposition}
\label{prop:Phifilteredcolimits}
The Bousfield--Kuhn functor $\Phi\colon \mathcal{S}_{v_n} \rightarrow \mathrm{Sp}_{T(n)}$ preserves filtered colimits.
\end{proposition}
\begin{proof}
Let $F\colon I \rightarrow \mathcal{S}_{v_n}$ be a filtered diagram. We need to verify that the canonical map
\begin{equation*}
\varinjlim_I \Phi \circ F \rightarrow \Phi(\varinjlim_I F)
\end{equation*}
is an equivalence of $T(n)$-local spectra. It suffices to check this after smashing both sides with $\mathbf{D}V$ for some finite type $n$ space $V$. We may choose $V$ so that it has a $v_n$ self-map $v\colon \Sigma^d V \rightarrow V$ and thus (by property (i) of $\Phi$ given in Section \ref{sec:mainresults}) reduce to checking that
\begin{equation*}
\varinjlim_I \Phi_v \circ F \rightarrow \Phi_v(\varinjlim_I F)
\end{equation*}
is an equivalence. But $\Phi_v$ is easily seen to preserve filtered colimits, since it is built from functors of the form $\mathrm{Map}(V, -)$ with $V$ finite.
\end{proof}

\begin{corollary}
The left adjoint $\Theta\colon \mathrm{Sp}_{T(n)} \rightarrow \mathcal{S}_{v_n}$ preserves compact objects.
\end{corollary}

\begin{remark}
\label{rmk:Tncompact}
An object of $\mathrm{Sp}_{T(n)}$ is compact precisely if it is a retract of a spectrum of the form $L_{T(n)}F$, with $F$ a finite spectrum of type $n$.
\end{remark}

In fact one can be much more explicit about the values of $\Theta$ when evaluated on (the $T(n)$-localizations of) finite type $n$ spectra (cf. Corollary 5.9 of \cite{bousfieldtelescopic}):

\begin{lemma}
\label{lem:valuetheta}
Let $V$ be a $(d_{n+1}-2)$-connected finite type $n$ space, so that $L_n^f\Sigma^2 V \in \mathcal{S}_{v_n}$. If $V$ admits a $v_n$ self map $v\colon \Sigma^d V \rightarrow V$, then there is a canonical equivalence 
\begin{equation*}
L_n^f \Sigma^2 V \simeq \Theta(\Sigma^\infty_{T(n)}\Sigma^2 V).
\end{equation*}
\end{lemma}
\begin{proof}
For any $i \geq 1$, the space $\Sigma^i\mathrm{cof}(v)$ is a $d_{n+1}$-connective suspension space of type $n+1$. Theorem \ref{thm:bousfieldclasses} then implies that $\mathrm{Map}_*(\Sigma^i\mathrm{cof}(v), X)$ is contractible for $X \in \mathcal{L}_n^f$. The evident long exact sequence argument gives the first isomorphism in the following sequence of identifications:
\begin{eqnarray*}
\pi_0\mathrm{Map}_*(\Sigma^2 V, X) & \cong & v^{-1}\pi_0\mathrm{Map}_*(\Sigma^2 V, X) \\
& \cong & \pi_0\bigl(\mathbf{D}(\Sigma^2 V) \otimes \Phi(X)\bigr) \\
& \cong & \pi_0\mathrm{Map}(\Sigma^\infty\Sigma^2 V, \Phi(X)) \\
& \cong & \pi_0\mathrm{Map}_*(\Theta(\Sigma^\infty_{T(n)}\Sigma^2 V), X).
\end{eqnarray*}
This implies the lemma.
\end{proof}

\begin{remark}
\label{rmk:valuetheta}
If one assumes that the type $n+1$ space $V_{n+1}$ used to define the localization $L_n^f$ has been chosen so that its connectivity is as low as possible, then the assumption on the connectivity of $V$ in the previous lemma can be omitted. Indeed, the connectivity $c$ of $\Sigma\mathrm{cof}(v)$ is by assumption at least that of $V_{n+1}$. But $c = \mathrm{conn}(V) + 1$, so that $V$ itself is at least $(d_{n+1} -1)$-connective, or equivalently $(d_{n+1}-2)$-connected.
\end{remark}

A consequence of the previous lemma is the following. Consider a finite type $n$ spectrum $F$. To describe $\Theta(L_{T(n)}F)$, choose a $v_n$ self-map $v\colon \Sigma^d F \rightarrow F$ and pick $j$ sufficiently large so that $\Sigma^{jd} F$ is equivalent to a suspension spectrum $\Sigma^\infty\Sigma^2 V$ for some highly connected finite pointed space $V$. Then the choice of $v$ and the preceding lemma give equivalences
\begin{equation*}
\Theta(L_{T(n)} F) \simeq \Theta(L_{T(n)} \Sigma^{jd} F) \simeq L_n^f\Sigma^2V.
\end{equation*}

From this one also gets a description of the value of $\Theta$ on the localized sphere spectrum $L_{T(n)} \mathbb{S}$. Indeed, first one chooses a directed system of finite type $n$ spectra
\begin{equation*}
F(1) \xrightarrow{f(1)} F(2) \xrightarrow{f(2)} F(3) \xrightarrow{f(3)} \cdots
\end{equation*}
with a $T(n)$-equivalence
\begin{equation*}
\varinjlim_k F(k) \rightarrow \mathbb{S}.
\end{equation*}
Thus we also have $\Theta(\mathbb{S}) \simeq \varinjlim_{k} \Theta(F(k))$. To make this more explicit using Lemma \ref{lem:valuetheta}, one `lifts' the diagram of $F(k)$'s to a diagram of pointed spaces of the following form (much as in Section 6.3 of \cite{kuhntelescopic}):
\[
\begin{tikzcd}
\Sigma^{2+ i(1)d(1)} V(1) \ar{d}[swap]{v(1)^{i(1)}} \ar{dr} & \Sigma^{2+ i(2)d(2)} V(2) \ar{d}[swap]{v(2)^{i(2)}} \ar{dr} & \Sigma^{2+ i(3)d(3)} V(3) \ar{d}[swap]{v(3)^{i(3)}}\ar{dr} & \\
\Sigma^2 V(1)	& \Sigma^2 V(2) & \Sigma^2 V(3) & \cdots . 
\end{tikzcd}
\]
To build this diagram, one first chooses a $v_n$ self-map $u(k)\colon \Sigma^{d(k)} F(k) \rightarrow F(k)$ for every $k \geq 1$. Then one chooses $V(k)$ to be a $d_{n+1}$-connective finite type $n$ space for which $\Sigma^\infty\Sigma^2 V(k)$ is equivalent to $\Sigma^{j(k)d(k)} F(k)$ for some sufficiently large $j(k)$ and for which $\Sigma^{j(k)-2}u(k)$ desuspends to a self-map $v(k)\colon \Sigma^{d(k)}V(k) \rightarrow V(k)$. Enlarging the $j(k)$ if necessary, one can assume they are such that there is a commutative diagram
\[
\begin{tikzcd}
\Sigma^{(i(k) + j(k))d(k)} F(k) \ar{d}[swap]{u(k)^{i(k)}} \ar{r}{f(k)} & \Sigma^{j(k+1)d(k+1)} F(k+1) \ar{d}{u(k+1)^{J}} \\
\Sigma^{j(k)d(k)} F(k) \ar{r}{f(k)} & \Sigma^{j(k)d(k)} F(k+1)
\end{tikzcd}
\]
for all $k$, where $j(k)d(k) + Jd(k+1) = j(k+1)d(k+1)$ and where we have omitted the necessary suspensions of $f(k)$ and $u(k)$ from the notation. Furthermore, one can assume that the top horizontal map desuspends to a map
\begin{equation*}
\Sigma^{2+ i(k)d(k)} V(k) \rightarrow \Sigma^2 V(k+1),
\end{equation*}
which is the map featuring in the earlier diagram above. After applying $L_n^f$ to that diagram, all the vertical arrows of course become equivalences. Moreover, Lemma \ref{lem:valuetheta} and the fact that $\Theta$ is left adjoint show that the colimit of the resulting diagram in $\mathcal{S}_{v_n}$ produces $\Theta(\mathbb{S})$. Informally one might write
\begin{equation*}
\Theta(\mathbb{S}) \simeq \varinjlim_k L_n^f(\Sigma^2V(k)),
\end{equation*}
although one should keep in mind that the spaces $V(k)$ themselves do not quite form a directed system.

% Theta monoidal
% first prove Phi is lax monoidal (straightforward)
% follows that Theta is colax monoidal
% prove comparison map is equivalence by checking on finite type n spectra
% everything is "nonunital" monoidal

\section{Lie algebras in $T(n)$-local spectra}
\label{sec:Liealgebras}

As with any adjunction, the adjoint pair $(\Theta, \Phi)$ gives a monad $\Phi\Theta$ on the $\infty$-category $\mathrm{Sp}_{T(n)}$. In other words, $\Phi\Theta$ has the structure of a monoid in the $\infty$-category of functors from $\mathrm{Sp}_{T(n)}$ to itself, of which the monoidal structure is given by composition of functors. A \emph{left module} (also called an \emph{algebra}) for this monad is a $T(n)$-local spectrum $X$ equipped with a map $\Phi\Theta(X) \rightarrow X$, and homotopies expressing the coherent associativity and unitality of this action. We write $\mathrm{LMod}_{\Phi\Theta}(\mathrm{Sp}_{T(n)})$ for the $\infty$-category of such left modules. The reader can consult Section 4.7 of \cite{higheralgebra} for a detailed treatment of monads in the $\infty$-categorical setting or \cite{riehlverity} for a different perspective.

The Bousfield--Kuhn functor factors through this $\infty$-category of algebras to give a functor
\begin{equation*}
\phi\colon \mathcal{S}_{v_n} \rightarrow \mathrm{LMod}_{\Phi\Theta}(\mathrm{Sp}_{T(n)}).
\end{equation*}
In joint work with Eldred, Mathew, and Meier \cite{ehmm} we prove the following theorem. We include a sketch of the proof for the reader's convenience.

\begin{theorem}
\label{thm:Phimonadic}
The functor $\phi$ is an equivalence of $\infty$-categories. In other words, the adjoint pair $(\Theta, \Phi)$ is monadic.
\end{theorem}
\begin{proof}[Sketch of proof]
We verify the following two facts:
\begin{itemize}
\item[(1)] The functor $\Phi$ is conservative, i.e., a map $\varphi$ in $\mathcal{S}_{v_n}$ is an equivalence if and only if $\Phi(\varphi)$ is an equivalence. This follows from Theorem \ref{thm:Mnf}(i).
\item[(2)] The functor $\Phi$ preserves geometric realizations (i.e., colimits of diagrams indexed by $\mathbf{\Delta}^{\mathrm{op}}$). We prove this below.
\end{itemize}
It follows that $\Phi$ satisfies the conditions of Lurie's version of the Barr-Beck theorem (cf. Theorem 4.7.4.5 of \cite{higheralgebra} or Theorem 7.2.4 of \cite{riehlverity}), which proves the theorem. To check (2), note that the same argument as in the proof of Proposition \ref{prop:Phifilteredcolimits} implies that it suffices to check that
\begin{equation*}
\Phi_v\colon \mathcal{S}_{v_n} \rightarrow \mathrm{Sp}_{T(n)}
\end{equation*}
preserves geometric realizations, where $\Phi_v$ is the telescopic functor associated to a $v_n$ self-map $v\colon \Sigma^d V_n \rightarrow V_n$. Recall from the beginning of Section \ref{sec:Mnf} that the construction of $\mathcal{S}_{v_n}$ involves choosing a finite type $n+1$ space $V_{n+1}$. For the purposes of this proof we choose this space so that its connectivity $d_{n+1}$ is larger than the dimension of $V_n$ (note that the statement of the theorem does not depend on this choice). Consider a diagram
\begin{equation*}
F\colon \mathbf{\Delta}^{\mathrm{op}} \rightarrow \mathcal{S}_{v_n}.
\end{equation*}
The inclusion $\mathcal{S}_{v_n} \rightarrow \mathcal{L}_n^f$ preserves colimits, so the colimit of $F$ may be computed in the latter $\infty$-category. Since $\mathcal{L}_n^f$ is a localization of $\mathcal{S}_*\langle d_{n+1} \rangle$ we have
\begin{equation*}
\varinjlim_{\mathbf{\Delta}^{\mathrm{op}}} F \simeq L_n^f \varinjlim_{\mathbf{\Delta}^{\mathrm{op}}} (i \circ F),
\end{equation*}
where the colimit on the right is computed in $\mathcal{S}_*\langle d_{n+1} \rangle$, or equivalently just in $\mathcal{S}_*$. Since $L_n^f$ preserves $v_n$-periodic homotopy groups it follows that there is a natural equivalence
\begin{equation*}
\Phi_v(\varinjlim_{\mathbf{\Delta}^{\mathrm{op}}} F) \simeq \Phi_v(\varinjlim_{\mathbf{\Delta}^{\mathrm{op}}} (i \circ F)).
\end{equation*}
Thus it suffices to show that the functor
\begin{equation*}
\Phi_v\colon \mathcal{S}_*\langle d_{n+1} \rangle \rightarrow \mathrm{Sp}_{T(n)}
\end{equation*}
preserves geometric realizations. This functor can be expressed as the following colimit:
\begin{equation*}
\Phi_v \simeq \varinjlim\bigl(\Sigma^\infty \mathrm{Map}_*(V,-) \rightarrow \Sigma^{\infty-d} \mathrm{Map}_*(V,-) \rightarrow  \Sigma^{\infty-2d} \mathrm{Map}_*(V,-)  \rightarrow \cdots \bigr).
\end{equation*}
But each of the functors
\begin{equation*}
\mathrm{Map}_*(V,-)\colon \mathcal{S}_*\langle d_{n+1} \rangle \rightarrow \mathcal{S}_*
\end{equation*}
preserves geometric realizations, since the connectivity $d_{n+1}$ of the spaces involved exceeds the dimension of $V$ (see Proposition 4.2 of \cite{ehmm}). This finishes the proof.
\end{proof}

We record the following aspect of this proof for future reference:

\begin{lemma}
\label{lem:Phisiftedcolims}
The Bousfield--Kuhn functor $\Phi\colon \mathcal{S}_{v_n} \rightarrow \mathrm{Sp}_{T(n)}$ preserves sifted colimits.
\end{lemma}
\begin{proof}
By Corollary 5.5.8.17 of \cite{htt} it suffices to show that $\Phi$ preserves filtered colimits and geometric realizations. The first is Proposition \ref{prop:Phifilteredcolimits}, the second was part of the proof of Theorem \ref{thm:Phimonadic} above (which first appeared as Proposition 4.1 in \cite{ehmm}).
\end{proof}

The purpose of this section is to prove Theorem \ref{thm:MnfLiealg}, stating that $\mathcal{S}_{v_n}$ is equivalent to the $\infty$-category of Lie algebras in $\mathrm{Sp}_{T(n)}$. Given Theorem \ref{thm:Phimonadic} above this amounts to understanding the monad $\Phi\Theta$. The algebraic and coalgebraic structures we consider in this section are conveniently described using operads and cooperads. We now give a brief informal review of these, as well as of bar-cobar (or Koszul) duality between the two. Then we outline our strategy for the remainder of this section.

For $\mathbf{C}$ a symmetric monoidal category, a \emph{symmetric sequence} in $\mathbf{C}$ is a collection $\{\mathcal{O}(k)\}_{k \geq 1}$ of objects of $\mathbf{C}$, where $\mathcal{O}(k)$ is equipped with an action of the symmetric group $\Sigma_k$. If $\mathbf{C}$ has sufficiently many colimits, such a symmetric sequence determines a functor
\begin{equation*}
F_{\mathcal{O}}\colon \mathbf{C} \rightarrow \mathbf{C}\colon X \mapsto \coprod_{k \geq 1} (\mathcal{O}(k) \otimes X^{\otimes k})_{\Sigma_k}.
\end{equation*}
The category of functors $\mathrm{Fun}(\mathbf{C}, \mathbf{C})$ is monoidal, with tensor product given by composition of functors. The category $\mathrm{SymSeq}(\mathbf{C})$ of symmetric sequences in $\mathbf{C}$ carries a corresponding monoidal structure $\circ$, called the \emph{composition product} of symmetric sequences, which is essentially determined by the requirement that the assignment
\begin{equation*}
\mathcal{O} \rightarrow F_{\mathcal{O}}
\end{equation*}
is monoidal, so that
\begin{equation*}
F_{\mathcal{O}' \circ \mathcal{O}} \simeq F_{\mathcal{O}'} \circ F_{\mathcal{O}}.
\end{equation*}
An operad in $\mathbf{C}$ is a monoid in the category of symmetric sequences (with respect to this composition product). Thus, an operad $\mathcal{O}$ gives rise to a monoid $F_{\mathcal{O}}$ in $\mathrm{Fun}(\mathbf{C}, \mathbf{C})$ or, in other words, a monad on $\mathbf{C}$. Similarly, a cooperad is a comonoid with respect the composition product, and a cooperad gives rise to a comonad on $\mathbf{C}$. The reader should note that the operads and cooperads we consider here are nonunital, in the sense that there is no $\mathcal{O}(0)$ term.

An algebra over an operad $\mathcal{O}$ is the same thing as an algebra for the monad $F_{\mathcal{O}}$, i.e., an object $X \in \mathbf{C}$ equipped with a structure map $F_{\mathcal{O}}(X) \rightarrow X$ satisfying the usual axioms. The evident dual of this definition does not give the usual notion of coalgebra for a cooperad; indeed, for a cooperad $\mathcal{C}$, a coalgebra for the comonad $F_{\mathcal{C}}$ is sometimes referred to as a \emph{conilpotent divided power} coalgebra (e.g. in \cite{francisgaitsgory}). This issue will not play a role for us here.

Algebras and coalgebras in differential graded categories are related by bar-cobar duality, or Koszul duality, and this duality was extended to operads in influential papers of Ginzburg and Kapranov \cite{ginzburgkapranov} and subsequently Getzler and Jones \cite{getzlerjones}. In particular, they identified the cobar construction of the commutative cooperad as (a degree shift of) the Lie operad. This bar-cobar duality was extended to the setting of stable homotopy theory by Ching \cite{ching,chingbar}, who works with operads and cooperads in the category of spectra. His motivating example is an operad whose underlying symmetric sequence of spectra $\{\partial_k \mathrm{id}\}_{k \geq 1}$ is given by the \emph{Goodwillie derivatives of the identity functor}. This operad is the cobar construction on the derivatives of the functor $\Sigma^\infty\Omega^\infty$, which form a cooperad essentially because $\Sigma^\infty\Omega^\infty$ is a comonad. This cooperad is easily identified as the commutative cooperad; consequently, the derivatives of the identity functor form an analogue of the (shifted) Lie operad in stable homotopy theory. (Taking integral homology of the spectra $\partial_k \mathrm{id}$ also reproduces a degree shift of the ordinary Lie operad in abelian groups.)

In Section \ref{subsec:Liealg} we discuss symmetric sequences, operads and cooperads in the $\infty$-category of $T(n)$-local spectra and define the $T(n)$-local Lie operad as the cobar construction of the $T(n)$-local commutative cooperad. We will use Lurie's version of bar-cobar duality between monoids and comonoids in an $\infty$-categorical setting.

In Section \ref{subsec:PhiTheta} we study the functor $\Phi\Theta$ and prove that it satisfies the formula
\begin{equation*}
\Phi\Theta(X) \simeq \bigoplus_{k \geq 1} L_{T(n)}(\partial_k \mathrm{id} \otimes X^{\otimes k})_{h\Sigma_k}
\end{equation*}
claimed in Section \ref{sec:mainresults}. In other words, $\Phi\Theta$ is the functor corresponding to the ($T(n)$-local) symmetric sequence of derivatives of the identity functor.

In Section \ref{subsec:PhiThetamonad} we study the monad structure of $\Phi\Theta$ and show it arises from the commutative cooperad by a cobar construction. In other words, the algebras for the monad $\Phi\Theta$ are precisely Lie algebras in $T(n)$-local spectra. We use this result to prove Theorem \ref{thm:MnfLiealg}.

Section \ref{subsec:applications} discusses Theorems \ref{thm:Phitypenspace} and \ref{thm:Sigmainftyfiltration} as applications of our results. Finally, Section \ref{subsec:Knlocal} establishes Corollary \ref{cor:Knlocal}, which is a $K(n)$-local version of Theorem \ref{thm:MnfLiealg}.

% symmetric sequences of spectra, operads, cooperads
% Ching's work on bar-cobar duality, Lie operad as dual of commutative cooperad
% Lie operad is given by derivatives of the identity

% outline proof of Theorem \ref{thm:MnfLiealg}
% reduce to proposition about equivalence \Phi\Theta = cobar(\Sigma^\infty\Omega^\infty)

% some exposition about vanishing of Tate spectra?

\subsection{Operads and cooperads of $T(n)$-local spectra}
\label{subsec:Liealg}

There are several approaches to the theory of operads in an $\infty$-categorical setting, for example the $\infty$-operads of Lurie \cite{higheralgebra} as well as the dendroidal sets of Moerwijk--Weiss \cite{moerdijkweiss} and Cisinski--Moerdijk \cite{cisinskimoerdijk1}. These all describe a theory of higher operads in spaces; the theory of $\infty$-operads in general symmetric monoidal $\infty$-categories was first considered in \cite{haugsengchu}. We will not need the general theory here. Rather, in this section we demonstrate that it is rather simple to develop a theory of operads in the symmetric monoidal $\infty$-category $\mathrm{Sp}_{T(n)}$ of $T(n)$-local spectra (with the $T(n)$-localized smash product as monoidal structure) using some facts from (dual) Goodwillie calculus. We take an approach using symmetric sequences, since our main examples arise in this way and because it allows for a smooth treatment of bar-cobar duality.

\begin{definition}
\label{def:symmetricsequences}
A functor $F\colon \mathrm{Sp}_{T(n)} \rightarrow \mathrm{Sp}_{T(n)}$ is \emph{coanalytic} if there is a natural equivalence
\begin{equation*}
F(X) \simeq \bigoplus_{k \geq 1} (\mathcal{O}(k) \otimes X^{\otimes k})_{h\Sigma_k}
\end{equation*}
for $\{\mathcal{O}(k)\}_{k \geq 1}$ a symmetric sequence in $\mathrm{Sp}_{T(n)}$. Write $\mathrm{coAn}(\mathrm{Sp}_{T(n)})$ for the full subcategory of $\mathrm{Fun}(\mathrm{Sp}_{T(n)}, \mathrm{Sp}_{T(n)})$ on the coanalytic functors.
\end{definition}

The term coanalytic is motivated by Theorem \ref{thm:recognition} below, which states that a functor is coanalytic if and only if its `dual Goodwillie tower' converges. (Beware that this is not precisely dual to Goodwillie's notion of analytic functor, which involves explicit connectivity estimates on the maps in the Goodwillie tower of a functor.) Observe that a smash product of coanalytic functors is again coanalytic. More importantly, the composition of coanalytic functors is coanalytic; therefore, the composition of functors restricts to give a monoidal structure on the $\infty$-category $\mathrm{coAn}(\mathrm{Sp}_{T(n)})$. This $\infty$-category will serve as the monoidal $\infty$-category of symmetric sequences of $T(n)$-local spectra. We will justify this in Proposition \ref{lem:coansymseq} below. For future reference we first state the following rather obvious fact, characterizing the coefficients of a symmetric sequence in terms of Goodwillie derivatives \cite[Section 5]{goodwillie3}, \cite[Section 6.3]{higheralgebra}. 
% A brief review of Goodwillie calculus can be found in Appendix \ref{app:Goodwillie}.

\begin{lemma}
\label{lem:coancoefficients}
Let $F \in \mathrm{coAn}(\mathrm{Sp}_{T(n)})$ and write $\partial_k F$ for the $k$th Goodwillie derivative of $F$. Then there is a natural equivalence
\begin{equation*}
F(X) \simeq \bigoplus_{k \geq 1} (\partial_k F \otimes X^{\otimes k})_{h\Sigma_k}.
\end{equation*}
\end{lemma}
\begin{proof}
Forming Goodwillie derivatives of functors between stable $\infty$-categories commutes with filtered colimits and direct sums. The conclusion follows from
\begin{equation*}
\partial_j  (\mathcal{O}(k) \otimes (-)^{\otimes k})_{h\Sigma_k} \simeq \begin{cases} \mathcal{O}(k) & \text{if } j=k, \\
0 & \text{otherwise}.
\end{cases}
\end{equation*}
\end{proof}

It will be useful to have a characterization of the class of coanalytic functors in terms of dual Goodwillie calculus (as introduced by McCarthy \cite{mccarthy}). We review the kind of dual calculus we have in mind in Appendix \ref{app:Goodwillie}. In particular, for a functor $F\colon \mathrm{Sp}_{T(n)} \rightarrow \mathrm{Sp}_{T(n)}$ preserving filtered colimits, its \emph{dual} $k$-excisive approximation is a natural transformation $P^k F \rightarrow F$ that is universal with respect to natural transformations from $k$-excisive functors (which also preserve filtered colimits) into $F$. These dual approximations assemble into a filtration
\[
\begin{tikzcd}
P^1F \arrow[r]\arrow[d] & P^2F \arrow[r]\arrow[dl] & P^3F \arrow[r]\arrow[dll] & \cdots \\
F
\end{tikzcd}
\]
which is formally dual to the usual Goodwillie tower. In particular, there is a natural map
\begin{equation*}
\varinjlim_k P^k F \rightarrow F.
\end{equation*}

We will apply the following recognition theorem in the next section to prove that $\Phi\Theta$ is coanalytic:

\begin{theorem}
\label{thm:recognition}
Let $F\colon \mathrm{Sp}_{T(n)} \rightarrow \mathrm{Sp}_{T(n)}$ be a reduced functor preserving filtered colimits. Then $F$ is coanalytic if and only if the map
\begin{equation*}
\varinjlim_k P^k F \rightarrow F
\end{equation*}
is an equivalence.
\end{theorem}

This result hinges on a nilpotence lemma of Mathew, which is contained in Appendix \ref{app:akhil}. Although Goodwillie's $k$-excisive approximation $P_k$ commutes with filtered colimits of functors, this is generally \emph{not} the case for the dual approximation $P^k$. Nonetheless, in the $T(n)$-local setting the following holds (cf. Lemma \ref{lem:mathew}):

\begin{lemma}
\label{lem:nilpotence}
If
\begin{equation*}
F \simeq \bigoplus_{j=1}^\infty F_j,
\end{equation*}
with $F_j\colon \mathrm{Sp}_{T(n)} \rightarrow \mathrm{Sp}_{T(n)}$ a $j$-homogeneous functor, then the natural map
\begin{equation*}
\bigoplus_{j=1}^k F_j \rightarrow P^k F
\end{equation*}
is an equivalence.
\end{lemma}

\begin{proof}[Proof of Theorem \ref{thm:recognition}]
Suppose $F$ is coanalytic. Then Lemma \ref{lem:nilpotence} clearly implies that $\varinjlim_k P^kF \rightarrow F$ is an equivalence. Conversely, suppose that this natural transformation is an equivalence. Kuhn \cite{kuhntate} shows that every $k$-excisive functor from $\mathrm{Sp}_{T(n)}$ to itself has a \emph{split} Goodwillie tower, meaning it is the direct sum of homogeneous functors:
\begin{equation*}
P^k F \simeq \bigoplus_{j=1}^k D_j P^k F.
\end{equation*}
He proves this by exploiting work of McCarthy \cite{mccarthy}, which classifies the Goodwillie towers of functors from $\mathrm{Sp}$ to itself in terms of certain Tate spectra associated to the symmetric groups. Kuhn shows that $T(n)$-locally such Tate spectra vanish (an alternative proof of the same fact is given in \cite{mathewclausen}). Taking the colimit over $k$ we conclude that $F$ is of the form
\begin{equation*}
F(X) \simeq \bigoplus_{k \geq 1} (\mathcal{O}(k) \otimes X^{\otimes k})_{h\Sigma_k}
\end{equation*}
for some symmetric sequence $\mathcal{O}$, so $F$ is coanalytic. Of course it also follows that $\mathcal{O}(k) = \partial_k F$.
\end{proof}

Note that the preceding argument in fact shows that for any reduced functor $F\colon \mathrm{Sp}_{T(n)} \rightarrow \mathrm{Sp}_{T(n)}$ preserving filtered colimits, the functor $\varinjlim_k P^k F$ is coanalytic. Generally, for functors $F$ from the $\infty$-category $\mathrm{Sp}$ to itself, the assignments $F \mapsto \varprojlim_k P_k F$ and $F \mapsto \varinjlim_k P^k F$ are not localizations or colocalizations; indeed, they need not even be idempotent. However, Theorem \ref{thm:recognition} and Lemma \ref{lem:nilpotence} show that the situation is much better in the $T(n)$-local case:

\begin{corollary}
\label{cor:recognition}
Write $\mathrm{Fun}_*^\omega(\mathrm{Sp}_{T(n)}, \mathrm{Sp}_{T(n)})$ for the $\infty$-category of reduced functors from $\mathrm{Sp}_{T(n)}$ to itself which preserve filtered colimits. Then the inclusion  
\begin{equation*}
\mathrm{coAn}(\mathrm{Sp}_{T(n)}) \rightarrow \mathrm{Fun}_*^\omega(\mathrm{Sp}_{T(n)}, \mathrm{Sp}_{T(n)})
\end{equation*}
admits a right adjoint $F \mapsto \varinjlim_k P^k F$. In particular, the inclusion preserves all colimits.
\end{corollary}
\begin{proof}
It suffices to prove that for any $F \in \mathrm{Fun}_*^\omega(\mathrm{Sp}_{T(n)}, \mathrm{Sp}_{T(n)})$ and any coanalytic functor $G$, the natural map
\begin{equation*}
\mathrm{Nat}(G, \varinjlim_k P^k F) \rightarrow \mathrm{Nat}(G, F)
\end{equation*}
is an equivalence. Since $G \simeq \varinjlim_j P^j G$, it suffices to prove the claim for every $P^j G$, i.e., we may reduce to the case where $G$ is $j$-excisive for some $j \geq 0$. In the diagram
\[
\begin{tikzcd}
\mathrm{Nat}(P^j G, \varinjlim_k P^k F) \ar{r} & \mathrm{Nat}(P^j G, F) \\
\mathrm{Nat}(P^j G, P^j(\varinjlim_k P^k F)) \ar{r}\ar{u} &  \mathrm{Nat}(P^j G, P^j F) \ar{u},
\end{tikzcd}
\]
the two vertical maps are equivalences by the universal property of $P^j$. Moreover, the lower horizontal map is an equivalence by Lemma \ref{lem:nilpotence}. Therefore the upper horizontal map is an equivalence as well.

% 5.2.7.4
\end{proof}

\begin{proposition}
\label{lem:coansymseq}
The $\infty$-category $\mathrm{coAn}(\mathrm{Sp}_{T(n)})$ of coanalytic functors is equivalent to the $\infty$-category $\mathrm{SymSeq}(\mathrm{Sp}_{T(n)})$ of symmetric sequences of $T(n)$-local spectra.
\end{proposition}
\begin{proof}
Consider the functor
\begin{equation*}
S\colon \mathrm{SymSeq}(\mathrm{Sp}_{T(n)}) \rightarrow \mathrm{coAn}(\mathrm{Sp}_{T(n)})
\end{equation*}
which assigns to a symmetric sequence $\{\mathcal{O}(k)\}_{k \geq 1}$ the functor given by the formula of Definition \ref{def:symmetricsequences}. Then $S$ is essentially surjective by definition. To see it is fully faithful, consider coanalytic functors $F$ and $G$, corresponding to symmetric sequences $\mathcal{O}$ and $\mathcal{P}$ respectively. We can analyze the space of natural transformations from $F$ to $G$ by first observing that
\begin{eqnarray*}
\mathrm{Nat}(F, G) & \simeq & \prod_{k \geq 1} \mathrm{Nat}\bigl((\mathcal{O}(k) \otimes (-)^{\otimes k})_{h\Sigma_k}, G\bigr) \\
& \simeq &  \prod_{k \geq 1}  \mathrm{Nat}\bigl((\mathcal{O}(k) \otimes (-)^{\otimes k})_{h\Sigma_k}, P^k G\bigr).
\end{eqnarray*}
Since $G$ is coanalytic, Lemma \ref{lem:nilpotence} applies and
\begin{equation*}
P^k G(X) \simeq \oplus_{j=1}^k (\mathcal{P}(j) \otimes X^{\otimes j})_{h\Sigma_j}.
\end{equation*}
The space of natural transformations between homogeneous functors to $\mathrm{Sp}_{T(n)}$ of different degrees is trivial. Indeed, if $K$ and $L$ are homogeneous of degrees $k$ and $l$ respectively and $k > l$, then 
\begin{equation*}
\mathrm{Nat}(K,L) \simeq \mathrm{Nat}(P_l K, L) \simeq *
\end{equation*}
since $P_l K$ vanishes. On the other hand, if $k < l$, we observe
\begin{equation*}
\mathrm{Nat}(K,L) \simeq \mathrm{Nat}(K, P^k L) \simeq *.
\end{equation*}
Here we have used that $L$ is also $l$-cohomogeneous (i.e. of the form $(\partial_l L \otimes (-)^{\otimes l})^{h\Sigma_l}$), which is another application of the $T(n)$-local vanishing of Tate spectra \cite{kuhntate}.

We conclude
\begin{equation*}
\mathrm{Nat}(F, G)  \simeq \prod_{k \geq 1} \mathrm{Nat}\bigl((\mathcal{O}(k) \otimes (-)^{\otimes k})_{h\Sigma_k}, (\mathcal{P}(k) \otimes (-)^{\otimes k})_{h\Sigma_k}\bigr).
\end{equation*}
But the space of natural transformations between homogeneous degree $k$ functors is equivalent to the space of $\Sigma_k$-equivariant maps between their coefficients, by Goodwillie's classification of homogeneous functors (cf. Theorem 3.5 of \cite{goodwillie3}). Consequently, $\mathrm{Nat}(F,G)$ is equivalent to the space of maps between the symmetric sequences $\mathcal{O}$ and $\mathcal{P}$, from which the lemma easily follows.
\end{proof}

\begin{remark}
It is useful to note that any coanalytic functor $F$ preserves sifted colimits. This is a straightforward consequence of the fact that the functor $X \rightarrow X^{\otimes k}$ preserves sifted colimits for any $k \geq 1$. In the next section, specifically Proposition \ref{prop:siftedcolimitscoanalytic}, we will show that the converse is true as well. It is rather striking that coanalytic functors on $\mathrm{Sp}_{T(n)}$ can be characterized in this way; the corresponding statement does not hold without $T(n)$-localization.
\end{remark}

\begin{definition}
\label{def:operad}
An \emph{operad in $T(n)$-local spectra} is an associative algebra object of $\mathrm{coAn}(\mathrm{Sp}_{T(n)})$. Likewise, a \emph{cooperad in $T(n)$-local spectra} is an associative coalgebra object of $\mathrm{coAn}(\mathrm{Sp}_{T(n)})$ or, equivalently, an associative algebra object of the opposite $\infty$-category $\mathrm{coAn}(\mathrm{Sp}_{T(n)})^{\mathrm{op}}$.
\end{definition}

Thus an operad in $T(n)$-local spectra is a monad on $\mathrm{Sp}_{T(n)}$ whose underlying functor is coanalytic. Therefore it also makes sense to speak of \emph{algebras} for such an operad, by using the definition of algebras for a monad.

We conclude this section with a brief review of Koszul duality for such operads and cooperads, which is a special case of bar-cobar duality between associative monoids and comonoids. The relevant results for us are contained in Section 5.2 of \cite{higheralgebra}.

Note that the identity functor of $\mathrm{Sp}_{T(n)}$, which is the unit object of the monoidal $\infty$-category $\mathrm{coAn}(\mathrm{Sp}_{T(n)})$, is both an operad and a cooperad in $T(n)$-local spectra in an essentially unique way; we will denote both by $\mathbf{1}$. An \emph{augmentation} of an operad $\mathcal{O}$ in $T(n)$-local spectra is a morphism of operads $\mathcal{O} \rightarrow \mathbf{1}$. Dually, a \emph{coaugmentation} of a cooperad $\mathcal{C}$ in $T(n)$-local spectra is a morphism of cooperads $\mathbf{1} \rightarrow \mathcal{C}$. We write $\mathrm{Op}^{\mathrm{aug}}(\mathrm{Sp}_{T(n)})$ and $\mathrm{coOp}^{\mathrm{aug}}(\mathrm{Sp}_{T(n)})$ for the $\infty$-categories of augmented operads and coaugmented cooperads in $T(n)$-local spectra respectively. By Theorem 5.2.2.17 and the subsequent Remark 5.2.2.19 of \cite{higheralgebra}, these $\infty$-categories are related by an adjunction
\[
\begin{tikzcd}
\mathrm{Op}^{\mathrm{aug}}(\mathrm{Sp}_{T(n)}) \ar[shift left]{r}{\mathrm{Bar}} & \mathrm{coOp}^{\mathrm{aug}}(\mathrm{Sp}_{T(n)}) \ar[shift left]{l}{\mathrm{Cobar}}.
\end{tikzcd}
\]
If $\varepsilon\colon \mathcal{O} \rightarrow \mathbf{1}$ is an augmented operad, the underlying coanalytic functor of the bar construction $\mathrm{Bar}(\mathcal{O})$ can be described as the geometric realization of a simplicial object as follows:
\[
\begin{tikzcd}
\cdots \ar[shift right = .18cm]{r}\ar[shift right = .06cm]{r}\ar[shift left = .06cm]{r}\ar[shift left = .18cm]{r}  & \mathcal{O} \circ \mathcal{O} \ar[shift left = .12cm]{r} \ar{r} \ar[shift right = .12cm]{r} & \mathcal{O} \ar[shift left = .06cm]{r} \ar[shift right = .06cm]{r} & \mathbf{1}.
\end{tikzcd}
\]
Here the simplicial face maps are formed using the augmentation and the multiplication of $\mathcal{O}$, whereas the degeneracies use the unit $\mathbf{1} \rightarrow \mathcal{O}$. In different words, regarding $\mathbf{1}$ as both a left and a right module over $\mathcal{O}$ via the augmentation $\varepsilon$, we can view $\mathrm{Bar}(\mathcal{O})$ as the geometric realization of the two-sided bar construction:
\begin{equation*}
\mathrm{Bar}(\mathcal{O}) = |\mathrm{Bar}(\mathbf{1}, \mathcal{O}, \mathbf{1})_{\bullet}|.
\end{equation*}
Dually, for a coaugmented cooperad $\mathbf{1} \rightarrow \mathcal{C}$, the underlying coanalytic functor of the cobar construction $\mathrm{Cobar}(\mathcal{C})$ is the totalization of a cosimplicial object as follows:
\[
\begin{tikzcd}
\mathbf{1} \ar[shift left = .06cm]{r} \ar[shift right = .06cm]{r} & \mathcal{C} \ar[shift left = .12cm]{r} \ar{r} \ar[shift right = .12cm]{r} & \mathcal{C} \circ \mathcal{C} \ar[shift right = .18cm]{r}\ar[shift right = .06cm]{r}\ar[shift left = .06cm]{r}\ar[shift left = .18cm]{r} & \cdots.
\end{tikzcd}
\]
Lurie establishes the bar-cobar adjunction by considering augmented associative algebras in the \emph{twisted arrow category}. For an $\infty$-category $\mathcal{E}$, the twisted arrow category $\mathrm{TwArr}(\mathcal{E})$ is the $\infty$-category whose $n$-simplices are maps
\begin{equation*}
\Delta^n \star (\Delta^n)^{\mathrm{op}} \rightarrow \mathcal{E}.
\end{equation*}
Restricting to the first and second part of this join gives a map
\begin{equation*}
\mathrm{TwArr}(\mathcal{E}) \rightarrow \mathcal{E} \times \mathcal{E}^{\mathrm{op}}.
\end{equation*}
Moreover if $\mathcal{E}$ is monoidal, then $\mathcal{E}^{\mathrm{op}}$ and $\mathrm{TwArr}(\mathcal{E})$ inherit monoidal structures in a natural way and the map above can be made a monoidal functor. In particular, an associative algebra object of $\mathrm{TwArr}(\mathcal{E})$ then projects to an associative algebra object of $\mathcal{E}$ and an associative algebra object of $\mathcal{E}^{\mathrm{op}}$. Lurie proves that maps $\mathrm{Bar}(\mathcal{O}) \rightarrow \mathcal{C}$ (or equivalently maps $\mathcal{O} \rightarrow \mathrm{Cobar}(\mathcal{C})$) are classified by lifts of the pair $(\mathcal{O}, \mathcal{C})$ along the functor above to an augmented algebra object of the monoidal $\infty$-category $\mathrm{TwArr}(\mathcal{E})$. More precisely his results give natural equivalences of spaces
\begin{eqnarray*}
\mathrm{Map}_{\mathrm{Alg}^{\mathrm{aug}}(\mathcal{E})}(\mathcal{O}, \mathrm{Cobar}(\mathcal{C})) & \simeq  & \{\mathcal{O}\} \times_{\mathrm{Alg}^{\mathrm{aug}}(\mathcal{E})} \mathrm{Alg}^{\mathrm{aug}}(\mathrm{TwArr}(\mathcal{E})) \times_{\mathrm{Alg}^{\mathrm{aug}}(\mathcal{E}^{\mathrm{op}})} \{\mathcal{C}\} \\
& \simeq & \mathrm{Map}_{\mathrm{coAlg}^{\mathrm{aug}}(\mathcal{E})}(\mathrm{Bar}(\mathcal{O}), \mathcal{C}).
\end{eqnarray*}
If one has an object $\mathcal{O} \rightarrow \mathcal{C}$ of $\mathrm{TwArr}(\mathcal{E})$ with the structure of an augmented associative algebra, its simplicial bar resolution gives a diagram in $\mathcal{E}$ as follows:
\[
\begin{tikzcd}
\cdots \ar[shift right = .18cm]{r}\ar[shift right = .06cm]{r}\ar[shift left = .06cm]{r}\ar[shift left = .18cm]{r}  & \mathcal{O} \circ \mathcal{O} \ar[shift left = .12cm]{r} \ar{r} \ar[shift right = .12cm]{r} \ar{d} & \mathcal{O} \ar[shift left = .06cm]{r} \ar[shift right = .06cm]{r} \ar{d} & \mathbf{1} \ar[equal]{d} \\
\cdots & \mathcal{C} \circ \mathcal{C} \ar[shift right = .18cm]{l}\ar[shift right = .06cm]{l}\ar[shift left = .06cm]{l}\ar[shift left = .18cm]{l} & \mathcal{C} \ar[shift left = .12cm]{l} \ar{l} \ar[shift right = .12cm]{l} & \mathbf{1}. \ar[shift left = .06cm]{l} \ar[shift right = .06cm]{l} 
\end{tikzcd}
\]
The commutativity of this diagram indeed induces maps in $\mathcal{E}$ of the form $\mathrm{Bar}(\mathcal{O}) \rightarrow \mathcal{C}$ (by considering the top row) and $\mathcal{O} \rightarrow \mathrm{Cobar}(\mathcal{C})$ (by considering the bottom row). The content of Lurie's result is that these are in fact maps of associative (co)algebra objects and that every such map arises in this way.

The example of a cooperad in $T(n)$-local spectra that will concern us here is the \emph{commutative cooperad}. As a symmetric sequence, it is simply the localized sphere spectrum $L_{T(n)} \mathbb{S}$ in every degree. In fact, we will take our cue from the observation of Arone and Ching that the usual commutative cooperad in the category of spectra arises form the derivatives of the comonad $\Sigma^\infty \Omega^\infty$, see Section 15 of \cite{aroneching}. We will show in the next section (Theorem \ref{thm:SigmaOmegacoanalytic}) that there is a natural equivalence
\begin{equation*}
\Sigma^\infty_{T(n)}\Omega^\infty_{T(n)}(X) \simeq L_{T(n)}\bigoplus_{k=1}^\infty X^{\otimes k}_{h\Sigma_k},
\end{equation*}
so in particular $\Sigma^\infty_{T(n)}\Omega^\infty_{T(n)}$ is coanalytic. It is also a comonad and therefore an associative coalgebra object of $\mathrm{coAn}(\mathrm{Sp}_{T(n)})$.

\begin{definition}
\label{def:commcooperad}
The \emph{commutative cooperad in $T(n)$-local spectra} is the comonad $\Sigma^\infty_{T(n)}\Omega^\infty_{T(n)}$.
\end{definition}

Ching \cite{ching} constructs an operad structure on the derivatives of the identity functor by observing that they arise from the commutative cooperad by the cobar construction. The resulting operad plays the role of the (shifted) Lie operad in stable homotopy theory. We therefore define the following:

\begin{definition}
\label{def:Lieoperad}
The \emph{Lie operad in $T(n)$-local spectra} is the cobar construction of the commutative cooperad.
\end{definition}

% note: every operad with O(1) = id is canonically augmented, simply using the projection O -> P_1O. This applies to \Phi\Theta. Dually for \Sigma^\infty\Omega^\infty.

\subsection{The functor $\Phi\Theta$}
\label{subsec:PhiTheta}

The goal of this section is to analyze the functor $\Phi\Theta$ and prove Theorem \ref{thm:PhiThetafunctor} below. In the next section we investigate the monad structure of $\Phi\Theta$ and relate it to Lie algebras in $T(n)$-local spectra.

\begin{theorem}
\label{thm:PhiThetafunctor}
There is a natural equivalence as follows:
\begin{equation*}
\Phi\Theta(X) \simeq L_{T(n)}\bigoplus_{k \geq 1} (\partial_k\mathrm{id} \otimes X^{\otimes k})_{h\Sigma_k}.
\end{equation*}
Here $\partial_k\mathrm{id}$ denotes the $k$th derivative of the identity functor on $\mathcal{S}_*$.
\end{theorem}

The theorem is a consequence of the following two results and Lemma \ref{lem:coancoefficients}:

\begin{theorem}
\label{thm:PhiThetaderivatives}
The $k$th derivative of $\Phi\Theta$ is equivalent to the spectrum $L_{T(n)}\partial_k\mathrm{id}$, where $\partial_k\mathrm{id}$ is the $k$th derivative of the identity functor of $\mathcal{S}_*$.
\end{theorem}

\begin{theorem}
\label{thm:PhiThetacoanalytic}
The functor $\Phi\Theta$ is coanalytic.
\end{theorem}

\begin{proof}[Proof of Theorem \ref{thm:PhiThetaderivatives}]
Since $\Theta$ preserves colimits and $\Phi$ preserves finite limits and filtered colimits, we have 
\begin{equation*}
P_k(\Phi\Theta)(X) \simeq \Phi P_k\mathrm{id}_{\mathcal{S}_{v_n}}(\Theta X).
\end{equation*}
The Goodwillie tower of the identity functor of $\mathcal{S}_{v_n}$ is easily determined. Recall that we write $\iota\colon \mathcal{S}_{v_n} \rightarrow \mathcal{L}_n^f$ for the inclusion and $r$ for its right adjoint. The latter preserves filtered colimits, so we find
\begin{equation*}
P_k\mathrm{id}_{\mathcal{S}_{v_n}} \simeq P_k(r \circ \iota) \simeq r P_k(\mathrm{id}_{\mathcal{L}_n^f}) \iota.
\end{equation*}
The localization $L_n^f$ preserves finite limits, so that
\begin{equation*}
P_k\mathrm{id}_{\mathcal{L}_n^f}(X) \simeq L_n^f P_k\mathrm{id}(X)
\end{equation*}
for $X \in \mathcal{L}_n^f$. Putting these together gives
\begin{equation*}
P_k\mathrm{id}_{\mathcal{S}_{v_n}} \simeq M \circ P_k\mathrm{id} \circ i.
\end{equation*}
It follows that
\begin{equation*}
D_k \mathrm{id}_{\mathcal{S}_{v_n}}(X) \simeq M D_k\mathrm{id}(iX) \simeq M\Omega^\infty(\partial_k \mathrm{id} \otimes \Sigma^\infty X^{\otimes k})_{h\Sigma_k}.
\end{equation*}
Note that $M\Omega^\infty \simeq \Omega^\infty_{T(n)}$ and the expression in parentheses is $T(n)$-equivalent to the smash product of $T(n)$-local spectra
\begin{equation*}
L_{T(n)}\partial_k \mathrm{id} \otimes \Sigma_{T(n)}^\infty X^{\otimes k}.
\end{equation*}
We conclude that
\begin{equation*}
D_k \mathrm{id}_{\mathcal{S}_{v_n}}(X) \simeq \Omega^\infty_{T(n)}(L_{T(n)}\partial_k \mathrm{id} \otimes \Sigma_{T(n)}^\infty X^{\otimes k})_{h\Sigma_k}.
\end{equation*}
Comparing coefficients gives $\partial_k \mathrm{id}_{\mathcal{S}_{v_n}} \simeq L_{T(n)}\partial_k\mathrm{id}$.
\end{proof}

Our original proof of Theorem \ref{thm:PhiThetacoanalytic} is rather involved and consists of the following steps:
\begin{itemize}
\item[(i)] When evaluated on any shift $L_{T(n)}\mathbb{S}^\ell$ of the $T(n)$-local sphere spectrum, the natural transformation $\varinjlim_k P^k(\Phi\Theta) \rightarrow \Phi\Theta$ of Theorem \ref{thm:recognition} is an equivalence. This gives the formula
\begin{equation*}
\Phi\Theta(L_{T(n)}\mathbb{S}^\ell) \simeq \bigoplus_{k=1}^{2p^n} L_{T(n)}(\partial_k\mathrm{id} \otimes (\mathbb{S}^\ell)^{\otimes k})_{h\Sigma_k}.
\end{equation*}
We will come back to this formula in Corollary \ref{cor:aronemahowald}. This first step is the most delicate. It uses the results of Arone-Mahowald \cite{aronemahowald} on the convergence of the $v_n$-periodic Goodwillie tower of spheres, combined with an explicit analysis of the effect of the Bousfield--Kuhn functor on the space $\Theta(L_{T(n)}\mathbb{S}^\ell)$ of which we gave a description in Section \ref{subsec:BKleftadjoint}.
\item[(ii)] The natural transformation $\varinjlim_k P^k(\Phi\Theta) \rightarrow \Phi\Theta$ is also an equivalence when evaluated on a finite sum
\begin{equation*}
X = L_{T(n)}(\mathbb{S}^{\ell_1} \oplus \cdots \oplus \mathbb{S}^{\ell_m})
\end{equation*}
of copies of shifts of the $T(n)$-local sphere spectrum. This statement can be deduced from (i) by using the Hilton-Milnor theorem and its interaction with the Goodwillie tower (as in Theorem 2.4 of \cite{brantnerheuts}) to decompose the spectra $P^k\Phi\Theta(X)$ into summands featuring only $\Phi\Theta$ evaluated on a single copy of a (shifted) sphere spectrum.
\item[(iii)] To conclude that $\varinjlim_k P^k(\Phi\Theta) \rightarrow \Phi\Theta$ is an equivalence in general, one uses that both sides commute with sifted colimits. For the domain this is clear since it is a coanalytic functor; for the codomain this follows because $\Theta$ commutes with all colimits and $\Phi$ with sifted colimits (cf. Lemma \ref{lem:Phisiftedcolims}). A general $T(n)$-local spectrum can be written as a sifted colimit of finite sums of copies of the sphere spectrum, so that one reduces the general case to (ii).
\end{itemize}

The proof of Theorem \ref{thm:PhiThetacoanalytic} we will present here proceeds differently. A corollary of the coanalyticity of $\Phi\Theta$ is that the functor $\Sigma^\infty_{T(n)}\Omega^\infty_{T(n)}$ is coanalytic as well (essentially only using that the latter is the bar construction of the former). This result can be proved directly though, following work of Kuhn \cite{kuhnAQ}. We shall take this fact as our starting point for a proof of Theorem \ref{thm:PhiThetacoanalytic}.

\begin{theorem}
\label{thm:SigmaOmegacoanalytic}
The functor $\Sigma^\infty_{T(n)}\Omega^\infty_{T(n)}$ is coanalytic. More precisely, there is a natural equivalence
\begin{equation*}
\Sigma^\infty_{T(n)}\Omega^\infty_{T(n)} X \simeq L_{T(n)}\bigoplus_{k=1}^\infty (X^{\otimes k})_{h\Sigma_k}.
\end{equation*}
\end{theorem}
\begin{proof}
For a spectrum $E$, apply $\Phi$ to the unit map $\Omega^\infty E \xrightarrow{\Omega^\infty\eta} \Omega^\infty\Sigma^\infty\Omega^\infty E$ and use that $\Phi\Omega^\infty$ is equivalent to the localization functor $L_{T(n)}$ to find a natural map
\begin{equation*}
\lambda\colon E \rightarrow L_{T(n)}\Sigma^\infty\Omega^\infty E.
\end{equation*}
The codomain of $\lambda$ is a nonunital commutative $T(n)$-local ring spectrum (since $\Omega^\infty E$ is an $\mathbf{E}_\infty$-space), so that $\lambda$ extends to a map of nonunital $T(n)$-local ring spectra
\begin{equation*}
L_{T(n)}\mathrm{Sym}^*(E) = L_{T(n)}\bigoplus_{k=1}^\infty (E^{\otimes k})_{h\Sigma_k} \rightarrow L_{T(n)}\Sigma^\infty\Omega^\infty E.
\end{equation*}
Theorems 2.5, 2.10, and 2.12 of \cite{kuhnAQ} imply that this map is an equivalence whenever $E$ is $d_{n+1}$-connected and $T(i)_*E = 0$ for $0< i < n$. Now consider a $T(n)$-local spectrum $X$. Recall from Remark \ref{rmk:MnfLTn} that the functor $\Omega^\infty_{T(n)}$ may be described as the composition of the equivalence $M_n^f\colon L_{T(n)}\mathrm{Sp} \rightarrow M_n^f\mathrm{Sp}$ and the functor $\Omega^\infty \langle d_{n+1} \rangle\colon M_n^f\mathrm{Sp} \rightarrow \mathcal{S}_{v_n}$. Hence we can write
\begin{equation*}
\Sigma^\infty_{T(n)} \Omega^\infty_{T(n)} X = L_{T(n)}\Sigma^\infty\Omega^\infty((M_n^f X)\langle d_{n+1} \rangle).
\end{equation*}
The spectrum $E = (M_n^f X)\langle d_{n+1} \rangle$ is $d_{n+1}$-connected and $T(i)$-acyclic for $i < n$, so that we find a natural equivalence
\begin{eqnarray*}
L_{T(n)}\Sigma^\infty \Omega^\infty_{T(n)} X \simeq L_{T(n)}\bigoplus_{k=1}^\infty (E^{\otimes k})_{h\Sigma_k}.
\end{eqnarray*}
This finishes the proof, using that $L_{T(n)} E \simeq L_{T(n)} X$.
\end{proof}

\begin{remark}
It is tempting to think of Theorem \ref{thm:SigmaOmegacoanalytic} as an analog of the Snaith splitting
\begin{equation*}
\Sigma^\infty\Omega^\infty\Sigma^\infty X \simeq \bigoplus_{k=1}^\infty (\Sigma^\infty X^{\otimes k})_{h\Sigma_k}
\end{equation*}
for a pointed space $X$, but this line of thought can be deceiving. Indeed, the latter splitting arises from the extension of the natural map
\begin{equation*}
\Sigma^\infty\eta\colon \Sigma^\infty X \rightarrow \Sigma^\infty\Omega^\infty\Sigma^\infty X
\end{equation*}
to a map of commutative ring spectra. The $T(n)$-localization of this map is in general not homotopic to the map 
\begin{equation*}
\lambda\colon L_{T(n)}\Sigma^\infty X \rightarrow L_{T(n)}\Sigma^\infty\Omega^\infty\Sigma^\infty X
\end{equation*}
used in the proof above; we refer the reader to Kuhn's paper \cite{kuhnAQ} for further discussion. The difference between the two maps can be expressed in terms of Rezk's logarithm \cite{rezk}. 
\end{remark}

Theorem \ref{thm:SigmaOmegacoanalytic} turns out to have the following powerful consequence:

\begin{proposition}
\label{prop:siftedcolimitscoanalytic}
Let $F\colon \mathrm{Sp}_{T(n)} \rightarrow \mathrm{Sp}_{T(n)}$ be a functor that preserves sifted colimits. Then $F$ is coanalytic.
\end{proposition}

Before proving Proposition \ref{prop:siftedcolimitscoanalytic} we observe the following:

\begin{proof}[Proof of Theorem \ref{thm:PhiThetacoanalytic}]
The functor $\Theta$ is a left adjoint, so it preserves sifted colimits. The Bousfield--Kuhn functor $\Phi$ preserves sifted colimits by Lemma \ref{lem:Phisiftedcolims}. The result now follows immediately from Proposition \ref{prop:siftedcolimitscoanalytic}.
\end{proof}

The following argument was suggested to us by Jacob Lurie:

\begin{proof}[Proof of Proposition \ref{prop:siftedcolimitscoanalytic}]
Let $F$ be as in the statement of the proposition. We will show that $F$ can be written as a colimit of coanalytic functors. 
Corollary \ref{cor:recognition} shows that the full subcategory of coanalytic functors is closed under colimits inside the $\infty$-category $\mathrm{Fun}_*^\omega(\mathrm{Sp}_{T(n)}, \mathrm{Sp}_{T(n)})$, so it will follow that $F$ itself is coanalytic.

For any spectrum $E$, the natural map
\begin{equation*}
\varinjlim_k \Sigma^{\infty - k}\Omega^{\infty-k}(E\langle d_{n+1} \rangle) \rightarrow E
\end{equation*}
is an equivalence. Indeed, if $E$ is already $d_{n+1}$-connected, this is clear (and coincides with the claim that the first Goodwillie derivative of the functor $\Sigma^\infty\Omega^\infty$ is the identity functor of $\mathrm{Sp}$). For a finite spectrum $E$, the claim then follows because $\Omega^{\infty-k}(E\langle d_{n+1} \rangle) \simeq \Omega^{\infty-k} E$ for $k$ sufficiently large; more precisely, large enough for $\Sigma^k E$ to be $d_{n+1}$-connected. For a general spectrum the claim follows by writing it as a filtered colimit of finite spectra. We conclude that for a $T(n)$-local spectrum $E$ we have
\begin{equation*}
F(E) \simeq \varinjlim_k F\bigl(L_{T(n)}\Sigma^{\infty - k}\Omega^{\infty-k}(E \langle d_{n+1} \rangle)\bigr)
\end{equation*}
so it suffices to show that the functor
\begin{equation*}
E \mapsto F\bigl(L_{T(n)}\Sigma^{\infty - k}\Omega^{\infty-k}(E \langle d_{n+1} \rangle)\bigr)
\end{equation*}
is coanalytic for any $k$.

The functor $L_{T(n)}\Sigma^{\infty-k}\colon \mathcal{S}_* \rightarrow \mathrm{Sp}_{T(n)}$ preserves colimits, so that $F \circ L_{T(n)}\Sigma^{\infty -k}$ preserves sifted colimits. The $\infty$-category $\mathcal{S}_*$ is projectively generated (Definition 5.5.8.23 and Example 5.5.8.24 of \cite{htt}) with compact projective generators given by finite pointed sets. In particular, any functor 
\begin{equation*}
\mathcal{S}_* \rightarrow \mathrm{Sp}_{T(n)}
\end{equation*}
preserving sifted colimits is the left Kan extension of its restriction along the inclusion
\begin{equation*}
i\colon \mathcal{F}\mathrm{in}_* \rightarrow \mathcal{S}_*.
\end{equation*}
Let us write $f\colon \mathcal{F}\mathrm{in}_* \rightarrow \mathrm{Sp}_{T(n)}$ for the composition $F \circ L_{T(n)}\Sigma^{\infty -k} \circ i$. The left Kan extension of $f$ along $i$ can be computed by a coend given by the following formula:
\begin{equation*}
\mathrm{Lan}_i f (X) \simeq \int^{I \in \mathcal{F}\mathrm{in}_*} \mathrm{Map}_*(I, X) \otimes f(I).
\end{equation*} 
This coend can equivalently be described as a colimit over the twisted arrow category $\mathrm{TwArr}(\mathcal{F}\mathrm{in}_*)$ (recall that we already discussed the twisted arrow construction in the previous section). Indeed, it is the colimit of the composition of functors
\begin{equation*}
\mathrm{TwArr}(\mathcal{F}\mathrm{in}_*) \rightarrow (\mathcal{F}\mathrm{in}_*)^{\mathrm{op}} \times \mathcal{F}\mathrm{in}_* \xrightarrow{\mathrm{Map}(-,X) \times f} \mathcal{S}_* \times \mathrm{Sp}_{T(n)} \xrightarrow{\otimes} \mathrm{Sp}_{T(n)},
\end{equation*}
where the last step denotes the tensoring of $\mathrm{Sp}_{T(n)}$ over $\mathcal{S}_*$. Note that we are implicitly assuming a choice of functor
\begin{equation*}
\mathrm{Map}_*(-,-)\colon \mathcal{S}_*^{\mathrm{op}} \times \mathcal{S}_* \rightarrow \mathcal{S}_*,
\end{equation*}
which can be taken to be any functor classified by the left fibration $\mathrm{TwArr}(\mathcal{S}_*) \rightarrow \mathcal{S}_*^{\mathrm{op}} \times \mathcal{S}_*$. The upshot of these observations is that $F \circ L_{T(n)}\Sigma^{\infty -k}$ can be written as a colimit of functors of the form
\begin{equation*}
\mathrm{Map}_*(I, -) \otimes F\bigl(L_{T(n)}\Sigma^{\infty -k}(J)\bigr)\colon \mathcal{S}_* \rightarrow \mathrm{Sp}_{T(n)}
\end{equation*}
for finite pointed sets $I$ and $J$.

It follows that the functor 
\begin{equation*}
\mathrm{Sp}_{T(n)} \rightarrow \mathrm{Sp}_{T(n)}\colon E \mapsto F\bigl(L_{T(n)}\Sigma^{\infty - k}\Omega^{\infty-k}(E \langle d_{n+1} \rangle)\bigr)
\end{equation*}
can be obtained as a colimit of functors of the form
\begin{equation*}
E \mapsto \mathrm{Map}_*\bigl(I, \Omega^{\infty-k}(E \langle d_{n+1} \rangle)\bigr) \otimes C,
\end{equation*}
for $C$ a $T(n)$-local spectrum. Thus to prove the theorem it suffices to show that
\begin{equation*}
E \mapsto L_{T(n)} \Sigma^\infty \mathrm{Map}_*\bigl(I, \Omega^{\infty-k}(E \langle d_{n+1} \rangle)\bigr)
\end{equation*}
is coanalytic. Let $I_0$ denote the complement of the basepoint in $I$. Then
\begin{equation*}
L_{T(n)} \Sigma^\infty \mathrm{Map}_*\bigl(I, \Omega^{\infty-k}(E \langle d_{n+1} \rangle)\bigr) \simeq L_{T(n)} \Sigma^\infty\bigl(\Omega^{\infty-k}(E \langle d_{n+1} \rangle)^{\times I_0}\bigr).
\end{equation*}
The right-hand side naturally breaks up as a sum of $T(n)$-local smash powers of the spectrum
\begin{equation*}
L_{T(n)}\Sigma^\infty\Omega^{\infty-k}(E \langle d_{n+1} \rangle) = L_{T(n)}\Sigma^\infty\Omega^\infty \Sigma^k(E \langle d_{n+1} \rangle),
\end{equation*}
so it suffices to prove that each of these factors is coanalytic as a functor of $E$. As in the proof of Theorem \ref{thm:SigmaOmegacoanalytic}, we may use that $\Sigma^k(E \langle d_{n+1} \rangle)$ is $d_{n+1}$-connected and $L_{n-1}^f$-acyclic to conclude that there is a natural equivalence
\begin{equation*}
L_{T(n)}\Sigma^\infty\Omega^{\infty-k}(E \langle d_{n+1} \rangle) \simeq L_{T(n)}\bigoplus_{l=1}^\infty (\Sigma^k E)^{\otimes l}_{h\Sigma_l},
\end{equation*}
finishing the proof.
\end{proof}

\subsection{The monad structure of $\Phi\Theta$}
\label{subsec:PhiThetamonad}

Since the monad $\Phi\Theta$ is a coanalytic functor, it is an operad in $T(n)$-local spectra in the sense of Definition \ref{def:operad}. It is in fact an augmented operad in an obvious way, by using the natural transformation
\begin{equation*}
\Phi\Theta \rightarrow P_1(\Phi\Theta) \simeq \mathrm{id}_{\mathrm{Sp}_{T(n)}}.
\end{equation*}
In terms of the splitting of Theorem \ref{thm:PhiThetafunctor} this augmentation is projection onto the first summand. Dually, the cooperad $\Sigma^\infty_{T(n)}\Omega^\infty_{T(n)}$ is coaugmented via the inclusion of the first summand, which is the natural transformation
\begin{equation*}
\mathrm{id}_{\mathrm{Sp}_{T(n)}} \simeq P^1(\Sigma^\infty_{T(n)}\Omega^\infty_{T(n)}) \rightarrow \Sigma^\infty_{T(n)}\Omega^\infty_{T(n)}.
\end{equation*}
The aim of this section is to show that $\Phi\Theta$ is the Lie operad, i.e., that $\Phi\Theta$ is equivalent to the cobar construction on the commutative cooperad, which is the coanalytic functor $\Sigma^\infty_{T(n)}\Omega^\infty_{T(n)}$. Our first step will be to construct a map
\begin{equation*}
\gamma\colon \Phi\Theta \rightarrow \mathrm{Cobar}(\Sigma^\infty_{T(n)}\Omega^\infty_{T(n)}).
\end{equation*}
As discussed in Section \ref{subsec:Liealg}, such a map corresponds to a lift of the pair $(\Phi\Theta, \Sigma^\infty_{T(n)}\Omega^\infty_{T(n)})$ to an augmented algebra object
\begin{equation*}
\Gamma\colon \Phi\Theta \rightarrow \Sigma^\infty_{T(n)}\Omega^\infty_{T(n)} 
\end{equation*}
of the twisted arrow category of $\mathrm{coAn}(\mathrm{Sp}_{T(n)})$. The natural transformation $\Gamma$ is defined by using the natural transformations
\begin{eqnarray*}
\Phi & \xrightarrow{\Phi\eta} & \Phi\Omega^\infty_{T(n)}\Sigma^\infty_{T(n)} \simeq \Sigma^\infty_{T(n)}, \\
\Theta & \xrightarrow{\eta\Theta} & \Omega^\infty_{T(n)}\Sigma^\infty_{T(n)}\Theta \simeq \Omega^\infty_{T(n)}.
\end{eqnarray*}
Note that $\Gamma$ factors as 
\begin{equation*}
\Phi\Theta \xrightarrow{\Phi\eta\Theta} \Phi\Omega^\infty_{T(n)}\Sigma^\infty_{T(n)}\Theta \simeq \mathrm{id}_{\mathrm{Sp}_{T(n)}} \rightarrow \Sigma^\infty_{T(n)}\Omega^\infty_{T(n)}.
\end{equation*}
We can rewrite this factorization as a diagram
\[
\begin{tikzcd}
\Phi\Theta \ar{r}{\Gamma}\ar{d} & \Sigma^\infty_{T(n)}\Omega^\infty_{T(n)} \\
\mathrm{id}_{\mathrm{Sp}_{T(n)}} \ar{r}{=} &  \mathrm{id}_{\mathrm{Sp}_{T(n)}}, \ar{u}
\end{tikzcd}
\]
which will define the augmentation of $\Gamma$ as an object of the twisted arrow category of $\mathrm{coAn}(\mathrm{Sp}_{T(n)})$.

To find a suitable algebra structure on $\Gamma$, one could try to reason as follows:
\begin{itemize}
\item[(i)] The twisted arrow category of $\mathrm{coAn}(\mathrm{Sp}_{T(n)})$ is a subcategory of the $\infty$-category 
\begin{equation*}
\mathrm{Fun}(\mathrm{TwArr}(\mathrm{Sp}_{T(n)}), \mathrm{TwArr}(\mathrm{Sp}_{T(n)})).
\end{equation*}
The inclusion sends an arrow $F \rightarrow G$ between coanalytic functors to the functor loosely described by
\begin{equation*}
(X \rightarrow Y) \mapsto (F(X) \rightarrow G(Y)),
\end{equation*}
where the latter arrow is the composition of the arrows $F(X) \rightarrow G(X)$ and $G(X) \rightarrow G(Y)$.
\item[(ii)] There are functors
\begin{equation*}
\mathrm{TwArr}(\mathrm{Sp}_{T(n)}) \rightarrow \mathrm{TwArr}(\mathcal{S}_{v_n})\colon (X \rightarrow Y) \mapsto (\Theta X \rightarrow \Omega^\infty_{T(n)}Y)
\end{equation*}
and 
\begin{equation*}
\mathrm{TwArr}(\mathcal{S}_{v_n}) \rightarrow \mathrm{TwArr}(\mathrm{Sp}_{T(n)})\colon (V \rightarrow W) \mapsto (\Phi V \rightarrow \Sigma^\infty_{T(n)} W),
\end{equation*}
with the former left adjoint to the latter. Hence their composition gives a monad on $\mathrm{TwArr}(\mathrm{Sp}_{T(n)})$ or, in other words, an associative algebra object of $\mathrm{Fun}(\mathrm{TwArr}(\mathrm{Sp}_{T(n)}), \mathrm{TwArr}(\mathrm{Sp}_{T(n)}))$.
\item[(iii)] The algebra object of item (ii) is in the image of the inclusion (i), so that one finds the desired algebra object $\Gamma$ of $\mathrm{TwArr}(\mathrm{coAn}(\mathrm{Sp}_{T(n)}))$.
\end{itemize}

This plan can be made precise, but it is a rather tedious exercise which requires some care. We will argue slightly differently, taking our cue from a strategy used by Lurie in Section 4.7 of \cite{higheralgebra}. The $\infty$-category $\mathrm{Fun}(\mathrm{Sp}_{T(n)}, \mathrm{Sp}_{T(n)})$ is a simplicial monoid in an evident way, using composition of functors. Moreover, this simplicial monoid acts on the simplicial set $\mathrm{Fun}(\mathcal{S}_{v_n}, \mathrm{Sp}_{T(n)})$, again by composition. This allows us to regard $\mathrm{Fun}(\mathcal{S}_{v_n}, \mathrm{Sp}_{T(n)})$ as left-tensored over $\mathrm{Fun}(\mathrm{Sp}_{T(n)}, \mathrm{Sp}_{T(n)})$. The functor $\mathrm{TwArr}$ preserves products, so that also $\mathrm{TwArr}(\mathrm{Fun}(\mathcal{S}_{v_n}, \mathrm{Sp}_{T(n)}))$ is left-tensored over the monoidal $\infty$-category $\mathrm{TwArr}(\mathrm{Fun}(\mathrm{Sp}_{T(n)}, \mathrm{Sp}_{T(n)}))$.

In general, if $\mathcal{D}$ is an $\infty$-category left-tensored over a monoidal $\infty$-category $\mathcal{C}$, and one has an action $a\colon E \otimes D \rightarrow D$ with $E \in \mathcal{C}$ and $D \in \mathcal{D}$, then Lurie calls $E$ an \emph{endomorphism object} of $D$ if it is universal for actions on $D$, i.e., if for every $C \in \mathcal{C}$ the induced map
\begin{equation}
\label{eqn:action}
\mathrm{Map}_{\mathcal{C}}(C, E) \xrightarrow{-\otimes D} \mathrm{Map}_{\mathcal{D}}(C \otimes D, E \otimes D) \xrightarrow{- \circ a} \mathrm{Map}_{\mathcal{D}}(C \otimes D, D)
\end{equation}
is an equivalence. In our situation, we have the following (compare Lemma 4.7.4.1 of \cite{higheralgebra}):

\begin{lemma}
\label{lem:PhiThetaendomorphism}
The natural transformation
\begin{equation*}
(\Phi\Theta \rightarrow \Sigma^\infty_{T(n)}\Omega^\infty_{T(n)}) \in \mathrm{TwArr}(\mathrm{Fun}(\mathrm{Sp}_{T(n)}, \mathrm{Sp}_{T(n)}))
\end{equation*}
is an endomorphism object of 
\begin{equation*}
(\Phi \rightarrow \Sigma^\infty_{T(n)}) \in \mathrm{TwArr}(\mathrm{Fun}(\mathcal{S}_{v_n}, \mathrm{Sp}_{T(n)})).
\end{equation*}
\end{lemma}
\begin{proof}
Implicit in the statement of the lemma is that the action of $(\Phi\Theta \rightarrow \Sigma^\infty_{T(n)}\Omega^\infty_{T(n)})$ on $(\Phi \rightarrow \Sigma^\infty_{T(n)})$ is given by the counit of $(\Phi,\Theta)$ and the unit of $(\Sigma^\infty_{T(n)}, \Omega^\infty_{T(n)})$, as in
\[
\begin{tikzcd}
\Phi\Theta\Phi \ar{r}\ar{d}{\Phi\varepsilon} & \Sigma^\infty_{T(n)}\Omega^\infty_{T(n)}\Sigma^\infty_{T(n)} \\
\Phi \ar{r} & \Sigma^\infty_{T(n)}. \ar{u}{\Sigma^\infty_{T(n)}\eta}
\end{tikzcd}
\]
Let $C = (F \rightarrow G)$ be an object of $\mathrm{TwArr}(\mathrm{Fun}(\mathrm{Sp}_{T(n)}, \mathrm{Sp}_{T(n)}))$. Letting $E = (\Phi\Theta \rightarrow \Sigma^\infty_{T(n)}\Omega^\infty_{T(n)})$, the map (\ref{eqn:action}) above takes the form
\begin{equation*}
\mathrm{Map}(F \rightarrow G, \Phi\Theta \rightarrow \Sigma^\infty_{T(n)}\Omega^\infty_{T(n)}) \rightarrow \mathrm{Map}(F\Phi \rightarrow G\Sigma^\infty, \Phi \rightarrow \Sigma^\infty_{T(n)}),
\end{equation*} 
with the mapping spaces computed in the relevant twisted arrow categories. One defines a map in the opposite direction by
\begin{eqnarray*}
\mathrm{Map}(F\Phi \rightarrow G\Sigma^\infty_{T(n)}, \Phi \rightarrow \Sigma^\infty_{T(n)}) & \rightarrow & \mathrm{Map}(F\Phi\Theta \rightarrow G\Sigma^\infty_{T(n)}\Omega^\infty_{T(n)}, \Phi\Theta \rightarrow \Sigma^\infty_{T(n)}\Omega^\infty_{T(n)}) \\
& \rightarrow & \mathrm{Map}(F \rightarrow G, \Phi\Theta \rightarrow \Sigma^\infty_{T(n)}\Omega^\infty_{T(n)}),
\end{eqnarray*}
where the first map is precomposition by $(\Theta \rightarrow \Omega^\infty_{T(n)})$ in both variables and the second map is precomposition in the first variable by the following map in $\mathrm{TwArr}(\mathrm{Sp}_{T(n)}, \mathrm{Sp}_{T(n)})$:
\[
\begin{tikzcd}
F \ar{r}\ar{d}{F\eta} & G \\
F\Phi\Theta \ar{r} & G\Sigma^\infty_{T(n)}\Omega^\infty_{T(n)}. \ar{u}{G\varepsilon}
\end{tikzcd}
\]
The triangle identities show that the map defined in this way is homotopy inverse to (1), proving the lemma.
\end{proof}

\begin{remark}
The preceding lemma and its proof are entirely formal; they apply to any two pairs of adjoint functors
\[
\begin{tikzcd}
\mathcal{C} \arrow[r,"L", shift left] & \mathcal{D} \arrow[r,"K", shift left] \arrow[l, "R", shift left] & \mathcal{C} \arrow[l, "S", shift left] 
\end{tikzcd}
\]
equipped with an identification of left adjoints $KL \simeq \mathrm{id}_{\mathcal{C}}$ (and an induced identification $\mathrm{id}_{\mathcal{C}} \simeq RS$). The relevant natural transformation $R \rightarrow K$ is then defined as the composition
\begin{equation*}
R \xrightarrow{R\eta} RSK \simeq K.
\end{equation*}
\end{remark}

% add remark that the given algebra structure on \Gamma is compatible with monad structure of \Phi\Theta

In Section 4.7.2 of \cite{higheralgebra} Lurie proves that an endomorphism object can be given the structure of an associative algebra in an essentially unique way (Theorem 4.7.2.34). Therefore Lemma \ref{lem:PhiThetaendomorphism} above provides the desired algebra object $\Gamma$ of $\mathrm{TwArr}(\mathrm{Fun}(\mathrm{Sp}_{T(n)}, \mathrm{Sp}_{T(n)}))$. It is straightforward to see that this algebra structure is compatible with the augmentation of $\Gamma$ we defined above, for example by modifying the argument of the previous lemma to show that $\Gamma$ with its augmentation is universal not just in the twisted arrow category  $\mathrm{TwArr}(\mathrm{Fun}(\mathrm{Sp}_{T(n)}, \mathrm{Sp}_{T(n)}))$, but also in the slice category of such twisted arrows augmented over the arrow $\mathrm{id}_{\mathrm{Sp}_{T(n)}} = \mathrm{id}_{\mathrm{Sp}_{T(n)}}$. This concludes our discussion of the construction of $\Gamma$ and the associated morphism of operads $\gamma$.

\begin{theorem}
\label{thm:comparisonoperads}
The map of operads
\begin{equation*}
\gamma\colon \Phi\Theta \rightarrow \mathrm{Cobar}(\Sigma^\infty_{T(n)}\Omega^\infty_{T(n)})
\end{equation*}
is an equivalence.
\end{theorem}
\begin{proof}
Since $\gamma$ is in particular a natural transformation between coanalytic functors, it suffices (by Lemma \ref{lem:coancoefficients}) to check that $\gamma$ induces an equivalence on $k$th derivatives, for every $k \geq 1$, or equivalently on $k$-excisive approximations $P_k$ for every $k$. The cobar construction $\mathrm{Cobar}(\Sigma^\infty_{T(n)}\Omega^\infty_{T(n)})$ is the totalization (in the $\infty$-category of coanalytic functors) of the cosimplicial object
\[
\begin{tikzcd}
\mathrm{id}_{\mathrm{Sp}_{T(n)}} \ar[shift left = .06cm]{r} \ar[shift right = .06cm]{r} & \Sigma^\infty_{T(n)}\Omega^\infty_{T(n)} \ar[shift left = .12cm]{r} \ar{r} \ar[shift right = .12cm]{r} &\Sigma^\infty_{T(n)}\Omega^\infty_{T(n)}\Sigma^\infty_{T(n)}\Omega^\infty_{T(n)} \ar[shift right = .18cm]{r}\ar[shift right = .06cm]{r}\ar[shift left = .06cm]{r}\ar[shift left = .18cm]{r} & \cdots
\end{tikzcd}
\]
which in turn arises from the cosimplicial object
\[
\begin{tikzcd}
\Omega^\infty_{T(n)}\Sigma^\infty_{T(n)} \ar[shift left = .06cm]{r} \ar[shift right = .06cm]{r} & \Omega^\infty_{T(n)}\Sigma^\infty_{T(n)}\Omega^\infty_{T(n)}\Sigma^\infty_{T(n)} \ar[shift left = .12cm]{r} \ar{r} \ar[shift right = .12cm]{r} & (\Omega^\infty_{T(n)}\Sigma^\infty_{T(n)})^3 \ar[shift right = .18cm]{r}\ar[shift right = .06cm]{r}\ar[shift left = .06cm]{r}\ar[shift left = .18cm]{r} & \cdots
\end{tikzcd}
\]
by precomposing with $\Theta$ and postcomposing with $\Phi$. In other words,
\begin{equation*}
\mathrm{Cobar}(\Sigma^\infty_{T(n)}\Omega^\infty_{T(n)}) \simeq \mathrm{Tot}\bigl(\Phi (\Omega^\infty_{T(n)}\Sigma^\infty_{T(n)})^{\bullet + 1}\Theta\bigr).
\end{equation*}
The cosimplicial object $(\Omega^\infty_{T(n)}\Sigma^\infty_{T(n)})^{\bullet + 1}$ has an evident coaugmentation (namely the unit $\eta$) from $\mathrm{id}_{\mathcal{S}_{v_n}}$. Unravelling the definitions, the natural transformation underlying $\gamma$ can then be identified as follows:
\begin{equation*}
\Phi\Theta = \Phi\mathrm{id}_{\mathcal{S}_{v_n}}\Theta \xrightarrow{\Phi\eta\Theta} \mathrm{Tot}\bigl(\Phi(\Omega^\infty_{T(n)}\Sigma^\infty_{T(n)})^{\bullet + 1}\Theta\bigr).
\end{equation*}
We stress again that the totalizations in the previous two displays are to be computed in the $\infty$-category of coanalytic functors. As in the proof of Theorem \ref{thm:PhiThetaderivatives} one can compute the $k$-excisive approximation of $\Phi\Theta$ by 
\begin{equation*}
P_k(\Phi\Theta) = \Phi P_k(\mathrm{id}_{\mathcal{S}_{v_n}}) \Theta
\end{equation*}
and similarly on the right-hand side one has
\begin{equation*}
P_k\mathrm{Tot}\bigl(\Phi (\Omega^\infty_{T(n)}\Sigma^\infty_{T(n)})^{\bullet + 1}\Theta\bigr) \simeq \Phi\mathrm{Tot}\bigl(P_k((\Omega^\infty_{T(n)}\Sigma^\infty_{T(n)})^{\bullet + 1})\bigr)\Theta.
\end{equation*}
Now the totalization on the right-hand side is computed in the $\infty$-category of functors from $\mathcal{S}_{v_n}$ to itself; the equivalence uses the fact that a limit of $k$-excisive functors is $k$-excisive, so that the right-hand side is indeed a coanalytic (even $k$-excisive) functor. It now suffices to verify that the natural transformation
\begin{equation*}
P_k(\mathrm{id}_{\mathcal{S}_{v_n}}) \rightarrow \mathrm{Tot}\bigl(P_k((\Omega^\infty_{T(n)}\Sigma^\infty_{T(n)})^{\bullet + 1})\bigr)
\end{equation*}
is an equivalence for every $k$. But this is a variation on a crucial result of Arone and Ching, namely Theorem 0.3 of \cite{aroneching}. It is also formulated in a way that applies here directly as Proposition B.4 of \cite{heutsgoodwillie}.

% unravelling in more detail:
% map \Phi\Theta \rightarrow Cobar(\Sigma^\infty\Omega^\infty) can be identified as coming from
% \Phi\Theta \rightarrow Cobar(\Phi\Omega^\infty, \Sigma^\infty\Omega^\infty, \Sigma^\infty\Theta)^\bullet
% latter comes from regarding \Phi\Omega^\infty \rightarrow \Sigma^\infty\Theta as left module and
% \Sigma^\infty\Theta \rightarrow \Phi\Omega^\infty as right module over
% \Phi\Theta \rightarrow \Sigma^\infty\Omega^\infty
% then check that this is compatible with the defined augmentation of this monoid

\end{proof}

Recall that our main result, Theorem \ref{thm:MnfLiealg}, states that $\mathcal{S}_{v_n}$ is equivalent to the $\infty$-category of Lie algebras in $\mathrm{Sp}_{T(n)}$. We have now established the necessary ingredients to conclude its proof:

\begin{proof}[Proof of Theorem \ref{thm:MnfLiealg}]
Combine Theorems \ref{thm:Phimonadic} and \ref{thm:comparisonoperads}.
\end{proof}

\subsection{Some applications}
\label{subsec:applications}

In this section we indicate how to deduce Theorems \ref{thm:Phitypenspace} and \ref{thm:Sigmainftyfiltration} from the results established above. Also, we include a discussion of the values of $\Phi\Theta$ on (shifts of) the $T(n)$-local sphere spectrum.

\begin{theorem*}[Theorem \ref{thm:Phitypenspace}]
Suppose $V$ is a pointed finite type $n$ space with a $v_n$ self-map and write $W = \Sigma^2 V$. Then there is an equivalence of spectra as follows:
\begin{equation*}
\Phi(W) \simeq L_{T(n)}\bigoplus_{k \geq 1}(\partial_k \mathrm{id} \otimes \Sigma^\infty W^{\otimes k})_{h\Sigma_k}.
\end{equation*}
\end{theorem*}
\begin{proof}[Proof of Theorem \ref{thm:Phitypenspace}]
By Lemma \ref{lem:valuetheta} and the subsequent remark we have
\begin{equation*}
L_n^f W \simeq \Theta(\Sigma^\infty_{T(n)} W).
\end{equation*}
(Here we are tacitly taking the space $V_{n+1}$ defining $L_n^f$ as in Remark \ref{lem:valuetheta}; the statement of Theorem \ref{thm:Phitypenspace} is clearly independent of this choice.) Therefore
\begin{equation*}
\Phi(W) \simeq \Phi(L_n^f W) \simeq \Phi\Theta(\Sigma^\infty_{T(n)} W)
\end{equation*}
and applying Theorem \ref{thm:PhiThetafunctor} to the last expression gives the desired conclusion.
\end{proof}

Our next goal is to give a formula for $\Phi\Theta(L_{T(n)}\mathbb{S}^\ell)$ for any integer $\ell$. For this we use work of Arone and Mahowald \cite{aronemahowald} on the $v_n$-periodic Goodwillie tower of spheres:

\begin{theorem}[Arone--Mahowald]
\label{thm:aronemahowald}
Let $X$ be a pointed space for which $L_{T(n)}\Sigma^\infty X$ is equivalent to $L_{T(n)} \mathbb{S}^\ell$, i.e., a shift of the $T(n)$-local sphere spectrum, with $\ell$ any integer. Then $\Phi D_k\mathrm{id}(X)$ is contractible for $k > p^n$ (if $\ell$ is odd) or $k > 2p^n$ (if $\ell$ is even). In other words, the tower
\begin{equation*}
\cdots \rightarrow \Phi P_3\mathrm{id}(X) \rightarrow \Phi P_2\mathrm{id}(X) \rightarrow \Phi P_1\mathrm{id}(X)
\end{equation*}
becomes constant at stage $p^n$ or $2p^n$, depending on the parity of $\ell$.
\end{theorem}
\begin{proof}
In Sections 4.1 and 4.2 of \cite{aronemahowald}, Arone and Mahowald prove that for $X = S^\ell$ a sphere of dimension at least 1, the spectrum
\begin{equation*}
L_{T(n)}(\partial_k \mathrm{id} \otimes \Sigma^\infty X^{\otimes k})_{h\Sigma_k}
\end{equation*}
is contractible for $k > p^n$ if $\ell$ is odd or $k > 2p^n$ otherwise. Also, if $\ell$ is odd then it is always contractible if $k$ is not a power of $p$ (or twice a power of $p$ in case $\ell$ is even). These statements only depend on the $T(n)$-localization of $\Sigma^\infty X$, so that they are true for any pointed space $X$ whose suspension spectrum is $T(n)$-equivalent to $L_{T(n)}\mathbb{S}^\ell$ for $\ell \geq 1$. The calculations of Arone and Mahowald are equally valid for negative values of $\ell$; this observation is already known at least to Greg Arone and Nick Kuhn. There are also various ways to deduce the result for negative $\ell$ from the positive one. In private communication Kuhn suggests using the Thom isomorphism to relate the mod $p$ cohomology of $(\partial_k \mathrm{id} \otimes Y^{\otimes k})_{h\Sigma_k}$ to that of $(\partial_k \mathrm{id} \otimes \Sigma^2 Y^{\otimes k})_{h\Sigma_k}$, where one determines the action of the Steenrod algebra on the latter by examining the effect of Steenrod operations on the Thom class. Once one establishes that the cohomology of these spectra is free over an appropriate subalgebra of the Steenrod algebra, a standard vanishing line argument in the Adams spectral sequence gives the required vanishing of $T(n)$-localizations. Rather than spell this out we offer an alternative proof using duality results of Arone and Dwyer. 

First we make a reduction using the EHP sequence, following Propositions 4.6 and 4.7 of \cite{aronemahowald} or Section 2.1 of \cite{behrensehp}. For positive odd $\ell$ it gives a fiber sequence
\begin{equation*}
S^\ell \rightarrow \Omega S^{\ell+1} \rightarrow \Omega S^{2\ell+1}.
\end{equation*}
By taking derivatives of the functors involved, one concludes that for any $k > 0$ the corresponding sequence
\begin{equation*}
\mathbf{D}_k(E) \rightarrow \Omega\mathbf{D}_k(\Sigma E) \rightarrow \Omega\mathbf{D}_{\frac{k}{2}}\bigl(\Sigma E^{\otimes 2} \bigr)
\end{equation*}
becomes a fiber sequence when evaluated on shifts of the sphere spectrum $E = \mathbb{S}^\ell$ for \emph{any} odd $\ell$ (not necessarily positive). The last term is to be interpreted as $0$ when $k$ is odd. Here $\mathbf{D}_k$ denotes the functor which gives $D_k$ after applying $\Omega^\infty$, i.e., 
\begin{equation*}
\mathbf{D}_k(E) = (\partial_k \mathrm{id} \otimes E^{\otimes k})_{h\Sigma_k}.
\end{equation*}
The fiber sequence above resolves the Goodwillie layers for an even sphere in terms of the layers of odd spheres, so it suffices to treat the odd case.

Take $\ell$ to be a positive odd integer. We aim to show that the spectrum
\begin{equation*}
L_{T(n)}(\partial_k \mathrm{id} \otimes (\mathbb{S}^{-\ell})^{\otimes k})_{h\Sigma_k}
\end{equation*}
is contractible for $k > p^n$. Note that the Thom isomorphism implies that if $k$ is not a power of $p$ then this spectrum is null, even before $T(n)$-localization, simply because the mod $p$ homology of these layers vanishes for positive odd spheres. Therefore it suffices to consider the case $k = p^m$ for $m > n$. Recall that the derivative $\partial_k\mathrm{id}$ is the Spanier--Whitehead dual of the partition complex,
\begin{equation*}
\partial_k\mathrm{id} \simeq \mathbf{D}(\Sigma \mathbf{Part}_k^{\diamond}).
\end{equation*}
The spectrum under consideration can be described as
\begin{eqnarray*}
L_{T(n)}\bigl(\mathbf{D}(\Sigma \mathbf{Part}_k^{\diamond} \otimes (\mathbb{S}^\ell)^{\otimes k})\bigr)_{h\Sigma_k} & \simeq &  \bigl(L_{T(n)} \mathbf{D}(\Sigma \mathbf{Part}_k^{\diamond} \otimes (\mathbb{S}^\ell)^{\otimes k})\bigr)^{h\Sigma_k} \\
& \simeq & L_{T(n)}\mathbf{D}\bigl((\Sigma \mathbf{Part}_k^{\diamond} \otimes (\mathbb{S}^\ell)^{\otimes k})_{h\Sigma_k}\bigr),
\end{eqnarray*}
where the first equivalence uses that Tate spectra for finite groups vanish in the $\infty$-category of $T(n)$-local spectra. Theorem 1.16 of \cite{aronedwyer}, which concerns a certain self-duality of the partition complex, gives an equivalence 
\begin{equation*}
L_{T(n)}\bigl((\Sigma \mathbf{Part}_k^{\diamond} \otimes (\mathbb{S}^\ell)^{\otimes k})_{h\Sigma_k}\bigr) \simeq L_{T(n)}\mathbf{D}\bigl(\Sigma^{2k}(\partial_k \mathrm{id} \otimes (\mathbb{S}^\ell)^{\otimes k})_{h\Sigma_k}\bigr).
\end{equation*} 
The spectrum on the right is contractible by the result of Arone and Mahowald for positive odd spheres.
\end{proof}

\begin{corollary}
\label{cor:aronemahowald}
There is a natural equivalence
\begin{equation*}
\Phi\Theta(L_{T(n)}\mathbb{S}^\ell) \simeq L_{T(n)}\bigoplus_{k=1}^{2p^n}(\partial_k \mathrm{id} \otimes (\mathbb{S}^\ell)^{\otimes k})_{h\Sigma_k},
\end{equation*}
where $2p^n$ may be replaced by $p^n$ if $\ell$ is odd.
\end{corollary}
\begin{proof}
Since $\Sigma^\infty_{T(n)}\Theta$ is equivalent to the identity functor of $\mathrm{Sp}_{T(n)}$, the spaces $\Theta(L_{T(n)}\mathbb{S}^\ell)$ satisfy the hypotheses of Theorem \ref{thm:aronemahowald}. The conclusion follows by applying Theorem \ref{thm:PhiThetafunctor}.
\end{proof}

Arone and Mahowald prove that when $X$ is a sphere, the natural map
\begin{equation*}
\Phi(\varprojlim_k P_k\mathrm{id}(X)) \rightarrow \varprojlim_k \Phi P_k\mathrm{id}(X)
\end{equation*}
is an equivalence. The previous corollary shows that the same is true when $X = \Theta(L_{T(n)}\mathbb{S}^\ell)$, thus giving a new class of nontrivial examples of $\Phi$-good spaces. These spaces can be thought of as `fake spheres'. Their suspension spectra are $T(n)$-locally equivalent to those of spheres, but their $v_n$-periodic homotopy groups are quite different. Indeed, the spectrum $\Phi\Theta(L_{T(n)}\mathbb{S}^\ell)$ splits as a sum of layers corresponding to the derivatives of the functor $\Phi\Theta$, whereas the Goodwillie tower for $\Phi(X)$ does not split when $X$ is an ordinary sphere. For example, when the height $n$ is 1 there is a fiber sequence
\begin{equation*}
\Phi(S^{2\ell + 1}) \rightarrow \Phi D_1\mathrm{id}(S^{2\ell + 1}) \rightarrow \Sigma \Phi D_p\mathrm{id}(S^{2\ell + 1}).
\end{equation*}
For $p$ odd, work of Bousfield \cite{bousfieldHspace} shows that this sequence can be identified with a fiber sequence 
\begin{equation*}
\Phi(S^{2\ell + 1}) \rightarrow L_{T(1)} \mathbb{S}^{2\ell + 1} \xrightarrow{p^\ell} L_{T(1)} \mathbb{S}^{2\ell + 1},
\end{equation*}
so that $\Phi(S^{2\ell + 1}) \simeq L_{T(1)}\mathbb{S}^{2\ell}/p^\ell$. In particular, the `attaching map' between $D_1$ and $D_p$ is nontrivial. By contrast, one finds
\begin{equation*}
\Phi\Theta(L_{T(1)}\mathbb{S}^\ell) \simeq L_{T(1)}\bigl(\mathbb{S}^\ell \oplus \mathbb{S}^{\ell - 1}).
\end{equation*}

The remainder of this section is devoted to the following result from Section \ref{sec:mainresults}:

\begin{theorem*}[Theorem \ref{thm:Sigmainftyfiltration}]
For $X \in \mathcal{S}_{v_n}$, the suspension spectrum $\Sigma^\infty_{T(n)} X$ admits a natural filtration with associated graded the spectrum
\begin{equation*}
\mathrm{Sym}^\ast(\Phi (X)) = L_{T(n)}\bigoplus_{k \geq 1}\Phi(X)^{\otimes k}_{h\Sigma_k}.
\end{equation*}
\end{theorem*}

We will now describe the filtration of the theorem. Recall that by Theorem \ref{thm:Phimonadic} every $X \in \mathcal{S}_{v_n}$ has a standard bar resolution
\[
\begin{tikzcd}
\cdots \ar[shift right = .18cm]{r}\ar[shift right = .06cm]{r}\ar[shift left = .06cm]{r}\ar[shift left = .18cm]{r}  & \Theta\Phi\Theta\Phi\Theta\Phi X  \ar[shift left = .12cm]{r} \ar{r} \ar[shift right = .12cm]{r} & \Theta\Phi\Theta\Phi X  \ar[shift left = .06cm]{r} \ar[shift right = .06cm]{r} & \Theta\Phi X \ar{r} & X
\end{tikzcd}
\]
and therefore the suspension spectrum $\Sigma^\infty_{T(n)} X$ will have a corresponding resolution
\[
\begin{tikzcd}
\cdots \ar[shift right = .18cm]{r}\ar[shift right = .06cm]{r}\ar[shift left = .06cm]{r}\ar[shift left = .18cm]{r}  & \Phi\Theta\Phi\Theta\Phi X  \ar[shift left = .12cm]{r} \ar{r} \ar[shift right = .12cm]{r} & \Phi\Theta\Phi X  \ar[shift left = .06cm]{r} \ar[shift right = .06cm]{r} & \Phi X \ar{r} & \Sigma^\infty_{T(n)} X,
\end{tikzcd}
\]
using that $\Sigma^\infty_{T(n)}\Theta \simeq \mathrm{id}_{\mathrm{Sp}_{T(n)}}$. We filter $\Sigma^\infty_{T(n)} X$ by filtering the terms in this resolution. Define
\begin{equation*}
F^k \Sigma^\infty_{T(n)} X := |P^k\bigl((\Phi\Theta)^\bullet \Phi\bigr)|(X). 
\end{equation*}
This gives an increasing filtration
\begin{equation*}
\Phi X \simeq F^1\Sigma^\infty_{T(n)} X \rightarrow F^2 \Sigma^\infty_{T(n)} X \rightarrow F^3 \Sigma^\infty_{T(n)} X \rightarrow \cdots
\end{equation*}
of $\Sigma^\infty_{T(n)}X$. Since $\Phi$ preserves limits, we also have
\begin{equation*}
P^k\bigl((\Phi\Theta)^\bullet \Phi\bigr)(X) \simeq P^k\bigl((\Phi\Theta)^\bullet\bigr) \Phi(X).
\end{equation*}
Recall that $\Phi\Theta$ corresponds to the symmetric sequence given by the ($T(n)$-local) derivatives of the identity functor, and $(\Phi\Theta)^\bullet$ corresponds to the $\bullet$-fold composition product of that symmetric sequence with itself. The functor $P^k\bigl((\Phi\Theta)^\bullet\bigr)$ then simply corresponds to the first $k$ terms of that composition product. (This is also how Behrens and Rezk define a filtration in Section 4 of \cite{behrensrezk}, although there the relevant operad is the commutative rather than the Lie operad. They refer to this as the \emph{Kuhn filtration}, which is developed much more generally in \cite{kuhnpereira}.) Thus the object $P^k\bigl((\Phi\Theta)^\bullet\bigr) \Phi(X)$ consists of $k$ graded pieces. The simplicial structure maps respect this grading, except possibly for ($P^k$ applied to) the `final face maps'
\begin{equation*}
(\Phi\Theta)^\bullet \Phi(X) = (\Phi\Theta)^{\bullet-1}\Phi\Theta\Phi(X) \xrightarrow{(\Phi\Theta)^{\bullet-1}\Phi\varepsilon} (\Phi\Theta)^{\bullet-1}\Phi(X)
\end{equation*}
coming from the $\Phi\Theta$-algebra structure of $\Phi(X)$. Note that Theorem \ref{thm:PhiThetacoanalytic} implies that
\begin{equation*}
\varinjlim_k P^k\bigl((\Phi\Theta)^\bullet\bigr) \simeq (\Phi\Theta)^\bullet,
\end{equation*}
so that the filtration we have defined is exhaustive:
\begin{equation*}
\varinjlim_k F^k \Sigma^\infty_{T(n)} X \simeq \Sigma^\infty_{T(n)} X.
\end{equation*}
We offer further remarks on this filtration after we have settled the following:

\begin{proof}[Proof of Theorem \ref{thm:Sigmainftyfiltration}]
Write $\mathrm{gr}^k\Sigma^\infty_{T(n)} X$ for the cofiber of $F^{k-1}\Sigma^\infty_{T(n)} X \rightarrow F^k \Sigma^\infty_{T(n)} X$. As already noted above, 
\begin{equation*}
\mathrm{gr}^1\Sigma^\infty_{T(n)} X = F^1\Sigma^\infty_{T(n)} X \simeq \Phi X,
\end{equation*}
simply because the simplicial object $P^1\bigl((\Phi\Theta)^\bullet\bigr) \Phi X$ is constant with value $\Phi X$. For $k >1$ we find
\begin{eqnarray*}
\mathrm{gr}^k\Sigma^\infty_{T(n)} X & \simeq & |D^k\bigl((\Phi\Theta)^\bullet\bigr) \Phi X | \\
& \simeq & \bigl(D^k \mathrm{Bar}(\Phi\Theta)\bigr)(\Phi X) \\
& \simeq & \bigl(D^k \Sigma^\infty_{T(n)}\Omega^\infty_{T(n)})(\Phi X) \\
& \simeq & L_{T(n)}(\Phi X)^{\otimes k}_{h\Sigma_k}.
\end{eqnarray*}
Here on the second line, $\mathrm{Bar}(\Phi\Theta)$ is the geometric realization of the simplicial object
\[
\begin{tikzcd}
\cdots \ar[shift right = .18cm]{r}\ar[shift right = .06cm]{r}\ar[shift left = .06cm]{r}\ar[shift left = .18cm]{r}  & \Phi\Theta\Phi\Theta  \ar[shift left = .12cm]{r} \ar{r} \ar[shift right = .12cm]{r} & \Phi\Theta  \ar[shift left = .06cm]{r} \ar[shift right = .06cm]{r} & \mathrm{id}_{\mathrm{Sp}_{T(n)}}.
\end{tikzcd}
\]
This simplicial object arises by precomposing with $\Omega^\infty_{T(n)}$ and postcomposing with $\Sigma^\infty_{T(n)}$ from the standard resolution of the identity functor of $\mathcal{S}_{v_n}$ associated with the comonad $\Theta\Phi$:
\[
\begin{tikzcd}
\cdots \ar[shift right = .18cm]{r}\ar[shift right = .06cm]{r}\ar[shift left = .06cm]{r}\ar[shift left = .18cm]{r}  & \Theta\Phi\Theta\Phi\Theta\Phi  \ar[shift left = .12cm]{r} \ar{r} \ar[shift right = .12cm]{r} & \Theta\Phi\Theta\Phi  \ar[shift left = .06cm]{r} \ar[shift right = .06cm]{r} & \Theta\Phi \ar{r} & \mathrm{id}_{\mathcal{S}_{v_n}}.
\end{tikzcd}
\]
In particular $\mathrm{Bar}(\Phi\Theta) \simeq \Sigma^\infty_{T(n)}\Omega^\infty_{T(n)}$, giving the equivalence between the second and the third line above. The identification of the first and second line follows because the `final face' maps in the simplicial object $D^k\bigl((\Phi\Theta)^\bullet\bigr) \Phi X$ are canonically null (compare also the proof of Proposition 4.5 of \cite{behrensrezk}). % add some explanation: multiplication map \Phi\Theta\Phi X \rightarrow \Phi X strictly decreases grading everywhere except degree one
Theorem \ref{thm:SigmaOmegacoanalytic} gives the equivalence between the third and fourth line.
\end{proof}

\begin{remark}
\label{rmk:Kuhnfiltration}
The kind of filtration we have just defined exists quite generally for algebras over an operad. In the case of commutative ring spectra used by Behrens and Rezk \cite{behrensrezk} it essentially originates with Kuhn's work \cite{kuhnmccord}, whereas a much more general version is treated by Kuhn and Pereira \cite{kuhnpereira}. In fact, the filtration of Theorem \ref{thm:Sigmainftyfiltration} already exists at the level of algebras. Let us only outline this informally, since we will not need it. Consider an operad $\mathbf{O}$ in spectra (or some other stable symmetric monoidal homotopy theory with sufficiently many colimits) and for simplicity assume $\mathbf{O}(1) = \mathbb{S}$ and our operads are nonunital (i.e. there is no $\mathbf{O}(0)$ term). The stabilization of the $\infty$-category of $\mathbf{O}$-algebras can be identified with the $\infty$-category of spectra (this was first observed, perhaps in different language, by Basterra-Mandell \cite{basterramandell}) and the corresponding functor
\begin{equation*}
\Sigma^\infty\colon \mathrm{Alg}(\mathbf{O}) \rightarrow \mathrm{Sp}
\end{equation*}
can be thought of as \emph{derived $\mathbf{O}$-indecomposables} and is also called \emph{topological Quillen homology}. We will describe a filtration of any $\mathbf{O}$-algebra $A$ which reproduces the kind of filtration considered above after applying $\Sigma^\infty$. Consider the category $\mathbf{Op}_{\leq k}$ of \emph{$k$-truncated operads} in spectra, which are simply truncated symmetric sequences $\{\mathbf{O}(j)\}_{1 \leq j \leq k}$ equipped with the usual structure maps required of an operad, as long as they make sense; in other words, multiplication maps
\begin{equation*}
\mathbf{O}(m) \otimes \mathbf{O}(j_1) \otimes \cdots \otimes \mathbf{O}(j_m) \rightarrow \mathbf{O}(j_1 + \cdots + j_m)
\end{equation*}
are only defined if $j_1 + \cdots + j_m \leq k$. There is an evident forgetful functor $\mathbf{Op} \rightarrow \mathbf{Op}_{\leq k}$ from all operads to $k$-truncated operads. This functor has both a left and a right adjoint. The right adjoint is easy to describe: it simply extends a truncated symmetric sequence by zero in arities above $k$. The left adjoint (let us write $f_k$) is more interesting: for a truncated operad $\mathbf{O}_{\leq k}$, the operad $f_k\mathbf{O}_{\leq k}$ agrees with $\mathbf{O}_{\leq k}$ in arities up to $k$ and is `freely generated' above that. In other words, all generators and relations are determined by $\mathbf{O}_{\leq k}$. One could give an explicit formula for $f_k$ using the construction of free operads, but we will not. For an operad $\mathbf{O}$, write $\varepsilon_k\colon f_k(\mathbf{O}_{\leq k}) \rightarrow \mathbf{O}$ for the counit. Corresponding to the morphism $\varepsilon_k$ there is an adjunction
\[
\begin{tikzcd}
\mathrm{Alg}(f_k(\mathbf{O}_{\leq k})) \ar[shift left]{r}{(\varepsilon_k)_!} & \mathrm{Alg}(\mathbf{O}). \ar[shift left]{l}{\varepsilon_k^*}
\end{tikzcd}
\]
For an $\mathbf{O}$-algebra $A$ one now obtains the promised filtration by
\begin{equation*}
(\varepsilon_1)_!\varepsilon_1^* A \rightarrow (\varepsilon_2)_!\varepsilon_2^* A \rightarrow (\varepsilon_3)_!\varepsilon_3^* A\rightarrow \cdots \rightarrow A.
\end{equation*}
Note that the operad $f_1(\mathbf{O}_{\leq 1})$ is simply the trivial operad, and the adjoint pair $((\varepsilon_1)_!,\varepsilon_1^*)$ can be identified with the free-forgetful adjunction for the $\infty$-category $\mathrm{Alg}(\mathbf{O})$. Composing topological Quillen homology with the free algebra functor gives the identity, so the composition of topological Quillen homology with the first stage of the filtration $(\varepsilon_1)_!\varepsilon_1^*$ is the forgetful functor; note that the `forgetful functor' $\Phi$ for $\Phi\Theta$-algebras was also the starting point of our filtration of $\Sigma^\infty_{T(n)} X$.

A final remark is that the filtration $\{(\varepsilon_k)_!\varepsilon_k^*\}_{k\geq1}$ of the identity functor of $\mathrm{Alg}(\mathbf{O})$ plays a role dual to that of the Goodwillie tower of the identity: the functor $(\varepsilon_k)_!\varepsilon_k^*$ can be shown to be $k$-\emph{co}excisive, and we would like to say that it is the universal $k$-coexcisive functor with a natural transformation to the identity functor. One obstruction to proving this is the lack of a good theory of dual Goodwillie calculus in an unstable setting. We will comment more on issues related to dual calculus in Appendix \ref{app:Goodwillie}.
\end{remark}

\subsection{A variation for $K(n)$-local homotopy theory}
\label{subsec:Knlocal}

The purpose of this section is to derive the following from Theorem \ref{thm:MnfLiealg}: 

\begin{corollary*}[Corollary \ref{cor:Knlocal}]
The localization of $\mathcal{S}_{v_n}$ at the $\Phi_{K(n)}$-equivalences exists. More precisely, there exists a full subcategory $\mathcal{M}_{K(n)} \rightarrow \mathcal{S}_{v_n}$ for which the inclusion admits a left adjoint, satisfying the following two properties: 
\begin{itemize}
\item[(i)] The unit is a $\Phi_{K(n)}$-equivalence.
\item[(ii)] A map in $\mathcal{M}_{K(n)}$ is an equivalence if and only if it is a $\Phi_{K(n)}$-equivalence.
\end{itemize}
Moreover, the $\infty$-category $\mathcal{M}_{K(n)}$ is equivalent to the $\infty$-category $\mathrm{Lie}(\mathrm{Sp}_{K(n)})$ of Lie algebras in $K(n)$-local spectra.
\end{corollary*}

Essentially one only needs the following observation:

\begin{lemma}
\label{lem:KnPhiTheta}
The functor $\Phi\Theta$ preserves $K(n)$-equivalences of $T(n)$-local spectra.
\end{lemma}
\begin{proof}
If $f\colon X \rightarrow Y$ is a $K(n)$-equivalence, then so is $Z \otimes f$ for any spectrum $Z$. Moreover, $K(n)$-equivalences (or equivalences for any homology theory) are closed under colimits. Hence the induced map
\begin{equation*}
L_{T(n)}\bigoplus_{k \geq 1} (\partial_k\mathrm{id} \otimes X^{\otimes k})_{h\Sigma_k} \rightarrow L_{T(n)}\bigoplus_{k \geq 1} (\partial_k\mathrm{id} \otimes Y^{\otimes k})_{h\Sigma_k}
\end{equation*}
is a $K(n)$-equivalence. The lemma now follows by applying Theorem \ref{thm:PhiThetafunctor}.
\end{proof}

We write $\mathrm{Lie}(\mathrm{Sp}_{K(n)})$ for the full subcategory on objects of $\mathrm{Lie}(\mathrm{Sp}_{T(n)})$ whose underlying spectrum is $K(n)$-local. Corollary \ref{cor:Knlocal} is a straightforward consequence of Theorem \ref{thm:MnfLiealg} together with the following:

\begin{proposition}
\label{prop:KnlocalLiealg}
The inclusion $\mathrm{Lie}(\mathrm{Sp}_{K(n)}) \rightarrow \mathrm{Lie}(\mathrm{Sp}_{T(n)})$ admits a left adjoint, which on the level of underlying spectra is the $K(n)$-localization functor $L_{K(n)}$. In particular, $\mathrm{Lie}(\mathrm{Sp}_{K(n)})$ is the localization of $\mathrm{Lie}(\mathrm{Sp}_{T(n)})$ at the $K(n)$-equivalences.
\end{proposition}

The idea is that if $E$ is a $T(n)$-local spectrum that is an algebra for the monad $\Phi\Theta$, then $L_{K(n)} E$ is also an algebra by using the map
\begin{equation*}
\Phi\Theta(L_{K(n)} E) \rightarrow L_{K(n)}\Phi\Theta(L_{K(n)} E) \simeq L_{K(n)}\Phi\Theta(E) \rightarrow L_{K(n)} E.
\end{equation*}
The first map is $K(n)$-localization, the middle equivalence follows from Lemma \ref{lem:KnPhiTheta}, and the last map is the $K(n)$-localization of the structure map $\Phi\Theta(E) \rightarrow E$. The proof of Proposition \ref{prop:KnlocalLiealg} is simply an exercise in making this precise.

\begin{proof}[Proof of Proposition \ref{prop:KnlocalLiealg}]
Let us first give an explicit description of the $\infty$-category $\mathrm{LMod}_{\Phi\Theta}(\mathrm{Sp}_{T(n)})$. In Section 4.2.1 of \cite{higheralgebra} Lurie defines an $\infty$-operad $\mathcal{LM}^\otimes$; it is the $\infty$-operad associated to the ordinary operad $\mathbf{LM}$ (in sets) on two colours $a$ and $m$, for which an algebra is precisely an associative monoid (corresponding to $a$) and a left module over it (corresponding to $m$). In particular, the full suboperad of $\mathbf{LM}$ on the colour $a$ is isomorphic to the associative operad. 

Let us write $\mathrm{Fun}_{K(n)}(\mathrm{Sp}_{T(n)},\mathrm{Sp}_{T(n)})$ for the full subcategory of $\mathrm{Fun}(\mathrm{Sp}_{T(n)},\mathrm{Sp}_{T(n)})$ on functors that preserve $K(n)$-equivalences. Lemma \ref{lem:KnPhiTheta} states that $\Phi\Theta$ is contained in this subcategory. The simplicial set $\mathrm{Fun}_{K(n)}(\mathrm{Sp}_{T(n)},\mathrm{Sp}_{T(n)})$ is in fact a simplicial monoid, using composition of functors. Moreover, $\mathrm{Sp}_{T(n)}$ is a left module for this monoid. This situation can thus be described as an algebra for the operad $\mathbf{LM}$ in the category of simplicial sets; passing to nerves, this algebra classifies a coCartesian fibration of $\infty$-operads which we denote $\mathcal{N}^\otimes \rightarrow \mathcal{LM}^\otimes$. Its fiber over the vertex $a$ is $\mathrm{Fun}_{K(n)}(\mathrm{Sp}_{T(n)},\mathrm{Sp}_{T(n)})$, and the fiber over $m$ is $\mathrm{Sp}_{T(n)}$. As in Definition 4.2.1.19 of \cite{higheralgebra}, one says that $\mathcal{N}^\otimes$ exhibits $\mathrm{Sp}_{T(n)}$ as left tensored over $\mathrm{Fun}_{K(n)}(\mathrm{Sp}_{T(n)},\mathrm{Sp}_{T(n)})$. By Definition 4.2.1.13, the $\infty$-category $\mathrm{LMod}_{\Phi\Theta}(\mathrm{Sp}_{T(n)})$ is the $\infty$-category of sections (preserving inert morphisms) of the fibration $\mathcal{N}^\otimes \rightarrow \mathcal{LM}^\otimes$ that restrict to the associative monoid $\Phi\Theta$ along the inclusion of the associative $\infty$-operad into $\mathcal{LM}^\otimes$.

Now consider the full subcategory $\mathrm{Sp}_{K(n)} \rightarrow \mathrm{Sp}_{T(n)}$. There is a corresponding full subcategory $\mathcal{N}_{K(n)}^\otimes \subseteq \mathcal{N}^\otimes$; an object $C_1 \oplus \cdots \oplus C_n$ of $\mathcal{N}^\otimes$ lies in $\mathcal{N}_{K(n)}^\otimes$ if all the $C_i$ that lie over $m \in \mathcal{LM}^\otimes$ are contained in $\mathrm{Sp}_{K(n)}$ (cf. the definition at the start of Section 2.2.1 of \cite{higheralgebra}). We consider the two localization functors $L_{K(n)}\colon \mathrm{Sp}_{T(n)} \rightarrow \mathrm{Sp}_{K(n)}$ and (trivially) the identity functor of $\mathrm{Fun}_{K(n)}(\mathrm{Sp}_{T(n)},\mathrm{Sp}_{T(n)})$. These localizations are compatible with the $\mathcal{LM}^\otimes$-monoidal structure of $\mathcal{N}^\otimes$ in the sense of Definition 2.2.1.6 of \cite{higheralgebra}. Indeed, this statement is simply the fact that for $F \in \mathrm{Fun}_{K(n)}(\mathrm{Sp}_{T(n)},\mathrm{Sp}_{T(n)})$, the map $F(E) \rightarrow F(L_{K(n)} E)$ is a $K(n)$-equivalence for every $T(n)$-local spectrum $E$. Proposition 2.2.1.9 of \cite{higheralgebra} then states that there is a morphism of $\infty$-operads
\[
\begin{tikzcd}
\mathcal{N}^\otimes \ar{rr}{L_{K(n)}^\otimes} \ar{dr} && \mathcal{N}^\otimes_{K(n)} \ar{dl} \\
& \mathcal{LM}^\otimes &
\end{tikzcd}
\]
such that $L_{K(n)}^\otimes$ is left adjoint to the inclusion and restricts to the functor $L_{K(n)}$ on $\mathrm{Sp}_{T(n)}$. This morphism $L_{K(n)}^\otimes$ gives the desired left adjoint functor
\begin{equation*}
\mathrm{Lie}(\mathrm{Sp}_{T(n)}) = \mathrm{LMod}_{\Phi\Theta}(\mathrm{Sp}_{T(n)}) \rightarrow \mathrm{LMod}_{\Phi\Theta}(\mathrm{Sp}_{K(n)}) = \mathrm{Lie}(\mathrm{Sp}_{K(n)})
\end{equation*}
by applying it to sections.
\end{proof}

\begin{remark}
\label{rmk:KnLiealgs}
One can also argue that $L_{K(n)}\Phi\Theta$ has the structure of a monad on the $\infty$-category $\mathrm{Sp}_{K(n)}$ of $K(n)$-local spectra and that $\mathrm{Lie}(\mathrm{Sp}_{K(n)})$ can equivalently be described as the $\infty$-category of algebras for this monad. Indeed, one uses that the $\infty$-category $\mathrm{Fun}(\mathrm{Sp}_{K(n)}, \mathrm{Sp}_{K(n)})$ is a (monoidal) localization of the $\infty$-category $\mathrm{Fun}_{K(n)}(\mathrm{Sp}_{T(n)}, \mathrm{Sp}_{T(n)})$ used in the preceding proof to construct a further morphism of $\infty$-operads from $\mathcal{N}_{K(n)}^\otimes$ to the $\infty$-operad which exhibits $\mathrm{Sp}_{K(n)}$ as left-tensored over $\mathrm{Fun}(\mathrm{Sp}_{K(n)}, \mathrm{Sp}_{K(n)})$. This morphism will induce an equivalence of $\infty$-categories between $\mathrm{LMod}_{\Phi\Theta}(\mathrm{Sp}_{K(n)})$ and $\mathrm{LMod}_{L_{K(n)}\Phi\Theta}(\mathrm{Sp}_{K(n)})$.
\end{remark}

\subsection{The Whitehead bracket}
\label{subsec:whiteheadbracket}

The aim of this short section is to make sense of Remark \ref{rmk:whiteheadbracket}, relating the classical Whitehead bracket to the Lie algebra structure on the Bousfield--Kuhn functor we have established. Essentially the relation will follow from applying Goodwillie calculus to the Hilton--Milnor theorem. Consider two pointed connected spaces $X$ and $Y$. The Hilton--Milnor theorem provides a certain splitting of the space $\Omega\Sigma(X \vee Y)$. To state it, write $B_2 = \{x,y,[x,y],\ldots\}$ for an ordered set of Lie words in the symbols $x$ and $y$ forming a basis for the free Lie algebra on two generators $x$ and $y$. For every such word $w$, define a space $w(X,Y)$ by letting the bracket act as the smash: for example, a word $w = [x,y]$ gives $w(X,Y) = X \wedge Y$. The Whitehead bracket (or Samelson product) corresponding to $w$ defines a map
\begin{equation*}
w(X,Y) \rightarrow \Omega\Sigma(X \vee Y),
\end{equation*}
which canonically extends to a loop map
\begin{equation*}
\Omega\Sigma w(X,Y) \rightarrow \Omega\Sigma(X \vee Y).
\end{equation*}
The theorem of Hilton--Milnor now states that multiplying these maps, for $w$ ranging through $B_2$, defines an equivalence from the restricted product (meaning the filtered colimit of finite products) over the set $B_2$ of the spaces $\Omega\Sigma w(X,Y)$ to the space $\Omega\Sigma(X \vee Y)$. Applying the Bousfield--Kuhn functor and using that it preserves products and filtered colimits, one finds a decomposition
\begin{equation*}
\bigoplus_{w \in B_2} \Phi(\Sigma w(X,Y)) \xrightarrow{\simeq} \Phi(\Sigma(X \vee Y)).
\end{equation*}
Let us write $\iota_w$ for the restriction of this map to the summand $\Phi(\Sigma w(X,Y))$ corresponding to $w$. Now suppose $\alpha, \beta \in \pi_* \Phi(X)$ are represented by maps 
\begin{equation*}
\mathbb{S}^{a+1} \xrightarrow{\alpha} \Phi(X), \quad \mathbb{S}^{b+1} \xrightarrow{\beta} \Phi(X).
\end{equation*}
By adjunction (and suppressing $T(n)$-localizations from the notation), these determine a map
\begin{equation*}
\Theta(\mathbb{S}^{a+1} \oplus \mathbb{S}^{b+1}) \xrightarrow{\hat{\alpha} + \hat{\beta}} X.
\end{equation*}
We define the \emph{Whitehead bracket} $[\alpha,\beta] \in \pi_{a+b+1} \Phi(X)$ to be the class represented by the following composite, with $w = [x,y]$:
\begin{equation*}
\mathbb{S}^{a+b+1} = \Sigma w(\mathbb{S}^a,\mathbb{S}^b) \xrightarrow{\iota_w} \Phi\Theta(\mathbb{S}^{a+1} \oplus \mathbb{S}^{b+1}) \xrightarrow{\Phi(\hat{\alpha} + \hat{\beta})} \Phi(X).
\end{equation*}
The Lie algebra structure of the functor $\Phi$ induces another product on homotopy groups. Indeed, identifying $\partial_2\mathrm{id}$ with $\mathbb{S}^{-1}$, the algebra structure of any $L \in \mathrm{Lie}(\mathrm{Sp}_{T(n)})$ in particular gives a map
\begin{equation*}
\mu\colon \mathbb{S}^{-1} \otimes L^{\otimes 2} \rightarrow (\mathbb{S}^{-1} \otimes L^{\otimes 2})_{h\Sigma_2} \rightarrow \Phi\Theta(L) \rightarrow L.
\end{equation*}
Here the second arrow is the inclusion of the degree 2 part of the free Lie algebra functor $\Phi\Theta$ (cf. Theorem \ref{thm:PhiThetafunctor}). The maps $\alpha$ and $\beta$ now give rise to a composite 
\begin{equation*}
\mathbb{S}^{a+b+1} = \mathbb{S}^{-1} \otimes \mathbb{S}^{a+1} \otimes \mathbb{S}^{b+1} \xrightarrow{\mathbb{S}^{-1} \otimes \alpha \otimes \beta} \mathbb{S}^{-1} \otimes (\Phi X)^{\otimes 2} \xrightarrow{\mu} \Phi X
\end{equation*}
which deserves to be called the \emph{Lie bracket} of the two.

Clearly there is nothing particular about $\mathbb{S}^a$ and $\mathbb{S}^b$ here; one might as well use general $T(n)$-local spectra $A$ and $B$ with maps $\alpha\colon \Sigma A \rightarrow \Phi X$ and $\beta\colon \Sigma B \rightarrow \Phi X$, in which case both procedures defined above then give a `bracket' $\Sigma A \otimes B \rightarrow \Phi X$. To show that the two operations agree, it suffices to treat the universal case where $X = \Theta(\Sigma(A \oplus B))$ and $\alpha$ and $\beta$ are adjoint to the inclusions $i_A\colon \Theta(\Sigma A) \rightarrow X$ and $i_B\colon \Theta(\Sigma B) \rightarrow X$ respectively. The brackets constructed above then yield two natural transformations
\begin{equation*}
\Sigma A \otimes B \rightarrow \Phi\Theta(\Sigma(A \oplus B))
\end{equation*}
between two-variable functors $\mathrm{Sp}_{T(n)} \times \mathrm{Sp}_{T(n)} \rightarrow \mathrm{Sp}_{T(n)}$. Let us call them $B_W$ and $B_L$ respectively.

\begin{proposition}
\label{prop:whiteheadbracket}
The natural transformations $B_W$ and $B_L$, corresponding to the Whitehead and Lie brackets respectively, agree up to a natural automorphism of the pair $(A,B)$, i.e., an automorphism of the identity functor of $\mathrm{Sp}_{T(n)} \times \mathrm{Sp}_{T(n)}$.
\end{proposition}
\begin{remark}
The indeterminacy in the statement of the proposition stems from the fact that we did not precisely specify the various identifications made in the definitions of the two products. With more care this can be avoided.
\end{remark}
\begin{proof}
Write $F(A,B) = \Sigma A \otimes B$ and $G(A,B) = \Phi\Theta(\Sigma(A \oplus B))$. The point is that $B_W$ and $B_L$ admit the same characterization in terms of (multivariable) Goodwillie calculus (as in \cite{aronekankaanrinta2,brantnerheuts}).  Indeed, both exhibit $F$ as $D_{1,1}G$, the linearization of $G$ in both variables separately. For the Whitehead bracket, defined in terms of the Hilton--Milnor splitting, this follows from Lemma 2.4 of \cite{brantnerheuts}. Indeed, that result shows that only the term corresponding to the word $w = [x,y]$ can contribute to $D_{1,1} G$ and that moreover the inclusion $\iota_w\colon F(A,B) \rightarrow G(A,B)$ induces an equivalence
\begin{equation*}
F(A,B) = \Sigma A \otimes B \xrightarrow{\simeq} D_1(\Phi)(\Sigma w(\Theta A,\Theta B)) \simeq D_{1,1} G(A,B).
\end{equation*}
For the Lie bracket, one first applies \cite[Lemma 2.1, Remark 2.2]{brantnerheuts} to see that the inclusion
\begin{equation*}
D_2 (\Phi\Theta)(\Sigma(A \oplus B)) \rightarrow \Phi\Theta(\Sigma(A \oplus B)) = G(A,B)
\end{equation*} 
induces an equivalence on $(1,1)$-derivatives. Secondly, it is clear that the map
\begin{equation*}
\Sigma A \otimes B \rightarrow D_2 (\Phi\Theta)(\Sigma(A \oplus B)) = (\mathbb{S}^{-1} \otimes (\Sigma(A \oplus B))^{\otimes 2})_{h\Sigma_2}
\end{equation*}
used to define the Lie bracket is precisely the $(1,1)$-derivative of the right-hand side.
\end{proof}

\section{The Goodwillie tower of $\mathcal{S}_{v_n}$}
\label{sec:MnfGoodwillie}

The main goal of this section is to explain the statement and give a proof of Theorem \ref{thm:GoodwillietowerMnf}. First we discuss the notion of a Goodwillie tower associated to a suitable $\infty$-category developed in \cite{heutsgoodwillie}. If $\mathcal{C}$ is a pointed compactly generated $\infty$-category (such as $\mathcal{S}_*$ or $\mathcal{S}_{v_n}$) then it admits a Goodwillie tower, which is a tower of pointed compactly generated $\infty$-categories
\[
\begin{tikzcd}
\mathcal{C} \ar[dr,"{\Sigma^\infty_3}" description] \ar[drr, bend left = 4, "{\Sigma^\infty_2}" description] \ar[drrr, bend left = 10, "{\Sigma^\infty_1}" description] &&& \\
\cdots \ar{r} & \mathcal{P}_3\mathcal{C} \ar{r} & \mathcal{P}_2\mathcal{C} \ar{r} & \mathcal{P}_1\mathcal{C}
\end{tikzcd}
\]
in which all functors are left adjoints. The first approximation $\mathcal{P}_1\mathcal{C}$ can be identified with the stabilization $\mathrm{Sp}(\mathcal{C})$ of $\mathcal{C}$. Writing $\Omega^\infty_k$ for the right adjoint of $\Sigma^\infty_k$, this tower enjoys the following properties:
\begin{itemize}
\item[(i)] The identity functor $\mathrm{id}_{\mathcal{P}_k\mathcal{C}}$ of $\mathcal{P}_k\mathcal{C}$ is a $k$-excisive functor.
\item[(ii)] The natural transformation $\Sigma^\infty_k\Omega^\infty_k \rightarrow \mathrm{id}_{\mathcal{P}_k\mathcal{C}}$ induced by the counit of the adjoint pair $(\Sigma^\infty_k, \Omega^\infty_k)$ gives an equivalence 
\begin{equation*}
P_k(\Sigma^\infty_k\Omega^\infty_k) \rightarrow \mathrm{id}_{\mathcal{P}_k\mathcal{C}}.
\end{equation*}
\item[(iii)] The functor $\Omega^\infty_k\Sigma^\infty_k$ is $k$-excisive and the unit $\mathrm{id}_{\mathcal{C}} \rightarrow \Omega^\infty_k\Sigma^\infty_k$ gives an equivalence
\begin{equation*}
P_k\mathrm{id}_{\mathcal{C}} \rightarrow \Omega^\infty_k\Sigma^\infty_k.
\end{equation*}
\end{itemize}
In fact, property (i) can be strengthened to say that $\mathcal{P}_k\mathcal{C}$ is a \emph{$k$-excisive $\infty$-category} (see Definition 2.3 of \cite{heutsgoodwillie}). The details of this property will not concern us here, but it might be helpful to know that an $\infty$-category is 1-excisive precisely if it is stable. The functors $\Sigma_k^\infty\colon \mathcal{C} \rightarrow \mathcal{P}_k\mathcal{C}$ have a universal property expressed in Theorem 2.7 of \cite{heutsgoodwillie}. In particular, the Goodwillie tower of $\mathcal{C}$ is unique up to canonical equivalence. 

The results of \cite{heutsgoodwillie} provide an explicit description of $\mathcal{P}_k\mathcal{C}$ in terms of ($k$-truncated) \emph{Tate coalgebras} in $\mathrm{Sp}(\mathcal{C})$. In the case of pointed spaces, i.e. $\mathcal{C} = \mathcal{S}_*$, a Tate coalgebra in $\mathrm{Sp}$ is roughly a commutative coalgebra $X$ for which the comultiplication maps
\begin{equation*}
\delta_k\colon X \rightarrow (X^{\otimes k})^{h\Sigma_k}
\end{equation*}
are required to be compatible with certain \emph{Tate diagonals}
\begin{equation*}
\tau_k\colon X \rightarrow (X^{\otimes k})^{t\Sigma_k}.
\end{equation*}
This leads to an equivalence between the $\infty$-category of simply-connected pointed spaces with the $\infty$-category of simply-connected Tate coalgebras in $\mathrm{Sp}$. The point here is not to give a detailed description of these Tate coalgebras, but rather to note that the situation simplifies drastically in the $T(n)$-local setting. There all Tate constructions for finite groups are contractible, meaning that there is no essential distinction between Tate coalgebras and `ordinary' commutative coalgebras. In Section \ref{sec:GoodwillietowerMnf} we will prove Theorem \ref{thm:GoodwillietowerMnf}, which describes the Goodwillie tower of $\mathcal{S}_{v_n}$ in terms of $\mathrm{coAlg}^{\mathrm{ind}}(\mathrm{Sp}^\otimes_{T(n)})$, the $\infty$-category of commutative ind-coalgebras in $T(n)$-local spectra (see Definition \ref{def:indcoalgebras}). 

The Goodwillie tower of a pointed compactly generated $\infty$-category $\mathcal{C}$ in particular gives a functor
\begin{equation*}
\mathcal{C} \rightarrow \varprojlim_k \mathcal{P}_k\mathcal{C}, 
\end{equation*}
which is generally not an equivalence of $\infty$-categories. However, it is fully faithful when restricted to the full subcategory $\mathcal{C}^{\mathrm{conv}} \subseteq \mathcal{C}$ on objects $X$ for which the Goodwillie tower of the identity converges, in the sense that the canonical map
\begin{equation*}
X \rightarrow \varprojlim_k P_k\mathrm{id}_{\mathcal{C}}(X)
\end{equation*}
is an equivalence. This means that the Goodwillie tower of the $\infty$-category $\mathcal{S}_{v_n}$ is a potentially useful tool in studying the spaces $X \in \mathcal{S}_{v_n}$ which are $\Phi$-good. In Section \ref{sec:GoodwillieBK} we will demonstrate that this is indeed the case by proving Theorem \ref{thm:BR} and Corollary \ref{cor:BR}.

\subsection{Stabilizations and Goodwillie towers}
\label{sec:GoodwillietowerMnf}

The main result of \cite{heutsgoodwillie} is a classification of Goodwillie towers of $\infty$-categories. To explain what we need of this classification, we begin with a digression on stable $\infty$-operads. We will use Lurie's version of the theory of $\infty$-operads, developed in \cite{higheralgebra}. His $\infty$-operads are certain fibrations $\mathcal{O}^\otimes \rightarrow N\mathcal{F}\mathrm{in}_*$, enjoying properties similar to those of (the nerve of) the category of operators associated to a simplicial operad. The $\infty$-operads we need are \emph{nonunital}, meaning that their structure map factors through $N\mathrm{Surj} \subseteq N\mathcal{F}\mathrm{in}_*$, with $\mathrm{Surj}$ denoting the category of finite pointed sets and surjections. The following definition is used in \cite{heutsgoodwillie}:

\begin{definition}
A nonunital $\infty$-operad $p\colon \mathcal{O}^\otimes \rightarrow N\mathrm{Surj}$ is \emph{stable} if it satisfies the following conditions:
\begin{itemize}
\item[(1)] It is corepresentable, meaning that the map $p$ is a locally coCartesian fibration. Equivalently, for every nonempty collection $X_1, \ldots, X_n$ of objects of $\mathcal{O}$ the functor
\begin{equation*}
\mathcal{O}^\otimes(X_1, \ldots, X_n; -)\colon \mathcal{O} \rightarrow \mathcal{S},
\end{equation*}
parametrizing operations of $\mathcal{O}$ with the given inputs, is corepresentable. We will denote the corepresenting object by $X_1 \otimes \cdots \otimes X_n$. This determines for every nonempty finite set $I$ a functor
\begin{equation*}
\otimes^I\colon \mathcal{O}^I \rightarrow \mathcal{O}\colon \{X_i\}_{i \in I} \mapsto \otimes^I\{X_i\}_{i \in I}.
\end{equation*}
\item[(2)] Its underlying $\infty$-category $\mathcal{O}$ is stable and compactly generated.
\item[(3)] For every finite set $I$, the functor $\otimes^I$ preserves colimits in each of its variables separately.
\end{itemize}
\end{definition}

As the tensor product notation already suggests, one can produce examples of stable $\infty$-operads from symmetric monoidal stable $\infty$-categories. More precisely, if $p\colon \mathcal{C}^\otimes \rightarrow N\mathcal{F}\mathrm{in}_*$ is a stable compactly generated symmetric monoidal $\infty$-category for which the tensor product commutes with colimits in each variable separately, one obtains a stable $\infty$-operad by pulling back the map $p$ along the inclusion $N\mathrm{Surj} \rightarrow N\mathcal{F}\mathrm{in}_*$. However, not all stable $\infty$-operads arise in this way. Generally, the definition forces the `tensor products' $\otimes^I$ to be associative only in a weak sense; see Remark 2.11 of \cite{heutsgoodwillie} and Section 6.3 of \cite{higheralgebra} for more discussion of this point. In this paper we will only be concerned with the special class of examples coming from symmetric monoidal $\infty$-categories.

Now consider a pointed compactly generated $\infty$-category $\mathcal{C}$. One can form a symmetric monoidal $\infty$-category $\mathcal{C}^\times \rightarrow N\mathcal{F}\mathrm{in}_*$ by using the Cartesian product as monoidal structure (cf. Section 2.4.1 of \cite{higheralgebra}). Pulling back to $N\mathrm{Surj}$ gives a nonunital $\infty$-operad for which we write $\mathcal{C}_{\mathrm{nu}}^\times$. Lurie proves that this $\infty$-operad admits a \emph{stabilization} $\mathrm{Sp}(\mathcal{C})^\otimes$ (cf. Section 6.2.5 of \cite{higheralgebra}). For our purposes, this is another nonunital $\infty$-operad equipped with a map of $\infty$-operads $\Omega^{\otimes,\infty}_{\mathcal{C}}\colon \mathrm{Sp}(\mathcal{C})^\otimes \rightarrow \mathcal{C}^{\times}_{\mathrm{nu}}$ satisfying the following properties:
\begin{itemize}
\item[(1)] The $\infty$-operad $\mathrm{Sp}(\mathcal{C})^\otimes$ is stable.
\item[(2)] The map $\Omega^{\otimes,\infty}_{\mathcal{C}}$ coincides with the functor $\Omega^\infty_{\mathcal{C}}\colon \mathrm{Sp}(\mathcal{C}) \rightarrow \mathcal{C}$ on underlying $\infty$-categories.
\item[(3)] The natural transformation 
\begin{equation*}
\times^I \circ (\Omega^\infty_{\mathcal{C}})^I \rightarrow \Omega^\infty_{\mathcal{C}} \circ \otimes^I
\end{equation*}
induced by the map of $\infty$-operads $\Omega^{\otimes,\infty}_{\mathcal{C}}$ exhibits $\otimes^I$ as the multilinearization of $\times^I$.
\end{itemize}
The example to keep in mind is $\mathcal{C} = \mathcal{S}_*$, in which case $\mathrm{Sp}(\mathcal{C})^\otimes$ is the symmetric monoidal $\infty$-category $\mathrm{Sp}^\otimes$ of spectra with their smash product. Generally, one can show that the functor
\begin{equation*}
\otimes^k\colon \mathrm{Sp}(\mathcal{C})^k \rightarrow \mathrm{Sp}(\mathcal{C})
\end{equation*}
induced by the stable $\infty$-operad $\mathrm{Sp}(\mathcal{C})^\otimes$ gives the $k$th derivative of $\Sigma^\infty_{\mathcal{C}}\Omega^\infty_{\mathcal{C}}$; more precisely
\begin{equation*}
D_k(\Sigma^\infty_{\mathcal{C}}\Omega^\infty_{\mathcal{C}})(X) \simeq (X^{\otimes k})_{h\Sigma_k}.
\end{equation*}
One can think of the stable $\infty$-operad $\mathrm{Sp}(\mathcal{C})^\otimes$ as capturing these derivatives together with their cooperadic structure. See Section 6.3 of \cite{higheralgebra} or Section 2 of \cite{heutsgoodwillie} for more discussion of this correspondence.

The $\infty$-category $\mathcal{P}_k\mathcal{C}$, for $k \geq 1$, can now be described in terms of certain $k$-truncated Tate coalgebras in the stable $\infty$-operad $\mathrm{Sp}(\mathcal{C})^\otimes$. Such a coalgebra is a spectrum $X \in \mathrm{Sp}(\mathcal{C})$ equipped with comultiplication maps
\begin{equation*}
\delta_j\colon X \rightarrow (X^{\otimes j})^{h\Sigma_j}
\end{equation*}
for $j \leq k$, together with a coherent system of homotopies expressing the associativity of these comultiplications (we will give a precise definition below in the case of interest to us here). Moreover, these maps should be compatible with certain `Tate diagonals'
\begin{equation*}
\tau_j\colon X \rightarrow (X^{\otimes j})^{t\Sigma_j}
\end{equation*}
which are determined by $\mathcal{C}$. In this paper we will only be concerned with stable $\infty$-categories in which all Tate spectra for finite groups are null. Then Corollary 2.23 of \cite{heutsgoodwillie} gives the following:

\begin{proposition}
\label{prop:nstages}
Let $\mathcal{C}$ be a pointed compactly generated $\infty$-category for which all Tate spectra associated to the symmetric groups are null in $\mathrm{Sp}(\mathcal{C})$. Then for every $k \geq 1$ there is a canonical equivalence of $\infty$-categories
\begin{equation*}
\mathcal{P}_k\mathcal{C} \simeq \mathrm{coAlg}^{\mathrm{ind}}(\tau_k\mathrm{Sp}(\mathcal{C}^\otimes)),
\end{equation*}
where $\mathrm{coAlg}^{\mathrm{ind}}(\tau_k\mathrm{Sp}(\mathcal{C})^\otimes)$ denotes the $\infty$-category of ind-coalgebras in $\tau_k\mathrm{Sp}(\mathcal{C})^\otimes$ (Definition 5.21 of \cite{heutsgoodwillie}, see Definitions \ref{def:coalgebras} and \ref{def:indcoalgebras} below).
\end{proposition}

In particular, under the conditions of the previous proposition the Goodwillie tower of $\mathcal{C}$ is completely determined by $\mathrm{Sp}(\mathcal{C})^\otimes$. For us the relevant stable $\infty$-operad will be $\mathrm{Sp}_{T(n)}^\otimes$, the symmetric monoidal $\infty$-category of $T(n)$-local spectra with smash product. In this case the definition of ind-coalgebras of \cite{heutsgoodwillie} is unnecessarily complicated and one may use the following instead. We will compare the two definitions in Appendix \ref{app:coalgebras}. All our coalgebras will be \emph{without} counits; we will not include explicit mention of this in our terminology or notation:

\begin{definition}
\label{def:coalgebras}
The $\infty$-category $\mathrm{coAlg}(\mathrm{Sp}_{T(n)}^\otimes)$ of commutative coalgebras in $T(n)$-local spectra is the opposite of the $\infty$-category of nonunital commutative algebra objects in the symmetric monoidal $\infty$-category $(\mathrm{Sp}_{T(n)}^{\mathrm{op}})^{\otimes}$. Recall that a nonunital commutative algebra object of a symmetric monoidal $\infty$-category $\mathcal{C}^\otimes$ is a section of the structure map $\mathcal{C}^\otimes \xrightarrow{p} N\mathcal{F}\mathrm{in}_*$ defined over $N\mathrm{Surj}$, preserving inert morphisms (cf. Definitions 2.1.2.7 and 5.4.4.1 of \cite{higheralgebra}):
\[
\begin{tikzcd}
& \mathcal{C}^\otimes \ar[d,"p"] \\
N\mathrm{Surj} \ar[r]\ar[ur, dashed] & N\mathcal{F}\mathrm{in}_*.
\end{tikzcd}
\]
Similarly, a \emph{$k$-truncated} nonunital commutative algebra object is such a section defined over $N\mathrm{Surj}_{\leq k}$, with $\mathrm{Surj}_{\leq k}$ the full subcategory of $\mathrm{Surj}$ on finite sets of cardinality at most $k$:
\[
\begin{tikzcd}
& \mathcal{C}^\otimes \ar[d,"p"] \\
N\mathrm{Surj}_{\leq k} \ar[r]\ar[ur, dashed] & N\mathcal{F}\mathrm{in}_*.
\end{tikzcd}
\]
Finally, the $\infty$-category $\mathrm{coAlg}(\tau_k\mathrm{Sp}_{T(n)}^\otimes)$ of $k$-truncated commutative coalgebras in $T(n)$-local spectra is the opposite of the $\infty$-category of $k$-truncated nonunital commutative algebra objects in the symmetric monoidal $\infty$-category $(\mathrm{Sp}_{T(n)}^{\mathrm{op}})^{\otimes}$.
\end{definition}

\begin{definition}
\label{def:indcoalgebras}
The $\infty$-category of commutative ind-coalgebras in $T(n)$-local spectra is defined by
\begin{equation*}
\mathrm{coAlg}^{\mathrm{ind}}(\mathrm{Sp}_{T(n)}^\otimes) := \mathrm{Ind}(\mathrm{coAlg}^{\mathrm{fin}}(\mathrm{Sp}_{T(n)}^\otimes)),
\end{equation*}
where $\mathrm{coAlg}^{\mathrm{fin}}(\mathrm{Sp}_{T(n)}^\otimes)$ is the full subcategory of $\mathrm{coAlg}(\mathrm{Sp}_{T(n)}^\otimes)$ on coalgebras whose underlying spectrum is a compact object of $\mathrm{Sp}_{T(n)}$. The $\infty$-category of $k$-truncated commutative ind-coalgebras in $T(n)$-local spectra is defined similarly:
\begin{equation*}
\mathrm{coAlg}^{\mathrm{ind}}(\tau_k\mathrm{Sp}_{T(n)}^\otimes) := \mathrm{Ind}(\mathrm{coAlg}^{\mathrm{fin}}(\tau_k\mathrm{Sp}_{T(n)}^\otimes)).
\end{equation*}
\end{definition}

\begin{remark}
\label{rmk:truncation}
It should be noted that there are evident `truncation functors'
\begin{equation*}
\mathrm{coAlg}^{\mathrm{ind}}(\mathrm{Sp}_{T(n)}^\otimes) \rightarrow \mathrm{coAlg}^{\mathrm{ind}}(\tau_k\mathrm{Sp}_{T(n)}^\otimes)
\end{equation*}
obtained by precomposing with the inclusion $N\mathrm{Surj}_{\leq k} \rightarrow N\mathrm{Surj}$. These functors are the identity on underlying objects $X$, but they simply forget the part of the coalgebraic structure concerning comultiplication maps
\begin{equation*}
\delta_j\colon X \rightarrow (X^{\otimes j})^{h\Sigma_j}
\end{equation*}
for $j > k$.
\end{remark}

\begin{remark}
\label{rmk:indvsnoind}
Introducing ind-coalgebras, as opposed to just coalgebras, is a necessary complication in \cite{heutsgoodwillie}, since the methods of Goodwillie towers apply to compactly generated $\infty$-categories only. The $\infty$-category $\mathrm{coAlg}(\mathrm{Sp}_{T(n)}^\otimes)$ is presentable (see Corollary 3.1.4 of \cite{lurieelliptic}), so the inclusion
\begin{equation*}
\mathrm{coAlg}^{\mathrm{fin}}(\mathrm{Sp}_{T(n)}^\otimes) \rightarrow \mathrm{coAlg}(\mathrm{Sp}_{T(n)}^\otimes)
\end{equation*}
gives by left Kan extension a colimit-preserving functor
\begin{equation*}
j\colon \mathrm{coAlg}^{\mathrm{ind}}(\mathrm{Sp}_{T(n)}^\otimes) \rightarrow \mathrm{coAlg}(\mathrm{Sp}_{T(n)}^\otimes).
\end{equation*}
However, there seems to be no reason why a finite coalgebra $X$ should be a compact object of the latter $\infty$-category, so that this functor need not be fully faithful. (In fact it is doubtful that the $\infty$-category of coalgebras in $T(n)$-local spectra has any nonzero compact objects at all.) The situation is much better in the truncated setting: a truncated coalgebra $X \in \mathrm{coAlg}^{\mathrm{fin}}(\tau_k\mathrm{Sp}_{T(n)})$ is a compact object of $\mathrm{coAlg}(\tau_k\mathrm{Sp}_{T(n)})$ and hence the corresponding functor
\begin{equation*}
j: \mathrm{coAlg}^{\mathrm{ind}}(\tau_k\mathrm{Sp}_{T(n)}^\otimes) \rightarrow \mathrm{coAlg}(\tau_k\mathrm{Sp}_{T(n)}^\otimes)
\end{equation*}
\emph{is} fully faithful (Corollary 6.15 of \cite{heutsgoodwillie}).
\end{remark}

In order to apply Proposition \ref{prop:nstages} we need to identify the stable $\infty$-operad $\mathrm{Sp}(\mathcal{S}_{v_n})^\otimes$. First of all, there is the morphism of $\infty$-operads
\begin{equation*}
\Omega^{\otimes, \infty}\colon \mathrm{Sp}^\otimes \rightarrow \mathcal{S}_*^{\times}
\end{equation*}
which exhibits $\mathrm{Sp}^{\otimes}$ as the stabilization of $\mathcal{S}_*^{\times}$. We could also replace $\mathcal{S}_*$ by $\mathcal{S}_*\langle d_{n+1} \rangle$ and take the $d_{n+1}$-connected cover of $\Omega^\infty$, because taking $d_{n+1}$-connected covers is a product-preserving functor. We can now restrict $\Omega^{\otimes, \infty}$ to the full symmetric monoidal subcategory of spectra which are $L_n^f$-local and obtain a morphism of $\infty$-operads
\begin{equation*}
\Omega^{\otimes, \infty}_{L_n^f}\colon L_n^f\mathrm{Sp}^\otimes \rightarrow (\mathcal{L}_n^f)^{\times}.
\end{equation*}
Recall the right adjoint functor $r\colon \mathcal{L}_n^f \rightarrow \mathcal{S}_{v_n}$ which, because it preserves products, induces a morphism
\begin{equation*}
r^\times\colon (\mathcal{L}_n^f)^{\times} \rightarrow (\mathcal{S}_{v_n})^{\times}.
\end{equation*}
Restricting $\Omega^{\otimes, \infty}_{L_n^f}$ further and postcomposing with $r^\times$ gives a morphism
\begin{equation*}
\Omega^{\otimes, \infty}_{M_n^f}\colon M_n^f\mathrm{Sp}^\otimes \rightarrow (\mathcal{S}_{v_n})^{\times}.
\end{equation*}
This expression $M_n^f\mathrm{Sp}^\otimes$ makes sense because $M_n^f$ is closed under smash products in $L_n^f$, which is immediate from the fact that $L_n^f$ is a smashing localization.

\begin{proposition}
\label{prop:SpMnf}
The morphism $\Omega^{\otimes,\infty}_{M_n^f}\colon M_n^f\mathrm{Sp}^\otimes \rightarrow (\mathcal{S}_{v_n})^\times$ exhibits $M_n^f\mathrm{Sp}^\otimes$ as the stabilization of $(\mathcal{S}_{v_n})^\times$. In other words, it induces an equivalence
\begin{equation*}
M_n^f\mathrm{Sp}^\otimes \simeq \mathrm{Sp}(\mathcal{S}_{v_n})^\otimes.
\end{equation*}
\end{proposition}

The result we are after follows immediately from this:

\begin{proof}[Proof of Theorem \ref{thm:GoodwillietowerMnf}]
As in Remark \ref{rmk:MnfLTn}, we are free to replace $M_n^f\mathrm{Sp}^\otimes$ by $\mathrm{Sp}_{T(n)}^\otimes$. The theorem then follows by combining Proposition \ref{prop:SpMnf} with Proposition \ref{prop:nstages}.
\end{proof}

\begin{proof}[Proof of Proposition \ref{prop:SpMnf}]
We already showed that $r\Omega^\infty\colon M_n^f\mathrm{Sp} \rightarrow \mathcal{S}_{v_n}$ exhibits $M_n^f\mathrm{Sp}$ as the stabilization of $\mathcal{S}_{v_n}$ (see Proposition \ref{prop:Mnfstabilization}). Therefore it remains to show that the natural maps 
\begin{equation*}
\times^I \circ (r\Omega^\infty)^I \rightarrow r\Omega^\infty \circ \otimes^I
\end{equation*}
induced by $\Omega^{\otimes,\infty}_{M_n^f}$ exhibit the smash product functors $\otimes^I$ on $M_n^f\mathrm{Sp}$ as the multilinearization of the product functors $\times^I$ on $\mathcal{S}_{v_n}$. The product $\times^I$ on the left is to be interpreted in the $\infty$-category $\mathcal{S}_{v_n}$, but by Lemma \ref{lem:Mnffinitelim} it can be computed simply as the usual product of pointed spaces.

To form the multilinearization of $\times^k$, one first forms its \emph{coreduction} (see Construction 6.2.3.6 of \cite{higheralgebra}), denoted $\mathrm{cored}(\times^k)$, and then linearizes this coreduction in each of its variables. The coreduction of $\times^k$ is the initial functor which is reduced in each of its variables and receives a natural transformation from $\times^k$. It can be constructed as follows:
\begin{equation*}
\mathrm{cored}(\times^k)(X_1, \ldots, X_k) \simeq L_n^f\mathrm{cof}\bigl(\varinjlim_{S \subsetneq \{1, \ldots, k\}} \prod_{i \in S} X_i \rightarrow \prod_{i = 1, \ldots, k} X_i\bigr).
\end{equation*}
Therefore
\begin{equation*}
\mathrm{cored}(\times^k)(X_1, \ldots, X_k) \simeq L_n^f(X_1 \wedge \cdots \wedge X_k).
\end{equation*}
Linearizing this expression in each of its variables gives
\begin{eqnarray*}
\varinjlim_{j_1, \ldots, j_k} \Omega^{j_1 + \cdots + j_k}L_n^f(\Sigma^{j_1}X_1 \wedge \cdots \wedge \Sigma^{j_k}X_k) & \simeq &  \varinjlim_{j_1, \ldots, j_k} L_n^f \Omega^{j_1 + \cdots + j_k}(\Sigma^{j_1}X_1 \wedge \cdots \wedge \Sigma^{j_k}X_k) \\
& \simeq & L_n^f\Omega^\infty(\Sigma^\infty X_1 \otimes \cdots \otimes \Sigma^\infty X_k).
\end{eqnarray*}
Here the first equivalence uses Theorem \ref{thm:Lnfinitelimits} and the second uses the fact that the smash product of spectra is the multilinearization of the smash product of spaces. This concludes the proof of the proposition.
\end{proof}

% the following remark has become redundant in the current version
%\begin{remark}
%\label{rmk:derivativesSigmainftyOmegainfty}
%As remarked before, the tensor products $\otimes^k_{\mathcal{C}}$ induced by the stable $\infty$-operad $\mathrm{Sp}(\mathcal{C})^\otimes$ correspond to the derivatives $D_k(\Sigma^\infty_{\mathcal{C}}\Omega^\infty_{\mathcal{C}})$ in the following sense (see Lemma B.3 of \cite{heutsgoodwillie}):
%\begin{equation*}
%D_k(\Sigma^\infty_{\mathcal{C}}\Omega^\infty_{\mathcal{C}})(X) \simeq (X^{\otimes^k_{\mathcal{C}}})_{h\Sigma_k}.
%\end{equation*}
%In particular, it follows from Proposition \ref{prop:SpMnf} that
%\begin{equation*}
%D_k(\Sigma^\infty_{T(n)}\Omega^\infty_{T(n)})(X) \simeq (X^{\otimes k})_{h\Sigma_k},
%\end{equation*}
%where the tensor product on the right now denotes the smash product of $T(n)$-local spectra.
%\end{remark}

\subsection{The Goodwillie tower of the Bousfield--Kuhn functor}
\label{sec:GoodwillieBK}

The goal of this section is to prove Theorem \ref{thm:BR} and Corollary \ref{cor:BR}. We begin with a discussion of the functor
\begin{equation*}
\mathrm{prim}\colon \mathrm{coAlg}(\mathrm{Sp}_{T(n)}^\otimes) \rightarrow \mathrm{Sp}_{T(n)}
\end{equation*}
taking the primitives of a coalgebra, as well as its variant
\begin{equation*}
\mathrm{prim}^{\mathrm{ind}}\colon \mathrm{coAlg}^{\mathrm{ind}}(\mathrm{Sp}_{T(n)}^\otimes) \rightarrow \mathrm{Sp}_{T(n)}
\end{equation*}
taking the primitives of an ind-coalgebra. We will construct the latter as the right adjoint of the trivial ind-coalgebra functor
\begin{equation*}
\mathrm{triv}^{\mathrm{ind}}\colon \mathrm{Sp}_{T(n)} \rightarrow \mathrm{coAlg}^{\mathrm{ind}}(\mathrm{Sp}_{T(n)}^\otimes).
\end{equation*}
This functor will be the left Kan extension of a corresponding functor
\begin{equation*}
\mathrm{triv}^{\mathrm{fin}}\colon \mathrm{Sp}_{T(n)}^{\mathrm{fin}} \rightarrow \mathrm{coAlg}^{\mathrm{fin}}(\mathrm{Sp}_{T(n)}^\otimes)
\end{equation*}
defined on compact objects (which are the retracts of $T(n)$-localizations of finite type $n$ spectra). Since finite spectra are dualizable with respect to the smash product, Spanier--Whitehead duality gives an equivalence of symmetric monoidal $\infty$-categories (see Section 3.2 of \cite{lurieelliptic})
\begin{equation*}
\mathbf{D}\colon \bigl(\mathrm{Sp}_{T(n)}^{\mathrm{fin}}\bigr)^{\otimes} \rightarrow \bigl((\mathrm{Sp}_{T(n)}^{\mathrm{fin}})^{\mathrm{op}}\bigr)^\otimes
\end{equation*}
and hence also an equivalence (which we denote by the same symbol)
\begin{equation*}
\mathbf{D}\colon \mathrm{coAlg}^{\mathrm{fin}}(\mathrm{Sp}_{T(n)}^\otimes) \rightarrow \bigl(\mathrm{CAlg}^{\mathrm{nu}}((\mathrm{Sp}^{\mathrm{fin}}_{T(n)})^\otimes)\bigr)^{\mathrm{op}}.
\end{equation*}
Here $\mathrm{CAlg}^{\mathrm{nu}}(\mathcal{C}^\otimes)$ denotes the $\infty$-category of nonunital commutative algebras in a symmetric monoidal $\infty$-category $\mathcal{C}$.

This $\infty$-category of nonunital commutative algebras admits a stabilization
\begin{equation*}
\Omega^\infty_{\mathrm{CAlg}}\colon \mathrm{Sp}_{T(n)}^{\mathrm{fin}} \rightarrow \mathrm{CAlg}^{\mathrm{nu}}((\mathrm{Sp}_{T(n)}^{\mathrm{fin}})^\otimes).
\end{equation*}
This follows from Theorem 7.3.4.7 of \cite{higheralgebra} for the $\infty$-category of augmented commutative algebras, which is equivalent to that of nonunital commutative algebras via the functor taking the augmentation ideal. The functor $\Omega^\infty_{\mathrm{CAlg}}$ can be thought of as equipping an object with the trivial (nonunital) commutative algebra structure. It preserves finite limits and the composition of $\Omega^\infty_{\mathrm{CAlg}}$ with the forgetful functor is naturally equivalent to the identity functor of $\mathrm{Sp}_{T(n)}^{\mathrm{fin}}$. It can be extended to a functor 
\begin{equation*}
\Omega^\infty_{\mathrm{CAlg}}\colon \mathrm{Sp}_{T(n)} \rightarrow \mathrm{CAlg}^{\mathrm{nu}}((\mathrm{Sp}_{T(n)})^\otimes)
\end{equation*}
and as such it admits a left adjoint
\begin{equation*}
\Sigma^\infty_{\mathrm{CAlg}}\colon \mathrm{CAlg}^{\mathrm{nu}}((\mathrm{Sp}_{T(n)})^\otimes) \rightarrow \mathrm{Sp}_{T(n)}.
\end{equation*}
This is the functor taking the \emph{derived indecomposables} or \emph{topological Andr\'{e}-Quillen homology} of a commutative ring spectrum. It was originally described in this fashion by Basterra and Mandell \cite{basterramandell}. Section 7.3 of \cite{higheralgebra} includes a treatment in the $\infty$-categorical setting.

As promised, we define the \emph{trivial coalgebra functor} for compact objects of $\mathrm{Sp}_{T(n)}$ to be the composition
\begin{equation*}
\mathrm{Sp}_{T(n)}^{\mathrm{fin}} \xrightarrow{\mathbf{D}} (\mathrm{Sp}_{T(n)}^{\mathrm{fin}})^{\mathrm{op}} \xrightarrow{\Omega^\infty_{\mathrm{CAlg}}} \mathrm{CAlg}^{\mathrm{nu}}((\mathrm{Sp}_{T(n)}^{\mathrm{fin}})^\otimes)^{\mathrm{op}} \xrightarrow{\mathbf{D}^{-1}} \mathrm{coAlg}^{\mathrm{fin}}(\mathrm{Sp}_{T(n)}^\otimes)
\end{equation*}
and we write 
\begin{equation*}
\mathrm{triv}^{\mathrm{ind}}\colon \mathrm{Sp}_{T(n)} \rightarrow \mathrm{coAlg}^{\mathrm{ind}}(\mathrm{Sp}_{T(n)}^\otimes)
\end{equation*}
for the left Kan extension of this functor. Since it preserves colimits and the $\infty$-category of ind-coalgebras is presentable, the adjoint functor theorem (Corollary 5.5.2.9(1) and Remark 5.5.2.10 of \cite{htt}) applies to guarantee the existence of the desired right adjoint
\begin{equation*}
\mathrm{prim}^{\mathrm{ind}}\colon \mathrm{coAlg}^{\mathrm{ind}}(\mathrm{Sp}_{T(n)}^\otimes) \rightarrow \mathrm{Sp}_{T(n)}.
\end{equation*}
Moreover, this right adjoint preserves filtered colimits by virtue of the fact that $\mathrm{triv}^{\mathrm{ind}}$ preserves compact objects, which is true by construction. Composing with the truncation functor (cf. Remark \ref{rmk:truncation})
\begin{equation*}
\mathrm{coAlg}^{\mathrm{ind}}(\mathrm{Sp}_{T(n)}^\otimes) \rightarrow \mathrm{coAlg}^{\mathrm{ind}}(\tau_k\mathrm{Sp}_{T(n)}^\otimes) 
\end{equation*}
we obtain ($k$-truncated) trivial coalgebra functors
\begin{equation*}
\mathrm{triv}^{\mathrm{ind}}_k\colon \mathrm{Sp}_{T(n)} \rightarrow \mathrm{coAlg}^{\mathrm{ind}}(\tau_k\mathrm{Sp}_{T(n)}^\otimes)
\end{equation*}
with right adjoints
\begin{equation*}
\mathrm{prim}^{\mathrm{ind}}_k\colon \mathrm{coAlg}^{\mathrm{ind}}(\tau_k\mathrm{Sp}_{T(n)}^\otimes) \rightarrow \mathrm{Sp}_{T(n)}.
\end{equation*}
Of course we could also take a left Kan extension of $\mathrm{triv}^{\mathrm{fin}}$ to a functor
\begin{equation*}
\mathrm{triv}\colon \mathrm{Sp}_{T(n)} \rightarrow \mathrm{coAlg}(\mathrm{Sp}_{T(n)}^\otimes)
\end{equation*}
which has a corresponding adjoint 
\begin{equation*}
\mathrm{prim}\colon \mathrm{coAlg}(\mathrm{Sp}_{T(n)}^\otimes) \rightarrow \mathrm{Sp}_{T(n)},
\end{equation*}
using that $\mathrm{coAlg}(\mathrm{Sp}_{T(n)}^\otimes)$ is a presentable $\infty$-category (cf. Remark \ref{rmk:indvsnoind}). Similarly one defines $\mathrm{triv}_k$ and $\mathrm{prim}_k$.

\begin{remark}
\label{rmk:indvsnoind2}
It is useful to note that in the truncated case the distinction between $\mathrm{prim}^{\mathrm{ind}}_k$ and $\mathrm{prim}_k$ is of little importance; indeed, as in Remark \ref{rmk:indvsnoind} the $\infty$-category of $k$-truncated ind-coalgebras is a full subcategory of that of $k$-truncated coalgebras via the fully faithful colimit-preserving functor
\begin{equation*}
j\colon \mathrm{coAlg}^{\mathrm{ind}}(\tau_k\mathrm{Sp}_{T(n)}^\otimes) \rightarrow \mathrm{coAlg}(\tau_k\mathrm{Sp}_{T(n)}^\otimes).
\end{equation*}
Writing $s$ for its right adjoint, we have $sj \simeq \mathrm{id}$. By construction we have $j\circ \mathrm{triv}^{\mathrm{ind}}_k = \mathrm{triv}_k$, so that by adjunction $\mathrm{prim}^{\mathrm{ind}}_k \circ s = \mathrm{prim}_k$. Thus, for a $k$-truncated ind-coalgebra $X$, we have
\begin{equation*}
\mathrm{prim}^{\mathrm{ind}}_k X \simeq \mathrm{prim}^{\mathrm{ind}}_k(sjX) \simeq \mathrm{prim}_k(jX). 
\end{equation*}
For this reason we will not distinguish in notation between $\mathrm{prim}^{\mathrm{ind}}_k$ and $\mathrm{prim}_k$ any longer and simply write the latter.
\end{remark}

\begin{remark}
The adjoint pair $(\mathrm{triv}^{\mathrm{ind}}, \mathrm{prim}^{\mathrm{ind}})$ is related to the forgetful-cofree pair by a diagram
\[
\begin{tikzcd}
\mathrm{Sp}_{T(n)} \arrow[r,"\mathrm{triv}^{\mathrm{ind}}", shift left] & \mathrm{coAlg}^{\mathrm{ind}}(\mathrm{Sp}_{T(n)}^\otimes) \arrow[r,"\mathrm{forget}", shift left] \arrow[l, "\mathrm{prim}^{\mathrm{ind}}", shift left] & \mathrm{Sp}_{T(n)} \arrow[l, "\mathrm{cofree}", shift left] 
\end{tikzcd}
\]
in which both horizontal composites are naturally equivalent to the identity functor of $\mathrm{Sp}_{T(n)}$. This is of course reminiscent of the pairs $(\Theta, \Phi)$ and $(\Sigma^\infty_{T(n)}, \Omega^\infty_{T(n)})$ with the identities $\Sigma^\infty_{T(n)}\Theta \simeq \mathrm{id}_{\mathrm{Sp}_{T(n)}}$ and $\Phi\Omega^\infty_{T(n)} \simeq \mathrm{id}_{\mathrm{Sp}_{T(n)}}$. In fact, we will demonstrate shortly that $\mathrm{triv}^{\mathrm{ind}}$ can be identified with the composite of $\Theta$ and the functor 
\begin{equation*}
\mathcal{S}_{v_n} \rightarrow \mathrm{coAlg}^{\mathrm{ind}}(\mathrm{Sp}_{T(n)}^\otimes)
\end{equation*}
assigning to $X \in \mathcal{S}_{v_n}$ its $T(n)$-local suspension spectrum with its natural ind-coalgebra structure.
\end{remark}

\begin{lemma}
\label{lem:primSWdual}
Let $X \in \mathrm{coAlg}^{\mathrm{fin}}(\mathrm{Sp}_{T(n)}^\otimes)$ be a finite coalgebra. Then $\mathrm{prim} X \simeq \mathbf{D}\Sigma^\infty_{\mathrm{CAlg}}(\mathbf{D}X)$. More precisely, there is a natural equivalence of functors between the restriction of $\mathrm{prim}$ to $\mathrm{coAlg}^{\mathrm{fin}}(\mathrm{Sp}_{T(n)}^\otimes)$ and the composition
\begin{equation*}
\mathrm{coAlg}^{\mathrm{fin}}(\mathrm{Sp}_{T(n)}^\otimes) \xrightarrow{\mathbf{D}} \mathrm{CAlg}^{\mathrm{nu}}((\mathrm{Sp}_{T(n)}^{\mathrm{fin}})^\otimes)^{\mathrm{op}} \xrightarrow{\Sigma^\infty_{\mathrm{CAlg}}} (\mathrm{Sp}_{T(n)})^{\mathrm{op}} \xrightarrow{\mathbf{D}} \mathrm{Sp}_{T(n)}.
\end{equation*}
\end{lemma}
\begin{proof}
The proof follows from straightforward manipulations of the adjunctions described above and is easily summarized by the following chain of natural equivalences, for $Y$ a compact object of $\mathrm{Sp}_{T(n)}$:
\begin{eqnarray*}
\mathrm{Map}_{\mathrm{Sp}_{T(n)}}(Y, \mathrm{prim}(X)) & \simeq & \mathrm{Map}_{\mathrm{coAlg}^{\mathrm{fin}}(\mathrm{Sp}_{T(n)}^\otimes)}(\mathrm{triv}(Y), X) \\
& \simeq & \mathrm{Map}_{ \mathrm{CAlg}^{\mathrm{nu}}((\mathrm{Sp}_{T(n)}^{\mathrm{fin}})^\otimes)}(\mathbf{D} X, \Omega^\infty_{\mathrm{CAlg}}(\mathbf{D}Y)) \\
& \simeq & \mathrm{Map}_{\mathrm{Sp}_{T(n)}}(\Sigma^\infty_{\mathrm{CAlg}}(\mathbf{D}X), \mathbf{D}Y) \\
& \simeq & \mathrm{Map}_{\mathrm{Sp}_{T(n)}}(Y, \mathbf{D}\Sigma^\infty_{\mathrm{CAlg}}(\mathbf{D}X)).
\end{eqnarray*}
\end{proof}

\begin{remark}
From the previous lemma one can also conclude that $\mathrm{prim}^{\mathrm{ind}}$ can be characterized as the best possible approximation of the functor
\begin{equation*}
X \mapsto  \mathbf{D}\Sigma^\infty_{\mathrm{CAlg}}(\mathbf{D}X)
\end{equation*}
by a functor which preserves filtered colimits.
\end{remark}

The suspension spectrum functor $\Sigma^\infty_{T(n)}\colon \mathcal{S}_{v_n} \rightarrow \mathrm{Sp}_{T(n)}$ is symmetric monoidal with respect to smash products. Every $X \in \mathcal{S}_{v_n}$ is canonically a commutative coalgebra with respect to the Cartesian product and hence also with respect to the smash product, simply using the natural map from product to smash product. Hence, we obtain a functor
\begin{equation*}
(\mathcal{S}_{v_n})^{\omega} \rightarrow \mathrm{coAlg}^{\mathrm{fin}}(\mathrm{Sp}_{T(n)}^\otimes)
\end{equation*}
which on underlying objects is simply $\Sigma^\infty_{T(n)}$. (The reader looking for a more detailed account can consult Construction 5.20 of \cite{heutsgoodwillie}.) Formally extending by filtered colimits (i.e., taking a left Kan extension) gives a functor
\begin{equation*}
C_{T(n)}\colon \mathcal{S}_{v_n} \rightarrow \mathrm{coAlg}^{\mathrm{ind}}(\mathrm{Sp}_{T(n)}^\otimes).
\end{equation*}

\begin{lemma}
\label{lem:Thetatriv}
The composite
\begin{equation*}
\mathrm{Sp}_{T(n)} \xrightarrow{\Theta} \mathcal{S}_{v_n} \xrightarrow{C_{T(n)}} \mathrm{coAlg}^{\mathrm{ind}}(\mathrm{Sp}_{T(n)}^\otimes)
\end{equation*}
is naturally equivalent to
\begin{equation*}
\mathrm{triv}^{\mathrm{ind}}\colon \mathrm{Sp}_{T(n)} \rightarrow \mathrm{coAlg}^{\mathrm{ind}}(\mathrm{Sp}_{T(n)}^\otimes).
\end{equation*}
In fact, our proof will yield a preferred equivalence.
\end{lemma}
\begin{proof}
Restricting to compact objects gives a functor
\begin{equation*}
C_{T(n)} \circ \Theta\colon \mathrm{Sp}_{T(n)}^{\mathrm{fin}} \rightarrow \mathrm{coAlg}^{\mathrm{fin}}(\mathrm{Sp}_{T(n)}^\otimes)
\end{equation*}
which preserves finite colimits. Under Spanier--Whitehead duality we obtain a functor
\begin{equation*}
\theta := (\mathbf{D} \circ C_{T(n)} \circ \Theta \circ \mathbf{D})^{\mathrm{op}}\colon \mathrm{Sp}_{T(n)}^{\mathrm{fin}} \rightarrow \mathrm{CAlg}^{\mathrm{nu}}((\mathrm{Sp}_{T(n)}^{\mathrm{fin}})^\otimes)
\end{equation*}
which preserves finite limits. Therefore $\theta$ canonically factors over the stabilization
\begin{equation*}
\Omega^\infty_{\mathrm{CAlg}}\colon \mathrm{Sp}_{T(n)}^{\mathrm{fin}} \rightarrow  \mathrm{CAlg}^{\mathrm{nu}}((\mathrm{Sp}_{T(n)}^{\mathrm{fin}})^\otimes)
\end{equation*}
i.e., we find an exact functor $\psi\colon \mathrm{Sp}_{T(n)}^{\mathrm{fin}} \rightarrow \mathrm{Sp}_{T(n)}^{\mathrm{fin}}$ and an equivalence $\theta \simeq \Omega^\infty_{\mathrm{CAlg}} \circ \psi$. We will show that $\psi$ is equivalent to the identity functor. Indeed, writing
\begin{equation*}
U\colon \mathrm{CAlg}^{\mathrm{nu}}((\mathrm{Sp}_{T(n)}^{\mathrm{fin}})^\otimes) \rightarrow \mathrm{Sp}_{T(n)}^{\mathrm{fin}}
\end{equation*}
for the forgetful functor and using the natural equivalence $U \circ \Omega^\infty_{\mathrm{CAlg}} \simeq \mathrm{id}_{\mathrm{Sp}_{T(n)}}$ gives an equivalence $U \theta \simeq \psi$. On the other hand, for $X \in \mathrm{Sp}_{T(n)}$,
\begin{eqnarray*}
U\theta(X) & = & U\mathbf{D}C_{T(n)}\Theta(\mathbf{D}X) \\
& \simeq & \mathbf{D}\Sigma^\infty_{T(n)}\Theta(\mathbf{D}X) \\
& \simeq & \mathbf{D}\mathbf{D}X \\
& \simeq & X.
\end{eqnarray*}
Here the second line uses the fact that the underlying spectrum of a coalgebra $C_{T(n)}(Y)$ is $\Sigma^\infty_{T(n)} Y$ and the third line uses the equivalence $\Sigma^\infty_{T(n)}\Theta \simeq \mathrm{id}_{\mathrm{Sp}_{T(n)}}$. Thus $\psi \simeq \mathrm{id}_{\mathrm{Sp}_{T(n)}}$ and we conclude that $\theta \simeq \Omega^\infty_{\mathrm{CAlg}}$. By our definition of $\mathrm{triv}$, it follows that we have found an equivalence $C_{T(n)} \circ \Theta \simeq \mathrm{triv}$, at least on compact objects. But since both functors preserve colimits this suffices.
\end{proof}

\begin{remark}
\label{rmk:BRcomparison}
Lemma \ref{lem:Thetatriv} yields a variant of the \emph{comparison map} of Behrens and Rezk (as in Section 6 of \cite{behrensrezk}). Indeed, write $D_{T(n)}$ for the right adjoint of $C_{T(n)}$, of which the existence is again guaranteed by the adjoint functor theorem. Taking right adjoints of the functors of the lemma gives an equivalence
\begin{equation*}
\Phi \circ D_{T(n)} \simeq \mathrm{prim}^{\mathrm{ind}}
\end{equation*}
and the unit of the adjoint pair $(C_{T(n)}, D_{T(n)})$  then induces a natural map
\begin{equation*}
\Phi \rightarrow \Phi \circ D_{T(n)} \circ C_{T(n)} \simeq \mathrm{prim}^{\mathrm{ind}} \circ C_{T(n)},
\end{equation*}
which is the version of the comparison map we will employ here.
\end{remark}

Recall the functors
\begin{equation*}
\mathcal{S}_{v_n} \xrightarrow{\Sigma^\infty_k} \mathcal{P}_k\mathcal{S}_{v_n} \simeq \mathrm{coAlg}^{\mathrm{ind}}(\tau_k\mathrm{Sp}_{T(n)}^\otimes),
\end{equation*}
where the second is the equivalence of Proposition \ref{prop:nstages}. This composite can be described explicitly as
\begin{equation*}
\mathcal{S}_{v_n} \xrightarrow{C_{T(n)}} \mathrm{coAlg}^{\mathrm{ind}}(\mathrm{Sp}_{T(n)}^\otimes) \xrightarrow{\tau_k} \mathrm{coAlg}^{\mathrm{ind}}(\tau_k\mathrm{Sp}_{T(n)}^\otimes),
\end{equation*}
where the second functor is the truncation discussed in Remark \ref{rmk:truncation}. Then Lemma \ref{lem:Thetatriv} immediately implies:

\begin{corollary}
\label{cor:Thetatrivk}
Under the identification of Proposition \ref{prop:nstages} the composite
\begin{equation*}
\mathrm{Sp}_{T(n)} \xrightarrow{\Theta} \mathcal{S}_{v_n} \xrightarrow{\Sigma^\infty_k} \mathcal{P}_k\mathcal{S}_{v_n}
\end{equation*}
is naturally equivalent to
\begin{equation*}
\mathrm{triv}_k\colon \mathrm{Sp}_{T(n)} \rightarrow \mathrm{coAlg}^{\mathrm{ind}}(\tau_k\mathrm{Sp}_{T(n)}^\otimes).
\end{equation*}
The proof of Lemma \ref{lem:Thetatriv} induces a preferred such equivalence.
\end{corollary}

\begin{proof}[Proof of Theorem \ref{thm:BR}]
As observed before, the fact that $\Phi$ preserves limits guarantees $P_k\Phi \simeq \Phi P_k\mathrm{id}_{\mathcal{S}_{v_n}}$. In turn, the natural transformation $\mathrm{id}_{\mathcal{S}_{v_n}} \rightarrow P_k\mathrm{id}_{\mathcal{S}_{v_n}}$ can be identified with the unit of the adjoint pair
\[
\begin{tikzcd}
\mathcal{S}_{v_n} \arrow[r,"\Sigma^\infty_k", shift left] & \mathcal{P}_k\mathcal{S}_{v_n} \arrow[l, "\Omega^\infty_k", shift left].
\end{tikzcd}
\]
In particular $P_k\Phi \simeq \Phi\Omega^\infty_k\Sigma^\infty_k$. The composite $\Phi\Omega^\infty_k$ is right adjoint to $\Sigma^\infty_k\Theta$ and therefore equivalent to $\mathrm{prim}_k$ by Corollary \ref{cor:Thetatrivk}. Using $\Sigma^\infty_k \simeq \tau_k C_{T(n)}$ we conclude that
\begin{equation*}
P_k\Phi \simeq \mathrm{prim}_k \circ \tau_k C_{T(n)}
\end{equation*}
as claimed.
\end{proof}

Recall that a space $X \in \mathcal{S}_{v_n}$ is \emph{$\Phi$-good} if the natural map
\begin{equation*}
\Phi(X) \rightarrow \varprojlim_k P_k\Phi(X)
\end{equation*}
is an equivalence. Thus to prove Corollary \ref{cor:BR} it suffices to combine Theorem \ref{thm:BR} with Lemma \ref{lem:primconvergence} below. Note that the statement concerns the functor
\begin{equation*}
\mathrm{prim}\colon \mathrm{coAlg}(\mathrm{Sp}_{T(n)})^\otimes \rightarrow \mathrm{Sp}_{T(n)}
\end{equation*}
rather than its variant for ind-coalgebras.

\begin{lemma}
\label{lem:primconvergence}
The evident natural transformation
\begin{equation*}
\mathrm{prim} \simeq \varprojlim_k (\mathrm{prim}_k \circ \tau_k)
\end{equation*}
is an equivalence.
\end{lemma}
\begin{proof}
For $Y \in \mathrm{Sp}_{T(n)}$ and $X \in \mathrm{coAlg}(\mathrm{Sp}_{T(n)})$, observe the following natural equivalences:
\begin{eqnarray*}
\mathrm{Map}_{\mathrm{Sp}_{T(n)}}(Y, \varprojlim_k (\mathrm{prim}_k(\tau_k X)) & \simeq & \varprojlim_k \mathrm{Map}_{\mathrm{coAlg}(\tau_k\mathrm{Sp}_{T(n)}^\otimes)}(\mathrm{triv}_k Y, \tau_k X) \\
& \simeq & \varprojlim_k \mathrm{Map}_{\mathrm{coAlg}(\tau_k\mathrm{Sp}_{T(n)}^\otimes)}(\tau_k\mathrm{triv} Y, \tau_k X).
\end{eqnarray*}
Lemma C.30 of \cite{heutsgoodwillie} implies that the evident functor
\begin{equation*}
\mathrm{coAlg}(\mathrm{Sp}_{T(n)}^\otimes) \rightarrow \varprojlim_k \mathrm{coAlg}(\tau_k\mathrm{Sp}_{T(n)}^\otimes)
\end{equation*}
is an equivalence, which gives further equivalences
\begin{eqnarray*}
\varprojlim_k \mathrm{Map}_{\mathrm{coAlg}(\tau_k\mathrm{Sp}_{T(n)}^\otimes)}(\tau_k\mathrm{triv} Y, \tau_k X) & \simeq & \mathrm{Map}_{\mathrm{coAlg}(\mathrm{Sp}_{T(n)}^\otimes)}(\mathrm{triv} Y, X) \\
& \simeq & \mathrm{Map}_{\mathrm{Sp}_{T(n)}}(Y, \mathrm{prim} X),
\end{eqnarray*}
finishing the proof.
\end{proof}

% subsection: \Phi-good spaces and Q-completion
% lemma on primitives of \Sigma^\infty\Omega^\infty X
% corollary: infinite loop spaces \Phi-good
% deduce X \Phi-good iff X Q-complete, using cobar resolution of coalgebras
% use that coalgebras are comonadic over Sp_{T(n)}

\appendix

\section{Dual Goodwillie calculus}
\label{app:Goodwillie}

The calculus of functors was developed by Goodwillie in \cite{goodwillie3}. It provides for every suitable functor $F\colon \mathcal{S}_* \rightarrow \mathcal{S}_*$ (or $\mathcal{S}_* \rightarrow \mathrm{Sp}$, $\mathrm{Sp} \rightarrow \mathrm{Sp}$, etc.) a natural transformation $F \rightarrow P_k F$, for every $k\geq 0$. Here $P_k F$ is a $k$-excisive functor and the natural transformation is initial with respect to maps to $k$-excisive functors. Goodwillie's arguments are sufficiently conceptual to generalize to a wide variety of homotopy theories. We make use of the version of the theory described in Chapter 6 of \cite{higheralgebra}, which works for $\infty$-categories satisfying only a few mild assumptions. 

Write $\mathrm{Fun}^{\omega}(\mathrm{Sp}_{T(n)}, \mathrm{Sp}_{T(n)})$ for the $\infty$-category of functors from $\mathrm{Sp}_{T(n)}$ to itself that preserve filtered colimits and $\mathrm{Fun}^{\leq k}(\mathrm{Sp}_{T(n)}, \mathrm{Sp}_{T(n)})$ for the full subcategory on $k$-excisive such functors. Writing $\mathrm{Sp}_{T(n)}^\omega$ for the full subcategory on compact objects of $\mathrm{Sp}_{T(n)}$, restriction along the inclusion gives an equivalence of $\infty$-categories
\begin{equation*}
\mathrm{Fun}^{\omega}(\mathrm{Sp}_{T(n)}, \mathrm{Sp}_{T(n)}) \rightarrow \mathrm{Fun}(\mathrm{Sp}_{T(n)}^\omega, \mathrm{Sp}_{T(n)})
\end{equation*}
and similarly for subcategories of $k$-excisive functors. Theorem 6.1.1.10 of \cite{higheralgebra} shows that the inclusion $\mathrm{Fun}^{\leq k}(\mathrm{Sp}_{T(n)}^\omega, \mathrm{Sp}_{T(n)}) \rightarrow \mathrm{Fun}(\mathrm{Sp}_{T(n)}^\omega, \mathrm{Sp}_{T(n)})$ admits a left adjoint
\begin{equation*}
\mathrm{Fun}(\mathrm{Sp}_{T(n)}^\omega, \mathrm{Sp}_{T(n)}) \rightarrow \mathrm{Fun}^{\leq k}(\mathrm{Sp}_{T(n)}^\omega, \mathrm{Sp}_{T(n)})\colon F \mapsto P_k F.
\end{equation*}
In the body of this paper we also need a notion of \emph{dual} calculus. A dual version of calculus for functors to spectra was developed by McCarthy \cite{mccarthy}. We will not need much theory though, only the following statement:

\begin{theorem}
\label{thm:dualcalculus}
The inclusion $\mathrm{Fun}^{\leq k}(\mathrm{Sp}_{T(n)}^\omega, \mathrm{Sp}_{T(n)}) \rightarrow \mathrm{Fun}(\mathrm{Sp}_{T(n)}^\omega, \mathrm{Sp}_{T(n)})$ admits a right adjoint $F \mapsto P^k F$. In other words, the subcategory of $k$-excisive functors is also a colocalization of $\mathrm{Fun}(\mathrm{Sp}_{T(n)}^\omega, \mathrm{Sp}_{T(n)})$.
\end{theorem}
\begin{proof}
We apply Theorem 6.1.1.10 of \cite{higheralgebra} to the opposite $\infty$-category 
\begin{equation*}
\mathrm{Fun}(\mathrm{Sp}_{T(n)}^\omega, \mathrm{Sp}_{T(n)})^{\mathrm{op}} = \mathrm{Fun}((\mathrm{Sp}_{T(n)}^\omega)^{\mathrm{op}}, \mathrm{Sp}_{T(n)}^{\mathrm{op}}).
\end{equation*}
The reason that this works is that $\mathrm{Sp}_{T(n)}^{\mathrm{op}}$, being stable and cocomplete, is still a \emph{differentiable} $\infty$-category in the sense of Definition 6.1.1.6 of \cite{higheralgebra}.
\end{proof}

Of course this theorem does not use any of the specifics of the $\infty$-category $\mathrm{Sp}_{T(n)}$ and works for any complete stable $\infty$-category. We have stated it only for emphasis and ease of reference. In applications we will always identify $\mathrm{Fun}(\mathrm{Sp}_{T(n)}^\omega, \mathrm{Sp}_{T(n)})$ with $\mathrm{Fun}^\omega(\mathrm{Sp}_{T(n)}, \mathrm{Sp}_{T(n)})$.

Theorem \ref{thm:dualcalculus} uses in an essential way that $\mathrm{Sp}_{T(n)}$ is stable to conclude that its opposite category is also differentiable. One could try to develop a dual version of the theory of Goodwillie calculus for functors valued in an $\infty$-category that is not stable. In this case one can ask if any functor $F$ admits a universal approximation $P^k F \rightarrow F$ by a $k$-\emph{co}excisive functor. (In the stable setting there is no distinction between $k$-excisive and $k$-coexcisive functors.) The fundamental role of the stabilization of an $\infty$-category $\mathcal{C}$ would then be taken over by its \emph{costabilization} $\mathrm{coSp}(\mathcal{C})$. For the $\infty$-category $\mathcal{S}_*$ the costabilization $\mathrm{coSp}(\mathcal{S}_*)$ is trivial, as a consequence of the fact that any infinite suspension space is contractible, and any attempt at a formally dual theory of Goodwillie calculus for functors to $\mathcal{S}_*$ therefore seems futile. However, the same cannot be said of $\mathcal{S}_{v_n}$, which contains many infinite suspension objects, namely spaces of the form $\Theta(E)$. Similarly, the $\infty$-category of algebras over an operad in the category of spectra has many infinite suspension objects, namely the free algebras. Moreover, as already explained in Remark \ref{rmk:Kuhnfiltration}, there seems to be a natural candidate for the \emph{dual Goodwillie filtration} of the identity functor of $\mathcal{S}_{v_n}$, or of the $\infty$-category of algebras over an operad in general, which should be formally dual to the Goodwillie tower of the identity. We intend to return to these questions in future work.

\section{A nilpotence lemma for differentiation}
\label{app:akhil}

\subsection{Statement of the lemma}

The goal of this section is to explain that the dual $k$-excisive approximation $P^k$ (for functors from $\mathrm{Sp}_{T(n)}$ to itself) commutes with a very specific kind of colimit, as in the following:

\begin{lemma}[Mathew]
\label{lem:mathew}
If 
\begin{equation*}
F = \bigoplus_{j=1}^\infty F_j
\end{equation*}
with $F_j\colon \mathrm{Sp}_{T(n)} \rightarrow \mathrm{Sp}_{T(n)}$ a $j$-homogeneous functor, then the natural map
\begin{equation*}
\bigoplus_{j=1}^k F_j \rightarrow P^k F
\end{equation*}
is an equivalence.
\end{lemma}

We learned of this result from Akhil Mathew. Since a proof has not yet appeared in the literature, we offer it here.

\begin{remark}
Our strategy of proof will apply equally well to show that in the $T(n)$-local setting one has
\begin{equation*}
P_k\Bigl(\prod_{j=1}^\infty F_j\Bigr) \simeq \prod_{j=1}^k F_j,
\end{equation*}
even though the $k$-excisive approximation $P_k$ need not commute with infinite products in general.
\end{remark}

The lemma above is equivalent to the following statement:

\begin{lemma}
\label{lem:statement2}
With $F_j$ as above, the functor
\begin{equation*}
D^\ell\Bigl(\bigoplus_{j=k+1}^\infty F_j \Bigr)
\end{equation*}
is null for every $\ell \leq k$.
\end{lemma}

The dual derivative $D^\ell$ of a functor $F$ may be constructed from the cocross effects of $F$. In detail, the cocross effect $\mathrm{cr}^\ell$ is the functor of $\ell$ variables defined by the total cofiber 
\begin{equation*}
(\mathrm{cr}^\ell F)(X_1, \ldots, X_\ell) = \mathrm{tcof}(F\mathcal{X}),
\end{equation*}
where $\mathcal{X}$ is the cube
\begin{equation*}
\mathcal{X}\colon \mathcal{P}(\ell) \rightarrow \mathrm{Sp}_{T(n)}\colon (U \subseteq \{1, \ldots, \ell\}) \mapsto \bigvee_{i \in U} X_i. 
\end{equation*}
Here $\mathcal{P}(\ell)$ denotes the power set of $\{1, \ldots, \ell\}$, regarded as a poset under inclusion. Colinearizing the cocross effect in each variable defines a functor
\begin{equation*}
L(X_1, \ldots, X_\ell) := \varprojlim_{m} \Sigma^{m\ell} (\mathrm{cr}^\ell F)(\Sigma^{-m}X_1, \ldots, \Sigma^{-m}X_\ell)
\end{equation*}
and the dual derivative is given by
\begin{equation*}
D^\ell F(X) \simeq L(X, \ldots, X)^{h\Sigma_\ell}.
\end{equation*}

Write $\Delta_\ell$ for the diagonal functor
\begin{equation*}
\mathrm{Sp}_{T(n)} \rightarrow \mathrm{Sp}_{T(n)}^\ell: X \mapsto (X, \ldots, X).
\end{equation*}
Lemma \ref{lem:statement2} above will be a consequence of the following `uniform nilpotence' result:

\begin{lemma}
\label{lem:statement3}
Let $V$ be a finite type $n$ spectrum. Then there exists a constant $C$ with the following property: for any $j > \ell$ and $j$-homogeneous functor $H$, the natural transformation
\begin{equation*}
V \otimes (\Sigma^{C\ell}\mathrm{cr}^\ell(H) \circ \Sigma^{-C}\Delta) \rightarrow V \otimes (\mathrm{cr}^\ell(H) \circ \Delta)
\end{equation*}
is null.
\end{lemma} 

\begin{remark}
The reason for calling this \emph{uniform nilpotence} is that the constant $C$ only depends on $V$, but not on $j$, $\ell$, and $H$.
\end{remark}

\begin{proof}[Proof of Lemma \ref{lem:statement2}]
Write $G$ for the functor featuring in the lemma, so that we should show $D^\ell G \simeq 0$. A consequence of Lemma \ref{lem:statement3} is that the pro-system
\begin{equation*}
\{V \otimes (\Sigma^{m\ell} \mathrm{cr}^\ell (G)\circ \Sigma^{-m}\Delta)\}_{m \geq 0}
\end{equation*}
is pro-trivial; indeed, it is the direct sum of pro-systems which are nilpotent of exponent $C$. Therefore
\begin{equation*}
V \otimes \varprojlim_m \bigl(\Sigma^{m\ell} \mathrm{cr}^\ell (G)\circ \Sigma^{-m}\Delta \bigr) \simeq 0
\end{equation*}
and the lemma follows.
\end{proof}

The remainder of this section is concerned with the proof of Lemma \ref{lem:statement3}.

\subsection{Some preliminaries on nilpotence}

We will denote by $V$ a finite type $n$ spectrum, which is fixed throughout. We work in the $\infty$-category $\mathrm{Sp}_{T(n)}$ and consistently omit $L_{T(n)}$ from the notation. For example, $\mathbb{S}$ will stand for the $T(n)$-local sphere spectrum. In this section we review some material from \cite{mathew} and observe a few elementary consequences.

\begin{definition}
Let $G$ be a finite group.
\begin{itemize}
\item[(i)] An object $X \in \mathrm{Fun}(BG, \mathrm{Sp}_{T(n)})$ is \emph{nilpotent} if it belongs to the thick subcategory generated by free $G$-objects, i.e., objects of the form $G_+ \otimes Z$ for $Z \in \mathrm{Sp}_{T(n)}$.
\item[(ii)] If $X \in \mathrm{Fun}(BG, \mathrm{Sp}_{T(n)})$ is nilpotent, we define the \emph{exponent} of $X$, denoted $\mathrm{exp}(X)$, as in 2.2 of \cite{mathew}. To be precise, write $\mathcal{F}$ for the collection $\{G_+ \otimes Z \, | \, Z \in \mathrm{Sp}_{T(n)}\}$ and define full subcategories $\mathrm{Thick}_m(\mathcal{F})$ of $\mathrm{Fun}(BG,\mathrm{Sp}_{T(n)})$ inductively as follows. The $\infty$-category $\mathrm{Thick}_1(\mathcal{F})$ is the full subcategory on retracts of objects in $\mathcal{F}$. An object $X \in \mathrm{Fun}(BG,\mathrm{Sp}_{T(n)})$ is in $\mathrm{Thick}_m(\mathcal{F})$ if it is a retract of an object $X_0$ which fits in a cofiber sequence
\begin{equation*}
Y \rightarrow X_0 \rightarrow Z,
\end{equation*}
where $Y \in \mathrm{Thick}_1(\mathcal{F})$ and $Z \in \mathrm{Thick}_{m-1}(\mathcal{F})$. Finally, the exponent of a nilpotent object $X$ is the least $m$ for which $X \in \mathrm{Thick}_m(\mathcal{F})$.
\item[(iii)] For a subgroup $H \leq G$, one similarly defines \emph{$H$-nilpotence} and the \emph{$H$-exponent} $\mathrm{exp}_H(X)$ by using the collection of induced objects 
\begin{equation*}
\{G/H_+ \otimes Z \, | \, Z \in \mathrm{Fun}(BH,\mathrm{Sp}_{T(n)})\}
\end{equation*}
in place of $\mathcal{F}$.
\end{itemize}
\end{definition}

The following is the crucial feature of exponents, indicating their use in the theory of descent. The proof is straightforward.

\begin{lemma}
\label{lem:charnilpotent}
Suppose $X \in \mathrm{Fun}(BG, \mathrm{Sp}_{T(n)})$ has $H$-exponent $m$ and suppose
\begin{equation*}
X_0 \rightarrow X_1 \rightarrow \cdots \rightarrow X_m = X
\end{equation*}
is a sequence of maps such that each $X_i \rightarrow X_{i+1}$ becomes null when restricted to $\mathrm{Fun}(BH,\mathrm{Sp}_{T(n)})$. Then the composite $X_0 \rightarrow X$ is null in $\mathrm{Fun}(BG,\mathrm{Sp}_{T(n)})$.
\end{lemma}

If $X \in \mathrm{Fun}(BG,\mathrm{Sp}_{T(n)})$ is nilpotent then $X^{tG} = 0$. Indeed, this is clearly the case for free $G$-objects, so that it is equally true for any object in the thick subcategory generated by such. This implication can be reversed for compact objects of $\mathrm{Sp}_{T(n)}$ (cf. Section 4.1 of \cite{mathew}):

\begin{lemma}
\label{lem:expTatevanishing}
For any $X \in \mathrm{Fun}(BG,\mathrm{Sp}_{T(n)})$, the object $V \otimes X$ is nilpotent. Moreover, the exponent $\mathrm{exp}(V \otimes X)$ depends only on $V$.
\end{lemma}
\begin{proof}
Clearly it suffices to prove that $V$ (with trivial $G$-action) is nilpotent; indeed, the observation that $Y \otimes X$ is a free $G$-object whenever $Y$ is free immediately shows that $V \otimes X$ will be nilpotent with exponent at most $\mathrm{exp}(V)$, for any $G$-object $X$. The lemma is a consequence of the $T(n)$-local vanishing of Tate spectra (compare the proof of Theorem 4.9 of \cite{mathew}).  Indeed, write $EG_\bullet = G^{\bullet + 1}$ for the usual simplicial model of the universal $G$-space and $\mathrm{sk}_m EG$ for its $m$-skeleton. Since Tate spectra of objects in $\mathrm{Fun}(BG,\mathrm{Sp}_{T(n)})$ vanish in $\mathrm{Sp}_{T(n)}$, the homotopy fixed point functor preserves filtered colimits and
\begin{equation*}
\mathbb{S}^{hG} \simeq \varinjlim_m \bigl((\mathrm{sk}_m EG)_+ \otimes \mathbb{S}\bigr)^{hG}.
\end{equation*}
The homotopy fixed points $\mathbb{S}^{hG}$ form a ring spectrum; we write 
\begin{equation*}
\mathbb{S} \xrightarrow{\eta} \mathbb{S}^{hG}
\end{equation*}
for its unit. Since $V$ is compact in $\mathrm{Sp}_{T(n)}$, the map $V \otimes \eta$ factors as
\begin{equation*}
V \otimes \mathbb{S} \rightarrow V \otimes \bigl((\mathrm{sk}_m EG)_+ \otimes \mathbb{S}\bigr)^{hG} \rightarrow V \otimes \mathbb{S}^{hG}
\end{equation*}
for some $m$. By adjunction it follows that there is a retract diagram 
\begin{equation*}
V \rightarrow V \otimes (\mathrm{sk}_m EG)_+ \rightarrow V
\end{equation*}
in $\mathrm{Fun}(BG,\mathrm{Sp}_{T(n)})$, where $V$ has the trivial $G$-action. But $V \otimes (\mathrm{sk}_m EG)_+$ is in the thick subcategory generated by $V \otimes G_+$ and therefore nilpotent.
\end{proof}

\begin{lemma}
\label{lem:expV}
There exists a constant $C$ (depending only on $V$) with the following property: for any finite group $G$, subgroup $H \leq G$ with $v_p(G) - v_p(H) \leq 1$, and any object $X \in \mathrm{Fun}(BG,\mathrm{Sp}_{T(n)})$, the $H$-exponent of $V \otimes X \in \mathrm{Fun}(BG,\mathrm{Sp}_{T(n)})$ is at most $C$. (Here $v_p(G)$ is the $p$-adic valuation of the order of $G$.)
\end{lemma}
\begin{proof}
One easily reduces to $p$-groups, so that $H$ has index $p$ in $G$ and is therefore normal with $G/H \simeq C_p$. One can now take $C$ to be the exponent of $V$ in $\mathrm{Fun}(BC_p,\mathrm{Sp}_{T(n)})$, which exists by Lemma \ref{lem:expTatevanishing}.
\end{proof}

We write $\rho_j$ for the standard $j$-dimensional (real) representation of $\Sigma_j$ and $\bar{\rho}_j$ for the reduced standard representation obtained from $\rho_j$ by quotienting out the diagonal, on which $\Sigma_j$ acts trivially. As usual, we write $\mathbb{S}^{\rho_j}$ and $\mathbb{S}^{\bar{\rho}_j}$ for the associated representation spheres, which are the (suspensoin spectra of the) one-point compactifications. The natural $\Sigma_j$-equivariant map
\begin{equation*}
X^{\otimes j} \rightarrow \Omega((\Sigma X)^{\otimes j})
\end{equation*}
when evaluated at $X = \mathbb{S}$ gives a map
\begin{equation*}
\mathbb{S} \xrightarrow{e} \mathbb{S}^{\bar{\rho}_j}
\end{equation*}
which is the \emph{Euler class} of $\bar{\rho}_j$.

\begin{lemma}
\label{lem:Eulerclass}
Take $C$ as in Lemma \ref{lem:expV} and $j \geq 2$. Then the map
\begin{equation*}
V \otimes e^C\colon V \rightarrow V \otimes \mathbb{S}^{C\bar{\rho}_j}
\end{equation*}
is nullhomotopic in $\mathrm{Fun}(B\Sigma_j, \mathrm{Sp}_{T(n)})$.
\end{lemma}
\begin{proof}
Pick a Young subgroup $H < \Sigma_j$ with $v_p(\Sigma_j) - v_p(H) \leq 1$. The existence of such $H$ is easy to establish; for $j$ not a power of $p$ one can even arrange $v_p(\Sigma_j) = v_p(H)$, whereas for $j = p^m$ one could use
$H = \Sigma_{p^{m-1}} \times \cdots \times \Sigma_{p^{m-1}}$. When restricted to $H$ the representation $\bar{\rho}_j$ admits a nonzero fixed point, which implies that the restriction of $e$ to $\mathrm{Fun}(BH,\mathrm{Sp}_{T(n)})$ is null. The lemma now follows by combining Lemmas \ref{lem:charnilpotent} and \ref{lem:expV}.
\end{proof}

\subsection{Proof}

\begin{proof}[Proof of Lemma \ref{lem:statement3}]
To orient the reader, let us first treat the case $\ell = 1$. Write $L$ for the multilinear functor of $j$ variables corresponding to $H$ via Goodwillie's classification of homogeneous functors, so that
\begin{equation*}
H \simeq (L \circ \Delta_j)_{h\Sigma_j}.
\end{equation*}
Take $C$ as in Lemma \ref{lem:Eulerclass} applied to the Spanier--Whitehead dual $\mathbf{D}V$ of $V$. Then it suffices to show that the $\Sigma_j$-equivariant map
\begin{equation*}
V \otimes \Sigma^{C}L \circ (\Sigma^{-C}\Delta_j) \rightarrow V \otimes L \circ \Delta_j
\end{equation*}
is null. Indeed, the claim of the lemma then follows upon taking homotopy orbits for $\Sigma_j$. After evaluating at $X \in \mathrm{Sp}_{T(n)}$ this map can be written
\begin{equation*}
V \otimes \mathbb{S}^{C(1 - \rho_j)} \otimes L(X, \ldots, X) \rightarrow V \otimes \mathbb{S} \otimes L(X, \ldots, X),
\end{equation*}
with $\mathbb{S}^{C(1 - \rho_j)}$ the representation sphere of the (virtual) representation $C(1-\rho_j) = -C\bar{\rho}_j$. Thus it suffices to show that the map
\begin{equation*}
V \otimes \mathbb{S}^{-C\bar{\rho}_j} \rightarrow V \otimes \mathbb{S}
\end{equation*}
is null in $\mathrm{Fun}(B\Sigma_j, \mathrm{Sp}_{T(n)})$. By Spanier--Whitehead duality this is equivalent to showing that
\begin{equation*}
\mathbf{D}V \otimes \mathbb{S} \rightarrow \mathbf{D}V \otimes \mathbb{S}^{C \bar{\rho}_j}
\end{equation*}
is null (this map is the $C$th power of the Euler class of $\bar{\rho}_j$, as before), which is precisely the content of Lemma \ref{lem:Eulerclass}.

We now treat the case of general $\ell$. We will show that the $\Sigma_j$-equivariant natural transformation
\begin{equation*}
V \otimes \Sigma^{C\ell}\mathrm{cr}^\ell(L \circ \Sigma^{-C} \Delta_j) \rightarrow V \otimes \mathrm{cr}^\ell(L \circ \Delta)
\end{equation*}
is null. The multilinear functor $L$ can be written as
\begin{equation*}
L(X_1, \ldots, X_j) \simeq \partial L \otimes X_1 \otimes \cdots \otimes X_j
\end{equation*}
for some spectrum $\partial L$ with $\Sigma_j$-action. Using this, one easily verifies that
\begin{equation*}
\Sigma^{C\ell} \mathrm{cr}^\ell(L \circ \Sigma^{-C}\Delta_j) \simeq \Bigl(\bigoplus_{f\colon \mathbf{j} \rightarrow \mathbf{l}} \mathbb{S}^{-C\bar{\rho}_{f^{-1}\{1\}}} \otimes \cdots \otimes \mathbb{S}^{-C\bar{\rho}_{f^{-1}\{\ell\}}} \Bigr) \otimes (L \circ \Delta_j),
\end{equation*}
where the sum is over all \emph{surjections} $f\colon \{1, \ldots, j\} \rightarrow \{1, \ldots, \ell\}$ and the notation 
\begin{equation*}
\mathbb{S}^{-C\bar{\rho}_{f^{-1}\{1\}}} \otimes \cdots \otimes \mathbb{S}^{-C\bar{\rho}_{f^{-1}\{\ell\}}}
\end{equation*}
indicates the action of the stabilizer $\Sigma_{f^{-1}\{1\}} \times \cdots \times \Sigma_{f^{-1}\{\ell\}}$ of a surjection $f$. In other words, the sum on the right-hand side is a sum of representation spheres induced from Young subgroups
\begin{equation*}
\Sigma_{j_1} \times \cdots \times \Sigma_{j_\ell} \leq \Sigma_j
\end{equation*}
with $j_1 + \cdots + j_\ell = j$. (Note that the assumption $j > \ell$ guarantees that there exists $i$ with $j_i \geq 2$.) Thus it suffices to show that for each such subgroup, the corresponding map
\begin{equation*}
V \otimes \mathbb{S}^{-C\bar{\rho}_{j_1}} \otimes \cdots \otimes \mathbb{S}^{-C\bar{\rho}_{j_\ell}} \rightarrow V \otimes \mathbb{S}
\end{equation*}
is null in $\mathrm{Fun}(B(\Sigma_{j_1} \times \cdots \times \Sigma_{j_\ell}), \mathrm{Sp}_{T(n)})$. Dualizing as before, we may as well consider the map
\begin{equation*}
\mathbf{D}V \otimes \mathbb{S} \rightarrow \mathbf{D}V \otimes \mathbb{S}^{C\bar{\rho}_{j_1}} \otimes \cdots \otimes \mathbb{S}^{C\bar{\rho}_{j_\ell}}.
\end{equation*}
This map is the smash product of $\ell$ maps of which at least one is null by the conclusion of Lemma \ref{lem:Eulerclass}.
\end{proof}

\section{Coalgebras}
\label{app:coalgebras}

We owe the reader a comparison between the notion of commutative coalgebras in a symmetric monoidal $\infty$-category $\mathcal{C}^\otimes$ used in this paper and the one used in \cite{heutsgoodwillie}. To simplify notation it will be convenient to replace $\mathcal{C}$ by $\mathcal{C}^{\mathrm{op}}$ and compare the following two notions of commutative algebras:

\begin{itemize}
\item[(1)] Lurie defines a commutative algebra in $\mathcal{C}^\otimes$ to be a section of the structure map $p\colon \mathcal{C}^\otimes \rightarrow N\mathcal{F}\mathrm{in}_*$ which preserves inert morphisms and writes $\mathrm{CAlg}(\mathcal{C})$ for the $\infty$-category of such objects (cf. Definition 2.1.3.1 of \cite{higheralgebra}).
\item[(2)] Definition 5.14 of \cite{heutsgoodwillie} (after passing to opposites and including units) specializes to the following: a commutative algebra object of $\mathcal{C}^\otimes$ is a fibration of $\infty$-operads $f\colon \mathcal{X}^\otimes \rightarrow \mathcal{C}^\otimes$ such that
\begin{itemize}
\item[(a)] the composite $pf\colon \mathcal{X}^\otimes \rightarrow N\mathcal{F}\mathrm{in}_*$ is a coCartesian fibration (making $\mathcal{X}^\otimes$ a symmetric monoidal $\infty$-category),
\item[(b)] the map $f$ is a symmetric monoidal functor,
\item[(c)] the map of underlying $\infty$-categories $\mathcal{X} \rightarrow \mathcal{C}$ is of the form $\mathcal{C}_{/X} \rightarrow \mathcal{C}$ for some $X \in \mathcal{C}$.
\end{itemize}
More informally, giving $X$ the structure of a commutative algebra is equivalent to upgrading the slice category $\mathcal{C}_{/X}$ to a symmetric monoidal $\infty$-category (compatible with the forgetful functor to $\mathcal{C}$). We write $\mathrm{CAlg}'(\mathcal{C})$ for the $\infty$-category of commutative algebra objects of $\mathcal{C}$ according to this second definition.
\end{itemize}

To compare the two we define an $\infty$-category $\mathcal{A}$ as follows. An object of $\mathcal{A}$ is an object $\mathcal{X}^\otimes \rightarrow \mathcal{C}^\otimes$ of $\mathrm{CAlg}'(\mathcal{C})$ together with a map $s\colon N\mathcal{F}\mathrm{in}_* \rightarrow \mathcal{X}^\otimes$ which preserves inert morphisms, and such that the composition
\begin{equation*}
N\mathcal{F}\mathrm{in}_* \xrightarrow{s} \mathcal{X}^\otimes \xrightarrow{f} \mathcal{C}^\otimes \xrightarrow{p} N\mathcal{F}\mathrm{in}_*
\end{equation*}
is the identity, and so that $s(\langle n \rangle)$ is a final object of $\mathcal{X}_{\langle n \rangle}^\otimes$ for every $n$. These last two properties can be summarized by saying that $s$ is a \emph{final section} of the structure map
\begin{equation*}
\mathcal{X}^\otimes \xrightarrow{pf} N\mathcal{F}\mathrm{in}_*.
\end{equation*}
In particular, $s$ defines a commutative algebra object of the symmetric monoidal $\infty$-category $\mathcal{X}^\otimes$ in the sense of (1). The $\infty$-category $\mathcal{A}$ can be defined as the full subcategory of $(\mathbf{Cat}_\infty)_{N\mathcal{F}\mathrm{in}_*/ /\mathcal{C}^\otimes}$ on the objects $(\mathcal{X}^\otimes \rightarrow \mathcal{C}^\otimes, s)$ as above. Note that there are evident forgetful maps
\begin{equation*}
\mathrm{CAlg}(\mathcal{C}) \xleftarrow{\varphi_1} \mathcal{A} \xrightarrow{\varphi_2} \mathrm{CAlg}'(\mathcal{C}).
\end{equation*}
Here $\varphi_1$ forms the composition $fs\colon N\mathcal{F}\mathrm{in}_* \rightarrow \mathcal{C}^\otimes$ and forgets $\mathcal{X}^\otimes$, whereas $\varphi_2$ forgets $s$ and simply retains $\mathcal{X}^\otimes \rightarrow \mathcal{C}^\otimes$.

\begin{proposition}
The maps $\varphi_1$ and $\varphi_2$ are equivalences.
\end{proposition}
\begin{proof}
The map $\varphi_2$ is a trivial fibration as a consequence of (the dual of) Proposition 2.4.4.9 of \cite{htt}, which states that final sections of a coCartesian fibration are homotopically unique in a strong sense. For $\varphi_1$, one can define a section $\sigma\colon \mathrm{CAlg}(\mathcal{C}) \rightarrow \mathcal{A}$ which assigns to a commutative algebra object $A$ the slice $\mathcal{C}^\otimes_{/A}$ (as in Notation 2.2.2.3 of \cite{higheralgebra}). This slice encodes the symmetric monoidal structure of $\mathcal{C}_{/A}$ induced by the commutative algebra structure of $A$. Then $\varphi_1 \sigma$ is the identity by construction and it remains to show that the functor $\sigma\varphi_1$ is naturally isomorphic to $\mathrm{id}_{\mathcal{A}}$. The space of natural transformations from $\sigma\varphi_1$ to $\mathrm{id}_{\mathcal{A}}$ can be expressed as a homotopy limit of spaces of the form
\begin{equation*}
\mathrm{Map}_{N\mathcal{F}\mathrm{in}_*/ /\mathcal{C}^\otimes}(\mathcal{C}_{/X}^\otimes, \mathcal{X}^\otimes)
\end{equation*}
where $\mathrm{Map}$ here is the maximal Kan complex in the $\infty$-category of all functors from $\mathcal{C}_{/X}^\otimes$ to $\mathcal{X}^\otimes$ compatible with the maps from $N\mathcal{F}\mathrm{in}_*$ and to $\mathcal{C}^\otimes$. In turn this space can be expressed as a homotopy limit of the spaces
\begin{equation*}
\mathrm{Map}_{\{\langle n \rangle\}/ /\mathcal{C}^\otimes_{\langle n \rangle}}(\mathcal{C}_{/X}^{\times n}, \mathcal{X}_{\langle n \rangle}^\otimes).
\end{equation*}
The $\infty$-category $\mathcal{X}_{\langle n \rangle}^\otimes$ is equivalent to $(\mathcal{X}_{\langle 1 \rangle}^{\otimes})^{\times n} = \mathcal{C}_{/X}^{\times n}$, and under this equivalence the vertex $\{\langle n \rangle\} \rightarrow \mathcal{X}_{\langle n \rangle}^\otimes$ is equivalent to the final object $(\mathrm{id}_X, \ldots, \mathrm{id}_X)$. Therefore the space above is equivalent to an $n$-fold Cartesian product of the space  
\begin{equation*}
\mathrm{Map}_{\{\langle 1 \rangle\}/ /\mathcal{C}}(\mathcal{C}_{/X}, \mathcal{C}_{/X}).
\end{equation*}
The map $\{\langle 1 \rangle\} \rightarrow \mathcal{C}_{/X}$ picks out the vertex $\mathrm{id}_X$ and is right anodyne. The projection $\mathcal{C}_{/X} \rightarrow \mathcal{C}$ is a right fibration, so that the space under consideration is equivalent to the one-point space 
\begin{equation*}
\mathrm{Map}_{\{\langle 1 \rangle\}/ /\mathcal{C}}(\{\mathrm{id}_X\}, \mathcal{C}_{/X}) \cong \Delta^0.
\end{equation*}
Since any limit of contractible spaces is contractible, there is an essentially unique natural transformation from $\sigma\varphi_1$ to $\mathrm{id}_{\mathcal{A}}$. The same argument applies in the other direction, from which we conclude that $\sigma\varphi_1$ and $\mathrm{id}_{\mathcal{A}}$ are naturally isomorphic.
\end{proof}

% include lemma (or reference): coAlg = lim_k coAlg_{\leq k}? right now just a reference to \cite{heutsgoodwillie}

\bibliographystyle{plain}
\bibliography{biblio}

\begin{thebibliography}{10}

\bibitem{adams}
J.F. Adams.
\newblock On the groups {J(X)} -- {IV}.
\newblock {\em Topology}, 5(1):21--71, 1966.

\bibitem{aroneching}
G.~Arone and M.~Ching.
\newblock {\em Operads and chain rules for the calculus of functors}.
\newblock Soci{\'e}t{\'e} math{\'e}matique de {F}rance, 2011.

\bibitem{aronedwyer}
G.~Arone and W.G. Dwyer.
\newblock Partition complexes, {T}its buildings and symmetric products.
\newblock {\em Proceedings of the London Mathematical Society}, 82(1):229--256,
  2001.

\bibitem{aronemahowald}
G.~Arone and M.E. Mahowald.
\newblock The {G}oodwillie tower of the identity functor and the unstable
  periodic homotopy of spheres.
\newblock {\em Inventiones {M}athematicae}, 135(3):743--788, 1999.

\bibitem{aronekankaanrinta2}
Greg Arone and Marja Kankaanrinta.
\newblock The homology of certain subgroups of the symmetric group with
  coefficients in lie (n).
\newblock {\em J. Pure Appl. Algebra}, 127(1):1--14, 1998.

\bibitem{basterramandell}
M.~Basterra and M.A. Mandell.
\newblock Homology and cohomology of ${E}_{\infty}$ ring spectra.
\newblock {\em Mathematische Zeitschrift}, 249(4):903--944, 2005.

\bibitem{behrensehp}
M.~Behrens.
\newblock {\em The Goodwillie tower and the EHP sequence}.
\newblock American Mathematical Soc., 2012.

\bibitem{behrensrezk}
M.~Behrens and C.~Rezk.
\newblock The {B}ousfield-{K}uhn functor and topological {A}ndr\'{e}-{Q}uillen
  cohomology.
\newblock preprint, 2012.

\bibitem{behrensrezk2}
M.~Behrens and C.~Rezk.
\newblock Spectral algebra models of unstable $v_n$-periodic homotopy theory.
\newblock {\em arXiv preprint arXiv:1703.02186}, 2017.

\bibitem{bousfieldlocalization}
A.K. Bousfield.
\newblock Localization and periodicity in unstable homotopy theory.
\newblock {\em Journal of the American Mathematical Society}, 7(4):831--873,
  1994.

\bibitem{bousfieldHspace}
A.K. Bousfield.
\newblock The {$K$}-theory localizations and $v_1$-periodic homotopy groups of
  {H}-spaces.
\newblock {\em Topology}, 38(6):1239--1264, 1999.

\bibitem{bousfieldtelescopic}
A.K. Bousfield.
\newblock On the telescopic homotopy theory of spaces.
\newblock {\em Transactions of the American Mathematical Society},
  353(6):2391--2426, 2001.

\bibitem{brantnerheuts}
D.~L.~B. Brantner and G.~S. K.~S. Heuts.
\newblock The $v_n$-periodic {G}oodwillie tower on wedges and cofibres.
\newblock arXiv:1612.02694.

\bibitem{ching}
M.~Ching.
\newblock Bar constructions for topological operads and the {G}oodwillie
  derivatives of the identity.
\newblock {\em Geom. Topol.}, 9:833--933, 2005.

\bibitem{chingbar}
M.~Ching.
\newblock Bar-cobar duality for operads in stable homotopy theory.
\newblock {\em Journal of {T}opology}, 2012.

\bibitem{haugsengchu}
H.~Chu and R.~Haugseng.
\newblock Enriched $\infty$-operads.
\newblock {\em arXiv preprint arXiv:1707.08049}, 2017.

\bibitem{cisinskimoerdijk1}
D.C. Cisinski and I.~Moerdijk.
\newblock Dendroidal sets as models for homotopy operads.
\newblock {\em Journal of {T}opology}, 4(2):257--299, 2011.

\bibitem{mathewclausen}
D.~Clausen and A.~Mathew.
\newblock A short proof of telescopic tate vanishing.
\newblock {\em Proceedings of the American Mathematical Society},
  145(12):5413--5417, 2017.

\bibitem{devinatzhopkinssmith}
E.S. Devinatz, M.J. Hopkins, and J.H. Smith.
\newblock Nilpotence and stable homotopy theory {I}.
\newblock {\em Annals of Mathematics}, 128(2):207--241, 1988.

\bibitem{ehmm}
R.~Eldred, G.~Heuts, A.~Mathew, and L.~Meier.
\newblock Monadicity of the bousfield-kuhn functor.
\newblock {\em arXiv preprint arXiv:1707.05986}, 2017.

\bibitem{farjoun}
E.~D. Farjoun.
\newblock {\em Cellular spaces, null spaces and homotopy localization}.
\newblock Springer, 2006.

\bibitem{francisgaitsgory}
J.~Francis and D.~Gaitsgory.
\newblock Chiral {K}oszul duality.
\newblock {\em Selecta Mathematica}, 18(1):27--87, 2012.

\bibitem{getzlerjones}
E.~Getzler and J.D.S. Jones.
\newblock Operads, homotopy algebra and iterated integrals for double loop
  spaces.
\newblock Available online at arXiv hep-th/9305013.

\bibitem{ginzburgkapranov}
V.~Ginzburg and M.~Kapranov.
\newblock Koszul duality for operads.
\newblock {\em Duke {M}athematical {J}ournal}, 76(1):203--272, 1994.

\bibitem{goodwillie3}
T.~G. Goodwillie.
\newblock Calculus {III}: {T}aylor {S}eries.
\newblock {\em Geometry and Topology}, 7:645--711, 2003.

\bibitem{heutsgoodwillie}
G.S.K.S. Heuts.
\newblock Goodwillie approximations to higher categories.
\newblock Available online at arXiv:1510.03304, 2015.

\bibitem{hopkinssmith}
M.J. Hopkins and J.H. Smith.
\newblock Nilpotence and stable homotopy theory {II}.
\newblock {\em Annals of Mathematics}, 148(1):1--49, 1998.

\bibitem{joyalpaper}
A.~Joyal.
\newblock Quasi-categories and {K}an complexes.
\newblock {\em Journal of Pure and Applied Algebra}, 175(1):207--222, 2002.

\bibitem{joyal}
A.~Joyal.
\newblock The theory of quasi-categories {I}.
\newblock {\em preprint}, 2008.

\bibitem{kuhninfiniteloop}
N.~J. Kuhn.
\newblock Morava {$K$}-theories and infinite loop spaces.
\newblock In {\em Algebraic topology}, pages 243--257. Springer, 1989.

\bibitem{kuhnmccord}
N.~J. Kuhn.
\newblock {The McCord model for the tensor product of a space and a commutative
  ring spectrum}.
\newblock In {\em {Categorical Decomposition Techniques in Algebraic
  Topology}}, pages 213--235. Springer, 2003.

\bibitem{kuhntate}
N.~J. Kuhn.
\newblock Tate cohomology and periodic localization of polynomial functors.
\newblock {\em Inventiones {M}athematicae}, 157(2):345--370, 2004.

\bibitem{kuhnAQ}
N.~J. Kuhn.
\newblock Localization of {A}ndr\'{e}-{Q}uillen-{G}oodwillie towers, and the
  periodic homology of infinite loopspaces.
\newblock {\em Advances in Mathematics}, 201(2):318--378, 2006.

\bibitem{kuhntelescopic}
N.~J. Kuhn.
\newblock A guide to telescopic functors.
\newblock {\em Homology, {H}omotopy and {A}pplications}, 10(3):291--319, 2008.

\bibitem{kuhnpereira}
N.J. Kuhn and L.~Pereira.
\newblock Operad bimodules and composition products on {A}ndr{\'e}--{Q}uillen
  filtrations of algebras.
\newblock {\em Algebraic \& Geometric Topology}, 17(2):1105--1130, 2017.

\bibitem{htt}
J.~Lurie.
\newblock {\em Higher topos theory}, volume 170.
\newblock Princeton University Press, 2009.

\bibitem{higheralgebra}
J.~Lurie.
\newblock Higher algebra.
\newblock Available online at math.harvard.edu/~lurie/, 2014.

\bibitem{lurieelliptic}
J.~Lurie.
\newblock Elliptic {C}ohomology {I}: {S}pectral {A}belian {V}arieties.
\newblock Available online at math.harvard.edu/~lurie/, 2016.

\bibitem{thursday}
J.~Lurie, M.J. Hopkins, et~al.
\newblock Unstable chromatic homotopy theory.
\newblock Available online at math.harvard.edu/~lurie/Thursday2017.html.

\bibitem{mahowaldImJ}
M.E. Mahowald.
\newblock The image of {J} in the {EHP} sequence.
\newblock {\em Annals of Mathematics}, pages 65--112, 1982.

\bibitem{mathew}
A.~Mathew.
\newblock Examples of descent up to nilpotence.
\newblock {\em arXiv preprint arXiv:1701.01528}, 2017.

\bibitem{mccarthy}
R.~McCarthy.
\newblock {\em Dual calculus for functors to spectra}, volume 271 of {\em
  Contemp. Math. Series}, pages 183--215.
\newblock A.M.S., 1999.

\bibitem{moerdijkweiss}
I.~Moerdijk and I.~Weiss.
\newblock On inner {K}an complexes in the category of dendroidal sets.
\newblock {\em Advances in Mathematics}, 221(2):343--389, 2009.

\bibitem{rationalhomotopy}
D.~G. Quillen.
\newblock Rational homotopy theory.
\newblock {\em Annals of Mathematics}, pages 205--295, 1969.

\bibitem{rezk}
C.~Rezk.
\newblock The units of a ring spectrum and a logarithmic cohomology operation.
\newblock {\em Journal of the American Mathematical Society}, 19(4):969--1014,
  2006.

\bibitem{riehlverity}
E.~Riehl and D.~Verity.
\newblock Homotopy coherent adjunctions and the formal theory of monads.
\newblock {\em Advances in Mathematics}, 286:802--888, 2016.

\bibitem{sullivan}
D.~Sullivan.
\newblock Infinitesimal computations in topology.
\newblock {\em Publications Math\'{e}matiques de l'IH\'{E}S}, 47(1):269--331,
  1977.

\bibitem{thompson}
R.~D. Thompson.
\newblock Unstable $v_1$-periodic homotopy groups of a moore space.
\newblock {\em Proceedings of the American Mathematical Society},
  107(3):833--845, 1989.

\bibitem{thompsonunstablesphere}
R.~D. Thompson.
\newblock The $v_1$-periodic homotopy groups of an unstable sphere.
\newblock {\em Transactions of the American Mathematical Society},
  319(2):535--559, 1990.

\bibitem{wang}
G.~Wang.
\newblock {\em Unstable chromatic homotopy theory}.
\newblock PhD thesis, Massachusetts Institute of Technology, 2015.

\bibitem{zhu}
Y.~Zhu.
\newblock Morava {$E$}-homology of {B}ousfield-{K}uhn functors on
  odd-dimensional spheres.
\newblock {\em Proceedings of the American Mathematical Society},
  146(1):449--458, 2018.

\end{thebibliography}

\end{document}